\documentclass[12pt]{article}

\usepackage[a4paper, top=2cm, bottom=2cm, right=2.5cm, left=1.5cm]{geometry} % To set up page margins and paper size

\usepackage[T1]{fontenc}               % For correct font encoding (essential for hyphenation)
\usepackage[italian, english]{babel}     % For multilingual support (last language is the default)
\usepackage{lmodern}                   % Modern, high-quality vector fonts
\usepackage{textcomp}                  % Additional text symbols
\usepackage{microtype}                 % Improves micro-typography and text justification
\usepackage{xcolor}                    % For colored text (required by \textcolor and )

\usepackage{mathtools, amssymb, amsthm}

\usepackage{mathbbol}                       % Provides \mathbb for blackboard bold Greek and lowercase letters

\usepackage{mathrsfs}                  % Provides the calligraphic font with \mathscr{L}

\usepackage{dsfont}                    % Provides blackboard-bold "1" with \mathds{1} for indicator functions

\usepackage{accents}                   % Allows for multiple accents on a single character

\usepackage{aliascnt}                  % Required for the \nuovothm macro that handles theorem counters

\usepackage{xcolor}                    % To define and use colors
\definecolor{myurlcolor}{rgb}{0,0,0.4}
\definecolor{mycitecolor}{rgb}{0,0.5,0}
\definecolor{myrefcolor}{rgb}{0.5,0,0}

\usepackage{graphicx}                  % To include images with \includegraphics
\usepackage{tikz}                      % To create programmable vector graphics
\usepackage{tikz-cd}                   % To easily create commutative diagrams
\usetikzlibrary{graphs, decorations.pathmorphing, decorations.markings}
\usepackage{caption}                   % To customize figure and table captions

\usepackage[sort]{natbib}                    % For bibliography management (e.g., author-year styles)

\usepackage[pagebackref=true, colorlinks=true]{hyperref} 
\hypersetup{
    linkcolor=myrefcolor,
    citecolor=mycitecolor,
    urlcolor=myurlcolor,
}
\usepackage[capitalize, nameinlink]{cleveref}   % For "intelligent" cross-references (\cref)

\usepackage{changepage}                % For custom margins (used in smallproof environment)
\usepackage{enumitem}                  % To customize lists (enumerate, itemize)
\usepackage{braket}                    % For Dirac notation in physics/mathematics
\usepackage{physics}

\newcommand\nuovothm[3]{%
\newaliascnt{#1}{theorem}%
\newtheorem{#1}[#1]{#2}%
\aliascntresetthe{#1}%
\crefname{#1}{#2}{#3}%
\AtBeginEnvironment{#1}{\addvspace{10pt}}%
}

\nuovothm{lemma}{Lemma}{Lemmas}
\nuovothm{proposition}{Proposition}{Propositions}
\nuovothm{corollary}{Corollary}{Corollaries}
\nuovothm{definition}{Definition}{Definitions}
\nuovothm{remark}{Remark}{Remarks}
\nuovothm{assumption}{Assumption}{Assumptions}
\nuovothm{example}{Example}{Examples}
\nuovothm{question}{Question}{Questions}
\nuovothm{theo}{Theorem}{Theorems}

\newcommand{\be}{\begin{equation}}
\newcommand{\ee}{\end{equation}}

\newcommand{\red}[1]{{#1}}

\newcommand{\de}{\partial}
\DeclareMathOperator\T{\textup{\textbf{T}}}

\usepackage{todonotes}                 % For inserting notes and TODOs in the margin (disable with the [disable] option)

\usepackage{quiver}
\usepackage{newpxtext}
\usepackage{etoolbox}

\usepackage{parskip}
\newcommand{\Lie}{\pounds}
\renewcommand{\Im}{\operatorname{Im}}
\NewDocumentCommand{\pder}{ O{} m }{ \frac{\partial #1}{\partial #2} }
\newcommand{\R}{\mathbb{R}}
\renewcommand{\[}{\begin{equation}}
\renewcommand{\]}{\end{equation}}

\title{Regularization of singular time-dependent Lagrangian systems}
\author{
    \texorpdfstring{M. de Le\'on$^{1,3,4}$ \href{https://orcid.org/0000-0002-8028-2348}{\includegraphics[scale=0.7]{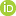}}}{M. De Leon},
    \texorpdfstring{R. Izquierdo-L\'opez$^{1,5}$ \href{https://orcid.org/0009-0007-8747-344X}{\includegraphics[scale=0.7]{ORCID.png}}}{R. Izquierdo-Lopez},
    \texorpdfstring{L. Schiavone$^{2,6}$ \href{https://orcid.org/0000-0002-1817-5752}{\includegraphics[scale=0.7]{ORCID.png}}}{L. Schiavone},
    \texorpdfstring{P. Soto-Mart\'in$^{1,7}$ \href{https://orcid.org/0009-0001-6958-3182}{\includegraphics[scale=0.7]{ORCID.png}}}{P. Soto-Martin} \\
    \footnotesize{$^{1}$\textit{Instituto de Ciencias Matemáticas, Campus Cantoblanco, Consejo Superior de Investigaciones Científicas,}} \\
    \footnotesize{\textit{Calle Nicolás Cabrera, 13–15, 28049, Madrid, Spain}} \\
    \footnotesize{$^{2}$\textit{Dipartimento di Matematica e Applicazioni Renato Caccioppoli, Università degli Studi di Napoli Federico II,}} \\
    \footnotesize{\textit{Via Cintia, Monte S. Angelo I, 80126, Napoli, Italy}} \\
    \footnotesize{$^{3}$\textit{Real Academia de Ciencias Exactas, Físicas y Naturales de España, C/Valverde, 22, Madrid 28004, Spain}}\\
    \footnotesize{$^{4}$\textit{e-mail: \texorpdfstring{\texttt{mdeleon@icmat.es}}{mdeleon@icmat.es}}} \\
    \footnotesize{$^{5}$\textit{e-mail: \texorpdfstring{\texttt{ruben.izquierdo@icmat.es}}{ruben.izquierdo@icmat.es}}} \\
    \footnotesize{$^{6}$\textit{e-mail: \texorpdfstring{\texttt{luca.schiavone@unina.it}}{luca.schiavone@unina.it}}} \\
    \footnotesize{$^{7}$\textit{e-mail: \texorpdfstring{\texttt{pablo@soto.es}}{pablo@soto.es}}}\\
}

\begin{document}
\maketitle

\begin{abstract}
One approach to studying the dynamics of a singular Lagrangian system is to attempt to regularize it, that is, to find an equivalent and regular system. 
%%%
In the case of time-independent singular Lagrangians, an approach due to \textit{A. Ibort} and \textit{J. Marín-Solano} is to use the coisotropic embedding theorem proved by \textit{M.J. Gotay} which states that any pre-symplectic manifold can be coisotropically embedded in a symplectic manifold. 
%%%
In this paper, we revisit these results and provide an alternative approach—also based on the coisotropic embedding theorem—that employs the Tulczyjew isomorphism and almost product structures, and allows for a slight generalization of the construction. In this revision, we also prove uniqueness of the Lagrangian regularization to first order.
%%%
Furthermore, we extend our methodology to the case of time-dependent singular Lagrangians.
%%%
\end{abstract}
\tableofcontents
\section{Introduction}

One of the greatest successes of symplectic geometry has been to serve as a setting for Hamiltonian mechanics as well as its Lagrangian description. Indeed, a Hamiltonian function is just a function $H$ defined on the cotangent bundle $\T^*Q$ of the configuration manifold $Q$, such that the Hamiltonian dynamics is provided by the corresponding Hamiltonian vector field $X_H$ obtained using the canonical symplectic form $\omega_Q$ on $\T^*Q$ \cite{AM,dLR1989}, say
$$
i_{X_H} \, \omega_Q =  \dd H \,.
$$
In the Lagrangian picture, given a Lagrangian function $L$ on the tangent bundle $\T Q$, one obtains a differential 2-form $\omega_L$ such that the equation
$$
i_{\xi_L} \, \omega_L =  \dd E_L
$$
provides the Euler-Lagrange vector field $\xi_L$ \cite{dLR1989}. 

We need to make two clarifications about this last equation: (1) $\omega_L$ is symplectic if and only if the Lagrangian $L$ is regular (the Hessian matrix of $L$ with respect to velocities is regular), and (2) $\xi_L$ is a second-order differential equation (SODE, for short) such that its solutions (the projections onto $Q$ of its integral curves) are the solutions of the Euler-Lagrange equations.
Lagrangians giving rise to degenerate 2-forms $\omega_L$ are usually referred to as \textit{singular}.
%%%
\noindent The Legendre transformation $\operatorname{Leg}: \T Q \longrightarrow \T^*Q$ connects both descriptions in a natural way (see \cite{godbillon,dLR1989} for more details).

One of the most interesting contributions in \textit{P.A.M. Dirac}'s and \textit{P.G. Bergmann}'s works was the introduction of the constraint algorithm for dealing with singular Lagrangians, now known as the Dirac-Bergmann algorithm (see \cite{dirac,bergmann,dirac1,dirac2}). That algorithm has been developed in geometric terms using the notion of pre-symplectic manifolds by \textit{M.J. Gotay} and collaborators \cite{gotay0,gotay_nester1,gotay_nester2}, incorporating in the Lagrangian picture the problem of the second-order differential equation, a remarkable distinction between the Lagrangian and the Hamiltonian descriptions.

The above algorithm has also been constructed for the case of singular Lagrangians depending explicitly on time \cite{chinea}. In that case, the geometric scenarios are usually taken to be $\T Q \times \mathbb{R}$ and $\T^*Q \times \mathbb{R}$, and now, the singular case uses the notion of pre-cosymplectic structure. More generally, we may regard time-dependent (or non-autonomous) Lagrangians as functions defined in $J^1 \pi$, where $\pi \colon \mathbf{Q} \longrightarrow \mathbb{R}$ is a (not necessarily trivialized) fiber bundle (see \cite{Krupkova-GeometryOVE-1997,manolo1,manolo2}).

In both time-independent and time-dependent cases, a Poisson bracket that provides the dynamics (the so-called Dirac bracket) and the dynamics themselves, modulo the kernel of the forms, can be found on the so-called \textit{final constraint submanifold} selected by the algorithm.

On the other hand, coisotropic submanifolds play an important role in classical mechanics and field theories because they allow for the development of a procedure for reducing dynamics. Indeed, a foundational result due to \textit{A. Weinstein} \cite{alan0} (see also \cite{AM}) establishes that the quotient space of a coisotropic submanifold by its characteristic foliation naturally inherits a symplectic structure, providing a rigorous geometric setting for Hamiltonian reduction. The reduction is then accomplished when we consider dynamics interpreted as a Lagrangian submanifold.
%%%

By taking the opposite road (unfolding vs reduction), coisotropic submanifolds are also relevant in the context of regularization problems. 
In fact, the Dirac-Bergmann algorithm is not the only way we could use to regularize a singular Lagrangian. Indeed, \textit{M.J. Gotay} \cite{gotay1} proved a \textit{coisotropic embedding theorem} (see also \cite{R_Marle-CoisotropicSubmanifolds-1983} for earlier contributions by C.M. Marle), stating that any pre-symplectic manifold can be embedded into a symplectic manifold as a coisotropic submanifold, and that this embedding is unique in a neighborhood of the original manifold. This result has been generalized to many other relevant geometric scenarios by the authors \cite{zoo, Schiavone-CoisotropicMultisymplectic-2026}, by taking advantage of the point of view developed in \cite{luca}. Using the coisotropic embedding theorem, as well as a natural classification of Lagrangians \cite{frans}, \textit{A. Ibort} and \textit{J. Marín-Solano} \cite{ibort-marin}  were able to develop a regularization method for certain types of Lagrangian systems (called type II Lagrangians). This classification for Lagrangian functions has been extended by \textit{M. de Le\'on et al.} (see \cite{mello}) and later reconsidered again by \textit{A. Ibort} and \textit{J. Mar{\'\i}n-Solano} \cite{ibortcos}.
%%%

The main objective of this paper is to advance this programme of regularization of singular Lagrangian systems, which will culminate in the case of classical field theories in a future work.
Thus, we begin by carefully re-examining the results of \textit{A. Ibort} and \textit{J. Marín-Solano} in the pre-symplectic case, using a new methodology based on the use of almost product structures and Tulczyjew triples \cite{lodekh,lodekl}. Most notably, we employ a (to our knowledge) novel generalization of the Tulczyjew triples adapted to a foliated manifold. The construction is presented in \cref{Subsec: A Tulczyjew isomorphism for foliations}. For completeness, let us recall that the use of almost product structures to deal with singular Lagrangian systems was introduced in \cite{almostproduct} (see also \cite{dLR1989}). This approach has allowed us to clarify some of these results and introduce an alternative for explicitly constructing a regular Lagrangian equivalent to the original singular one by using an auxiliary connection. In particular, we would like to highlight the following results in the pre-symplectic case:
\begin{itemize}
    \item In \cref{thm:Lagrangian_coisotropic_embedding_symplectic}, we give a global description of the Lagrangian in the regularized manifold, proving that the regularization provides a \emph{global} Lagrangian, and not only locally so. This description is more general than the one previously found in the literature. Indeed, it depends on strictly less geometric ingredients (a connection instead of a metric). 
    \item Finally, this regularization is proved to be unique and independent of all choices to first order in \cref{thm:uniqueness_sympelctic_first_order}, a result which was not present in the literature.
\end{itemize}

Next, we studied the case of singular Lagrangians explicitly dependent on time, obtaining similar results using pre-cosymplectic geometry (namely \cref{thm:Lagrangian_pre-cosymplectic_regularization} and \cref{thm:precosymplectic_Lag_uniqueness}, respectively). Most importantly, the Reeb vector field needs to be taken into consideration, which is a significant difference from the autonomous case. In addition, our methodology does not explicitly rely on the mentioned classification of Lagrangians (even if it requires some conditions, equivalent to those used in \cite{ibort-marin}, to be fulfilled by the Lagrangian), and so opens a clear path to extend it to the case of singular Lagrangian field theories.

The paper is structured as follows. Following the introduction, we devote \cref{Sec: Preliminaries} to recall some well-known notions and results on distributions and foliations, as well as to prove an extension of Tulczyjew's triple to the case of foliations defined on a smooth manifold; tangent structures and stable tangent structures are also reviewed. \cref{Sec:Regularization of autonomous systems} is devoted to reconsidering the regularization of singular autonomous Lagrangians and developing a new technique that differs from that developed in \cite{ibort-marin}, as mentioned above. In \cref{Sec:Regularization of non-autonomous systems}, we extend this regularization scheme to the context of singular time-dependent Lagrangians. The construction presented in \cref{Sec:Regularization of non-autonomous systems} is illustrated by studying the trivialized case and degenerate metrics. Finally, we include a section on conclusions and an outlook for future work to be carried out.

\section{Preliminaries}
\label{Sec: Preliminaries}

\subsection{Distributions and foliations}
\label{Subsec: Distributions and foliations}

In this section, we recall the basic definitions and geometric properties of distributions and foliations on smooth manifolds. 
%%%

\begin{definition}[\textsc{Regular Distribution}]
Let $Q$ be a smooth manifold of dimension $n$. 
%%%
A \textbf{distribution} $D$ on $Q$ of rank $r$ is a smooth assignment of an $r$-dimensional subspace $D_q \subset \T_q Q$ to each point $q \in Q$. 
%%%
The distribution is said to be \textbf{regular} if the rank $r$ is constant over $Q$.
\end{definition}

\begin{definition}[\textsc{Involutivity and Integrability}]
A distribution $D$ is said to be \textbf{involutive} if it is closed under the Lie bracket, i.e.,
\be
[\, X, Y \,] \in D \,,\quad \forall X, Y \in \Gamma(D) \,,
\ee
where $\Gamma(D)$ denotes the space of smooth sections of $D$.
%%%
A distribution is said to be \textbf{integrable} if, for every point $q \in Q$, there exists an integral submanifold of $D$ passing through $q$ (i.e., a submanifold whose tangent space at each point coincides with $D$).
\end{definition}

\begin{remark}
Sometimes, one distinguishes between maximal integral submanifolds and integral submanifolds, when maximal dimension means the rank of the involutive distribution.
    \end{remark}

\noindent The fundamental link between these concepts is provided by the Frobenius Theorem \cite{warner}.
%%%

\begin{theo}[\textsc{Frobenius Theorem}]
A regular distribution $D$ on a smooth manifold $Q$ is integrable if and only if it is involutive.
\end{theo}
%%%

\begin{definition}[\textsc{Regular Foliation}]
A \textbf{regular foliation} $\mathcal{F}$ of dimension $r$ (and codimension $k = n-r$) on a manifold $Q$ is a partition of $Q$ into a family of disjoint, connected, immersed submanifolds $\{\mathcal{L}_\alpha\}_{\alpha \in A}$ called \textbf{leaves}, such that:
\begin{enumerate}
    \item For every $q \in Q$, there exists a unique leaf $\mathcal{L}_q$ containing $q$.
    \item Around every point $q \in Q$, there exists a local coordinate chart $(U, \phi)$ with coordinates $(x^1, \dots, x^k, f^1, \dots, f^r)$ such that the connected components of the intersection of any leaf with $U$ are described by the equations
    \be
    x^a = c^a \,, \quad a = 1, \dots, k \,,
    \ee
    where the constants $c^a$ determine the local leaf.
\end{enumerate}
%%%
Such a chart is called a \textbf{foliated chart} or \textbf{adapted chart}.
\end{definition}

\begin{remark}[\textsc{Relation to Distributions}]
Every regular foliation $\mathcal{F}$ defines a unique involutive regular distribution $D$, where $D_q = \T_q \mathcal{L}_q$. 
%%%
Conversely, by the Frobenius Theorem, every regular involutive distribution $D$ generates a regular foliation $\mathcal{F}_D$ whose leaves are the maximal integral manifolds of $D$. 
%%%
\end{remark}

\noindent In an adapted coordinate system $(x^a, f^A)$, the distribution $D$ is locally spanned by the vector fields
\be
D \,=\, \operatorname{span}\left\{ \frac{\partial}{\partial f^A} \right\}_{A=1,\dots,r} \,.
\ee
The coordinates $x^a$ serve as local coordinates on the space of leaves when it exists as a quotient manifold, while $f^A$ serve as coordinates along the leaf.

\begin{definition}[\textsc{Tangent Bundle of a Foliation}]
Let $\mathcal{F}$ be a regular foliation on $Q$. The \textbf{tangent bundle of the foliation}, denoted by $\mathscr{T}\mathcal{F}$, is the disjoint union of the tangent bundles of its leaves:
\be
\mathscr{T}\mathcal{F} \,:=\, \bigcup_{L \in \mathcal{F}} \T L \,.
\ee
This set carries the structure of a smooth vector bundle of rank $r$ over $Q$, and it is isomorphic to the distribution $D \subset \T Q$ associated with $\mathcal{F}$.
\end{definition}

\noindent Given an adapted chart $(x^a, f^A)$ on $Q$, we induce local coordinates on $\mathscr{T}\mathcal{F}$ denoted by:
\be
\left\{\, x^a,\, f^A,\, {v_f}^A \,\right\}_{a=1,\dots,k;\, A=1,\dots,r} \,.
\ee
Here, a point in $\mathscr{T}\mathcal{F}$ is locally represented as vector $v = {v_f}^A \frac{\partial}{\partial f^A}$ attached to the point $(x^a, f^A)$. Note that the "transverse velocities" are identically zero, ${v_x}^a = 0$.

\begin{definition}[\textsc{Cotangent Bundle of a Foliation}]
The \textbf{cotangent bundle of the foliation}, denoted by $\mathscr{T}^*\mathcal{F}$, is the disjoint union of the cotangent bundles of its leaves:
\be
\mathscr{T}^*\mathcal{F} \,:=\, \bigcup_{L \in \mathcal{F}} \T^* L \,.
\ee
It carries the structure of a smooth vector bundle of rank $r$ over $Q$. It is canonically isomorphic to the dual of the distribution $D$, say $D^*$.
\end{definition}
%%%

\noindent In the adapted chart defined above, local coordinates on $\mathscr{T}^*\mathcal{F}$ are denoted by:
\be
\left\{\, x^a,\, f^A,\, {p_f}_A \,\right\}_{a=1,\dots,k;\, A=1,\dots,r} \,.
\ee
A point in $\mathscr{T}^*\mathcal{F}$ is locally represented by a covector $\alpha = {p_f}_A \dd f^A$ restricted to the tangent space of the leaf.
%%%

\begin{remark}
It is important to note that the sets $\mathscr{T}\mathcal{F}$ and $\mathscr{T}^*\mathcal{F}$ are smooth vector bundles of rank $r$ over the base manifold $Q$, being a subbundle and a quotient bundle of the tangent and cotangent bundles of $Q$, respectively.
%%%

\noindent Indeed, by definition, the fiber of $\mathscr{T}\mathcal{F}$ at $q$ is the tangent space to the leaf passing through $q$, i.e., $(\mathscr{T}\mathcal{F})_q = \T_q \mathcal{L}_q$. 
%%%
Since $\mathcal{F}$ is generated by the regular distribution $D$, we have $\T_q \mathcal{L}_q = D_q$. 
%%%
Thus, $\mathscr{T}\mathcal{F}$ coincides with the total space of the distribution $D$.
%%%
Since $D$ is a regular distribution, it is by definition a vector subbundle of $\T Q$.
%%%

\noindent On the other hand, the fiber of $\mathscr{T}^*\mathcal{F}$ at $q$ is the dual space of the tangent space to the leaf, i.e., $(\mathscr{T}^*\mathcal{F})_q = (\T_q \mathcal{L}_q)^* = {D_q}^*$.
%%%
Consider the \textit{annihilator} of the distribution, denoted by $D^\circ \subset \T^* Q$, which is the subbundle of covectors that vanish on $D$. 
%%%
We have a short exact sequence of vector bundles over $Q$:
\be
0 \longrightarrow D^\circ \longrightarrow \T^* Q \xrightarrow{\;\;\pi\;\;} \mathscr{T}^*\mathcal{F} \longrightarrow 0 \,,
\ee
where the map $\pi$ is the restriction of a covector in $\T^*_q Q$ to the subspace $D_q$.
%%%
Since this restriction is surjective with kernel $D^\circ$, by the first isomorphism theorem for vector spaces applied fiber-wise, we have the canonical isomorphism:
\be
\mathscr{T}^*\mathcal{F} \,\cong\, \T^* Q \,/\, D^\circ \,.
\ee
Thus, $\mathscr{T}^*\mathcal{F}$ carries the structure of a (quotient) vector subbundle of $\T^* Q$.
%%%
\end{remark}

\begin{remark}
As vector bundles over $Q$, $\mathscr{T}\mathcal{F}$ and $\mathscr{T}^*\mathcal{F}$ are dual to each other.
%%%
The duality pairing
\be
\langle \,\cdot\,,\, \cdot\, \rangle \;\colon\; \mathscr{T}^*\mathcal{F} \times_Q \mathscr{T}\mathcal{F} \to \mathbb{R}
\ee
is defined naturally by the evaluation map. Let $v \in (\mathscr{T}\mathcal{F})_q$ and $\alpha \in (\mathscr{T}^*\mathcal{F})_q$. Since $\alpha$ is a linear functional on $(\mathscr{T}\mathcal{F})_q$, the pairing is simply $\langle \alpha, v \rangle_{\mathcal{F}} = \alpha(v)$.

\noindent In adapted local coordinates $(x^a, f^A)$, a vector $v \in \mathscr{T}\mathcal{F}$ reads 
\be
v = {v_f}^A \frac{\partial}{\partial f^A} \,.
\ee
%%%
A covector in the ambient space $\T^* Q$ reads $\widetilde{\alpha} = p_a \dd x^a + {p_f}_A \dd f^A$.
%%%
Since the 1-forms $\dd x^a$ annihilate the distribution $D = \operatorname{span}\{\frac{\partial}{\partial f^A}\}$, they form a local basis for the annihilator $D^\circ$.
%%%
Therefore, the equivalence class in the quotient $\T^* Q / D^\circ$ (which represents the element $\alpha \in \mathscr{T}^*\mathcal{F}$) is determined solely by the components ${p_f}_A$.
%%%
The pairing is thus given explicitly by:
\be \label{Eq: pairing foliations}
\langle\, \alpha,\, v \,\rangle \,=\, \left( p_a \dd x^a + {p_f}_A \dd f^A \right) \left( {v_f}^B \frac{\partial}{\partial f^B} \right) \,=\, {p_f}_A {v_f}^A \,.
\ee
\end{remark}

\begin{definition}[\textsc{Almost-Product Structure}]\cite{dLR1989}
\label{Def: Almost Product Structure}
An \textbf{almost-product structure} on a smooth manifold $Q$ is a smooth point-wise splitting of its tangent bundle into a direct sum of two complementary distributions. That is, for each $q \in Q$, the tangent space decomposes as:
\be
\T_q Q \,=\, D_q \oplus H_q \,,
\ee
where $D$ and $H$ are smooth subbundles of $\T Q$.
%%%

\noindent Equivalently, such a splitting is uniquely characterized by a smooth $(1,1)$-tensor field $P \in \Gamma(\T^* Q \otimes \T Q)$ that is idempotent, namely $P \circ P = P$. 
%%%
This tensor acts as a projector onto $D$ along $H$, meaning that $\operatorname{Im}(P) = D$ and $\operatorname{ker}(P) = H$.
%%%

Given a pre-existing regular distribution $D$ on $Q$, an almost-product structure $P$ is said to be \textbf{adapted} to $D$ if its image coincides with the distribution, $\operatorname{Im}(P) = D$. 
%%%
In this scenario, choosing $P$ is equivalent to smoothly assigning the complementary horizontal distribution $H = \operatorname{ker}(P)$.
%%%
If the distribution $D$ is integrable, it generates a regular foliation $\mathcal{F}$. 
%%%
Let $(x^a, f^A)$ be a system of local adapted coordinates, such that $D = \operatorname{span}\left\{\frac{\partial}{\partial f^A}\right\}$. 
%%%
Since $P$ must act as the identity on its image, we have $P\left(\frac{\partial}{\partial f^A}\right) = \frac{\partial}{\partial f^A}$. 
%%%
Thus, the most general local expression for an almost-product structure $P$ adapted to $D$ is given by:
\be \label{Eq: almost product adapted}
P \,=\, \left( \dd f^A - P^A_a(x, f) \dd x^a \right) \otimes \frac{\partial}{\partial f^A} \,,
\ee
where the local functions $P^A_a(x, f)$ uniquely determine the choice of the complementary distribution $\operatorname{ker}(P) = \operatorname{span}\left\{\frac{\partial}{\partial x^a} + P^A_a \frac{\partial}{\partial f^A}\right\}$.
\end{definition}

\subsection{A Tulczyjew isomorphism for foliations}
\label{Subsec: A Tulczyjew isomorphism for foliations}

Since it will be relevant for the whole manuscript, we devote this section to adapting the notion of one of the \textit{Tulczyjew isomorphisms} to the context of regular foliations.
%%%
The isomorphism we are interested in is the one existing between the iterated bundles $\mathbf{TT}^* Q$ and $\mathbf{T}^* \mathbf{T}Q$ over a smooth differential manifold $Q$ that we recall in the following lines \cite{lodekh,lodekl,dLR1989}.
%%%

Consider an $n$-dimensional smooth differential manifold $Q$ with the system of local coordinates
\be
\left\{\, q^j \,\right\}_{j=1,...,n} \,.
\ee
%%%
Its tangent bundle
\be
\pi_{Q} \;\;:\;\; \mathbf{T}Q \to Q \,,
\ee
inherits the natural system of local coordinates
\be
\left\{\, q^j,\, \dot{q}^j \,\right\}_{j=1,...,n} \,,
\ee
where
\be
\pi_{Q}(q^j,\, \dot{q}^j) \,=\, q^j \,.
\ee
%%%
Now, consider the double bundle $\mathbf{TT}Q$ with the system of local coordinates
\be
\left\{\, q^j,\, \dot{q}^j,\, {v_q}^j,\, {v_{\dot{q}}}^j \,\right\}_{j=1,...,n} \,.
\ee
%%%
It can be given two structures of vector bundle over $\mathbf{T}Q$, namely
\be
\begin{tikzcd}
& \mathbf{TT}Q \arrow[dl, swap, "\pi_{\mathbf{T}Q}"] \arrow[dr, "T\pi_{Q}"] & \\
\mathbf{T}Q \arrow[dr, swap, "\pi_Q"] & & \mathbf{T}Q \arrow[dl, "\pi_{Q}"] \\
& Q &
\end{tikzcd}
\ee
where 
\be
\pi_{\mathbf{T}Q}(q^j,\, \dot{q}^j,\, {v_q}^j,\, {v_{\dot{q}}}^j) \,=\, (q^j,\, \dot{q}^j) \,,
\ee
and
\be
T \pi_Q (q^j,\, \dot{q}^j,\, {v_q}^j,\, {v_{\dot{q}}}^j) \,=\, (q^j,\, {v_q}^j) \,.
\ee
%%%
There exists a natural isomorphism of fiber bundles of $\pi_{\mathbf{T}Q}$ and $T \pi_{Q}$ defined categorically as the unique double vector bundle isomorphism $\delta: \mathbf{TT}Q \to \mathbf{TT}Q$ that interchanges the two vector bundle projections---meaning it satisfies $\pi_{\mathbf{T}Q} \circ \delta = T\pi_Q$ and $T\pi_Q \circ \delta = \pi_{\mathbf{T}Q}$, while acting as the identity map on the core of the double vector bundle (which is canonically isomorphic to $\mathbf{T}Q$).
%%%
In local coordinates, it reads:
\be
\delta \;\;:\;\; \mathbf{TT}Q \to \mathbf{TT}Q \;\;:\;\; (q^j,\, \dot{q}^j,\, {v_q}^j,\, {v_{\dot{q}}}^j) \mapsto (q^j,\, {v_q}^j,\, \dot{q}^j,\, {v_{\dot{q}}}^j) \,.
\ee
%%%

\noindent Consider the iterated bundle $\mathbf{T}^* \mathbf{T}Q$, with the system of local coordinates
\be
\left\{\, q^j,\, \dot{q}^j,\, {p_q}_j,\, {p_{\dot{q}}}_j \,\right\}_{j=1,...,n} \,.
\ee
%%%
It is the dual vector bundle to $\pi_{\mathbf{T}Q}$ with respect to the pairing
\be
\langle\, \rho,\, \xi \,\rangle \,=\, {p_q}_j {v_q}^j + {p_{\dot{q}}}_j {v_{\dot{q}}}^j \,,
\ee
where $\rho$ is an element of $\mathbf{T}^* \mathbf{T}Q$ and $\xi$ is an element of $\mathbf{TT}Q$.
%%%

\noindent On the other hand, the iterated bundle $\mathbf{TT}^* Q$, with the system of local coordinates
\be
\left\{\, q^j,\, p_j,\, \dot{q}^j,\, \dot{p}_j \,\right\}_{j=1,...,n} \,,
\ee
is canonically the dual vector bundle to $T\pi_{Q}$.
%%%
The duality pairing is defined intrinsically as the tangent lift of the canonical pairing between $\mathbf{T}^*Q$ and $\mathbf{T}Q$.
%%%
Specifically, if we consider a curve $\gamma(t) = (q^j(t),\, \dot{q}^j(t))$ in $\mathbf{T}Q$ and a curve $\lambda(t) = (q^j(t),\, p_j(t))$ in $\mathbf{T}^*Q$ projecting to the same base curve on $Q$, the pairing is the time derivative of the contraction $\langle \lambda(t), \gamma(t) \rangle$.
%%%
In local coordinates, this operation reads:
\be
\frac{\dd}{\dd t} (p_j(t) \dot{q}^j(t)) \,=\, \dot{p}_j(t) \dot{q}^j(t) + p_j(t) \ddot{q}^j(t) \,,
\ee
%%%
yielding the pairing:
\be
\langle \, \eta,\, \psi\, \rangle' \,=\, \dot{p}_j {v_q}^j + p_j {v_{\dot{q}}}^j \,,
\ee
where $\eta$ is an element of $\mathbf{TT}^* Q$ and $\psi$ is an element of $\mathbf{TT}Q$.
%%%

\noindent The transpose map of $\delta$ with respect to the pairing $\langle\,\cdot\,,\,\cdot\,\rangle'$ is the Tulczyjew isomorphism between $\mathbf{TT}^* Q$ and $\mathbf{T}^* \mathbf{T}Q$.
%%%
It reads locally
\be
\alpha \,=\, \delta^T \;\;:\;\; \mathbf{TT}^* Q \to \mathbf{T}^*\mathbf{T}Q \;\;: \;\; (q^j,\, p_j,\, \dot{q}^j,\, \dot{p}_j) \mapsto (q^j,\, \dot{q}^j,\, {p_q}_j \,=\, \dot{p}_j,\, {p_{\dot{q}}}_j \,=\, p_j) \,.
\ee
%%%

Now, given a regular foliation $\mathcal{F}_{\undertilde{K}}$ on $Q$ generated by a regular integrable distribution $\undertilde{K}$ on $Q$, let us consider the tangent distribution $K \,=\,{ \undertilde{K}^C}$. That is, the distribution generated by the vertical and the complete lifts of the vector fields generating $\undertilde{K}$\footnote{We refer to \cite{yanoishihara,dLR1989} or to \cref{Def: lifts vector fields} for the definition of vertical and complete lifts of a vector field.}. 
%%%
{It is a regular distribution} on $\T Q$ that provides a regular foliation $\mathcal{F}_K$ on $\T Q$.
%%%

\noindent Denote by
\be
\left\{\, x^a,\, f^A \,\right\}_{a=1,...,l;A=1,...,r} \,,
\ee
a system of local coordinates on $Q$ adapted to the foliation $\mathcal{F}_{\undertilde{K}}$, and by
\be
\left\{\, x^a,\, \dot{x}^a,\, f^A,\, \dot{f}^A\,\right\}_{a=1,...,l;A=1,...,r} \,,
\ee
a system of local coordinates on $\T Q$ adapted to the foliation $\mathcal{F}_{K} \,=\, \mathcal{F}_{{ \undertilde{K}^C}}$.
%%%
The set of coordinates
\be
\left\{\, x^a \,\right\}_{a=1,...,l} \,,
\ee
and 
\be
\left\{\, x^a,\, \dot{x}^a \,\right\}_{a=1,...,l} \,,
\ee
represent systems of coordinates on the spaces of leaves of $\mathcal{F}_{\undertilde{K}}$ and $\mathcal{F}_K$, {which may be treated as smooth manifolds locally}.
%%%

\noindent Let us denote by
\be
\left\{\, x^a,\, \dot{x}^a,\, f^A,\, \dot{f}^A,\, {v_f}^A,\, {v_{\dot{f}}}^A \,\right\}_{a=1,...,l,A=1,...,r} \,,
\ee
a system of local coordinates on $\mathscr{T}\mathcal{F}_{{ \undertilde{K}^C}} \,=\, \mathscr{T} \mathcal{F}_K$, and by
\be
\left\{\, x^a,\, f^A,\, \dot{f}^A,\, \dot{x}^a,\, {v_f}^A,\, {v_{\dot{f}}}^A \,\right\}_{a=1,...,l,A=1,...,r} \,,
\ee
a system of local coordinates on $\mathbf{T}\mathscr{T}\mathcal{F}_{\undertilde{K}}$.
%%%
As for the bundles $\pi_{\mathbf{T}Q}$ and $T \pi_{Q}$, an isomorphism between the bundles $\mathscr{T}\mathcal{F}_{{ \undertilde{K}^C}}$ and $\mathbf{T}\mathscr{T}\mathcal{F}_{\undertilde{K}}$ exists. Intrinsically, this isomorphism is exactly the restriction of the canonical involution $\delta : \mathbf{TT}Q \to \mathbf{TT}Q$ to the subbundle $\mathscr{T}\mathcal{F}_{\mathbf{T} \undertilde{K}}$. 
%%%
Indeed, since $\mathscr{T}\mathcal{F}_{\undertilde{K}}$ is a smooth submanifold of $\mathbf{T}Q$ (being the total space of the distribution $\undertilde{K}$), its tangent bundle $\mathbf{T}\mathscr{T}\mathcal{F}_{\undertilde{K}}$ embeds naturally into $\mathbf{TT}Q$. 
%%%
At the same time, $\mathscr{T}\mathcal{F}_{{ \undertilde{K}^C}}$ is naturally a subbundle of $\mathbf{TT}Q$. 
%%%
It is straightforward to show that the canonical involution $\delta$ maps this subbundle exactly onto $\mathbf{T}\mathscr{T}\mathcal{F}_{\undertilde{K}}$. 
%%%
We can therefore define:
\be
\delta_{\mathcal{F}} := \delta \big|_{\mathscr{T}\mathcal{F}_{{ \undertilde{K}^C}}} \;\;:\;\; \mathscr{T}\mathcal{F}_{{\undertilde{K}^C}} \to \mathbf{T}\mathscr{T}\mathcal{F}_{\undertilde{K}} \,,
\ee
which locally reads:
\be
\delta_{\mathcal{F}} (x^a,\, \dot{x}^a,\, f^A,\, \dot{f}^A,\, {v_f}^A,\, {v_{\dot{f}}}^A) \mapsto (x^a,\, f^A,\, \dot{f}^A \,=\, {v_f}^A,\, \dot{x}^a,\, {v_f}^A \,=\, \dot{f}^A,\, {v_{\dot{f}}}^A) \,.
\ee

\noindent Similarly to what happens for the iterated bundles considered by Tulczyjew, the bundle $\mathscr{T}^*\mathcal{F}_{{ \undertilde{K}^C}}$, where we chose the system of local coordinates
\be
\left\{\, x^a,\, \dot{x}^a,\, f^A,\, \dot{f}^A,\, {\mu_f}_A,\, {\mu_{\dot{f}}}_A \,\right\}_{a=1,...,l,A=1,...,r} \,,
\ee
is the dual bundle to $\mathscr{T}\mathcal{F}_{{ \undertilde{K}^C}}$ with respect to the pairing
\be
\langle \, \rho,\, \xi\,\rangle \,=\, {\mu_f}_A {v_f}^A + {\mu_{\dot{f}}}_A {v_{\dot{f}}}^A \,,
\ee
where $\rho$ is an element of $\mathscr{T}^* \mathcal{F}_{{ \undertilde{K}^C}}$ and $\xi$ is an element of $\mathscr{T}\mathcal{F}_{{ \undertilde{K}^C}}$.
%%%

\noindent Similarly, as for the bundle $\mathbf{TT}^* Q$, the bundle $\mathbf{T}\mathscr{T}^* \mathcal{F}_{\undertilde{K}}$ is the dual vector bundle to $\mathbf{T}\mathscr{T}\mathcal{F}_{\undertilde{K}}$.
%%%
The duality pairing is given by the tangent lift of the canonical pairing \cref{Eq: pairing foliations} between $\mathscr{T}^*\mathcal{F}_{\undertilde{K}}$ and $\mathscr{T}\mathcal{F}_{\undertilde{K}}$.
%%%
Explicitly, this means that for any curve $\gamma(t)$ in $\mathscr{T}\mathcal{F}_{\undertilde{K}}$ and any curve $\lambda(t)$ in $\mathscr{T}^*\mathcal{F}_{\undertilde{K}}$ projecting to the same base curve on $Q$, the pairing of their tangent vectors is the time derivative of their contraction:
\be
\frac{\dd}{\dd t} \langle \lambda(t), \gamma(t) \rangle \,=\, \frac{\dd}{\dd t} \left( \mu_A(t) f^A(t) \right) \,=\, \dot{\mu}_A f^A  + \mu_A \dot{f}^A \,,
\ee
%%%
yielding the pairing
\be
\langle\, \eta,\, \psi \, \rangle' \,=\, \dot{\mu}_A f^A + \mu_A \dot{f}^A \,,
\ee
where $\eta$ is an element of $\mathbf{T}\mathscr{T}^* \mathcal{F}_{\undertilde{K}}$ (with coordinates $(x^a, f^A, \mu_A, \dot{x}^a, \dot{f}^A, \dot{\mu}_A)$) and $\psi$ is an element of $\mathbf{T}\mathscr{T}\mathcal{F}_{\undertilde{K}}$.

\noindent The transpose map of $\delta_{\mathcal{F}}$ with respect to the pairing $\langle\,\cdot\,,\,\cdot\,\rangle'$ is an isomorphism between $\mathbf{T}\mathscr{T}^* \mathcal{F}_{\undertilde{K}}$ and $\mathscr{T}^* \mathcal{F}_{{ \undertilde{K}^C}}$.
%%%
It reads
\be \label{Eq: Tulczyjew foliations}
\begin{split}
\alpha \,=\, {\delta_\mathcal{F}}^T \;\;:\;\; &\mathbf{T}\mathscr{T}^* \mathcal{F}_{\undertilde{K}} \to \mathscr{T}^*\mathcal{F}_{{ \undertilde{K}^C}}  \\
:\;\; &(x^a,\, f^A,\, \mu_A,\, \dot{x}^a,\, \dot{f}^A,\, \dot{\mu}_A) \mapsto (x^a,\, \dot{x}^a,\, f^A,\, \dot{f}^A,\, {\mu_f}_A \,=\, \dot{\mu}_{A},\, {\mu_{\dot{f}}}_A \,=\, \mu_A) \,.
\end{split}
\ee
%%%

\noindent For completeness, we would like to remark that Tulczyjew triples have been generalized to wider contexts, besides tangent bundles and foliated manifolds. In fact, the isomorphism presented here is a particular instance of the Tulczyjew isomorphism of Lie algebroids when considering an integrable distribution as such (we refer the reader to \cite{GGU2006,GG2008,Martinez2001} for more information).

\subsection{Tangent structures}
\label{Subsec: Tangent structures}

In this section we introduce tangent structures, which correspond to the geometric structures generalizing the local picture of $\T Q$, for some configuration manifold $Q$ (see \cite{dLR1989,yanoishihara} for further details). First, let us define some elementary operations on the tangent bundle $\T Q$ of a configuration manifold $Q$ and study its geometry.

\begin{definition}[\textsc{Lifts of vector fields}]
\label{Def: lifts vector fields}
Let $X \in \mathfrak{X}(Q)$ be a vector field on $Q$.
\begin{itemize}
    \item The \textbf{vertical lift} of $X$, denoted by $X^V \in \mathfrak{X}(\T Q)$, is the unique vector field on $\T Q$ such that for any 1-form $\alpha$ on $Q$, $X^V(i_\alpha) = \alpha(X) \circ \tau$, where $i_\alpha$ is the fiber-wise linear function on $\T Q$ induced by $\alpha$ (locally $v^i \alpha_i$).
    %%%
    In local coordinates $(q^i, v^i)$, if $X = X^i(q) \frac{\partial}{\partial q^i}$, then:
    \be
    \label{eq:local_expression_verticallift}
    X^V = X^i(q) \frac{\partial}{\partial v^i} \,.
    \ee
    %%%
    
    \item The \textbf{complete lift} (or \textbf{tangent lift}) of $X$, denoted by $X^C \in \mathfrak{X}(\T Q)$, is the vector field on $\T Q$ whose flow is the tangent lift of the flow of $X$.
%%%
That is, if $\phi_t$ is the flow of $X$, then $\Phi_t = T \phi_t$ is the flow of $X^C$.
    %%%
    In local coordinates, it reads:
    \be
    X^C = X^i(q) \frac{\partial}{\partial q^i} + v^k \frac{\partial X^i}{\partial q^k} \frac{\partial}{\partial v^i} \,.
    \ee
\end{itemize}
%%%

\noindent The mapping $X \mapsto X^C$ is a Lie algebra homomorphism from $\mathfrak{X}(Q)$ to $\mathfrak{X}(\T Q)$, while the vertical lift is commutative.
%%%
Specifically, for any $X, Y \in \mathfrak{X}(Q)$, the following bracket relations hold:
\begin{align}
    [X^V, Y^V] &= 0 \,, \label{Eq: VV bracket} \\
    [X^C, Y^V] &= [X, Y]^V \,, \label{Eq: CV bracket} \\
    [X^C, Y^C] &= [X, Y]^C \,. \label{Eq: CC bracket}
\end{align}
\end{definition}

\begin{remark}
\label{remark:complete_generates_tangent}
Notice that complete lifts of vector fields $X^{C}$ generate $\T Q$ point-wise, for every $v \in \T Q$. As a consequence, we may use the complete lift to compute the lift of tensors of different degree.
\end{remark}

\noindent A first example of the idea presented in \cref{remark:complete_generates_tangent} is the lift of almost-product structures to the tangent bundle as well. The study of these lifts will result useful in the sequel. {For the sake of exposition, we restrict to the case of almost product structures complementing an integrable distribution.}
%%% Their complete lift is defined below.
%%%

\begin{definition}[\textsc{Complete lift of an almost product structure}]
\label{Def: Complete Lift}
Let $\undertilde{P}$ be an almost product structure on the configuration manifold $Q$ {complementing an integrable distribution}. {Locally, with adapted coordinates,} the projector {reads as}:
\be
\undertilde{P} \,=\, \left(\dd f^A - \undertilde{P}^A_a(x, f) \dd x^a \right) \otimes \frac{\partial}{\partial f^A} \,,
\ee
where $(x^a,\,f^A)$ denote local coordinates on $Q$ adapted to the foliation induced by $\undertilde{P}$.
%%%
The complete lift of $\undertilde{P}$ to the tangent bundle $\T Q$ is defined by the condition
\be
{\undertilde{P}}^C(X^C) \,=\, \left[\undertilde{P}(X)\right]^C \,,\;\;\; \forall \;\; X \in \mathfrak{X}(Q) \,. 
\ee
%%%
In the induced local coordinates $(x^a, f^A, {v_x}^a, {v_f}^A)$ on $\T Q$, the complete lift $P$ takes the specific form:
\be \label{Eq: Complete Lift Coordinates}
P \,=\, P_f^A \otimes \frac{\partial}{\partial f^A} + P_v^A \otimes \frac{\partial}{\partial {v_f}^A} \,,
\ee
where the projection 1-forms are given by:
\begin{align}
P_f^A &= \dd f^A - \undertilde{P}^A_a \dd x^a \,, \\
P_v^A &= \dd {v_f}^A - \undertilde{P}^A_a \dd {v_x}^a - \left( {v_x}^b \frac{\partial \undertilde{P}^A_a}{\partial x^b} + {v_f}^B \frac{\partial \undertilde{P}^A_a}{\partial f^B} \right) \dd x^a \,.
\end{align}
\end{definition}

\begin{remark} If $\undertilde{P}$ complements the integrable distribution $\undertilde{K}$, its complete lift $\undertilde{P}^C$ complements the complete lift of the distribution $\undertilde{K}^C$.
\end{remark}
{}

\begin{definition}[\textsc{The geometry of the tangent bundle}]
\label{def:geometry_of_tangent}
Given a smooth $d$-dimensional differential manifold $Q$, with local coordinates $\left\{\, q^j \,\right\}_{j=1,...,d}$, the geometry of its tangent bundle $\T Q$ is characterized by two canonical objects (see \cite{godbillon,dLR1989}):
\begin{itemize}
    \item The \textbf{vertical endomorphism} (or \textbf{soldering form}) $S$, which is a $(1,1)$-tensor field $S$ on $\T Q$ that locally, using the system of coordinates $\left\{\, q^j,\, v^j \,\right\}_{j=1,...,d}$ for $\T Q$, reads 
    \be 
    \label{eq:soldering_form}
    S = \dd q^j \otimes \frac{\partial}{\partial v^j} \,. \ee 
    %%%
    It defines the vertical distribution $\mathcal{V}(\T Q) = \operatorname{Im}(S)$. Identified as a map $S \colon \T Q \longrightarrow \T Q$, it can be intrinsically defined as the unique map satisfying
    \be
    S(X(v)) := ((\pi_ Q)_\ast X(v))^V\,,
    \ee
    for every $X \in \mathfrak{X}(\T Q)$.
    \item The \textbf{Liouville vector field} $\Delta$, which is the infinitesimal generator of dilations along the fibers. 
%%%    
Locally, 
\be
\Delta = v^j \frac{\partial}{\partial v^j} \,.
\ee
\end{itemize}
%%%
These tensors satisfy the following properties:
\begin{eqnarray} \label{Eq: cond tangent}
\mathrm{Im} S \,&=&\, \mathrm{ker} S \,, \label{Eq: cond tangent1} \\
N_S \,&=&\, 0 \,, \label{Eq: cond tangent2}\\
\Delta &\in& \mathrm{Im} S \,, \label{Eq: cond tangent3}\\
\mathcal{L}_\Delta S \,&=&\, -S \,,\label{Eq: cond tangent4}
\end{eqnarray}
where $N_S \,=\, [S,S]_{F-N}$ (where $[\,\cdot\,,\,\cdot\,]_{F-N}$ denotes the Fr\"olicher--Nijenhuis bracket \cite{kobayashi}) is the Nijenhuis tensor of $S$. We will denote by $\dd_S$ the operator on forms given by $\dd_S \alpha := \dd i_S \alpha + i_S \dd \alpha$.
%%%
\end{definition}
%%%

\noindent Properties \eqref{Eq: cond tangent1}-\eqref{Eq: cond tangent4} are not incidental; they uniquely characterize the tangent bundle structure (see \cref{thm:global_tangent_manifold}).

\noindent With the above discussion in mind, the following definition is natural.
%%%

\begin{definition}[\textsc{(Almost) tangent structure}] An {\bf almost tangent structure} on a manifold $M$ is a $(1,1)$-tensor field $S$ such that, when identified as an endomorphism $S \colon \T M \longrightarrow \T M$, it satisfies $\Im S = \ker S$. 
%%%
An almost tangent structure $S$ on $M$ is called a {\bf tangent structure} or {\bf involutive} if the endomorphism $S$ satisfies $N_S = [S,S]_{F-N}  = 0$.
\end{definition}

\begin{remark} 
\label{remark:Tangent_bundle_structure}
Let $Q$ be an arbitrary configuration manifold. In light of \cref{def:geometry_of_tangent}, we have that $M = \T Q$ admits a canonical tangent structure.
\end{remark}

\noindent In fact, any almost tangent structure on a manifold $M$, say $S \in \Gamma \left(\T^\ast M \otimes \T M \right)$ has the local expression of Eq. \eqref{eq:soldering_form} if and only if $[S, S]_{F-N} = 0$ (see {\cite{kobayashi}).
%%%
Whether a tangent structure $S \in \Gamma \left(\T^\ast M \otimes \T M \right)$ is isomorphic to the canonical tangent structure on $\T Q$ is characterized by the existence of a \emph{Liouville vector field} satisfying properties \cref{Eq: cond tangent3}-\eqref{Eq: cond tangent4}. 
%%% 

\noindent Indeed, we have:

\begin{theo}[\textsc{Characterization of tangent bundles}]\cite{nagano,thompson,DeFilippo-Landi-Marmo-Vilasi-TangentBundle-1989}
\label{thm:global_tangent_manifold}
Given a $2d$-dimensional manifold $M$, equipped with a $(1,1)$-tensor $S$ and a vector field $\Delta$ such that:
\begin{itemize}
    \item The vector field $\Delta$ is complete.
    \item The set of zeroes of $\Delta$, $Q := \{ m \in M \mid \Delta_m = 0 \}$, is a smooth $d$-dimensional embedded submanifold of $M$.
    \item The limit of the flow of $\Delta$, $\lim_{t \to -\infty} F_t^\Delta(m)$, exists for all $m \in M$ (and defines the projection onto $Q$).
    \item $S$ and $\Delta$ satisfy the relations in \cref{Eq: cond tangent1,Eq: cond tangent2,Eq: cond tangent3,Eq: cond tangent4}.
\end{itemize}
Then the manifold $M$ is diffeomorphic to the tangent bundle of $Q$, $M \cong \T Q$.
\end{theo}

 \subsection{Jet structures}
\label{subsection:jet_structures}
In order to deal with the regularization of time-dependent (or non-autonomous) singular Lagrangian systems, the geometry of the so-called \textit{jet bundle} of a fiber bundle over the real line $\mathbf{Q}\longrightarrow \mathbb{R}$ will take a primary role. The first jet of the manifold is defined similarly to the tangent bundle:
%%%
\begin{definition} Let $\pi\colon \mathbf{Q} \longrightarrow\mathbb{R}$ be a fiber bundle. As a set, its \textbf{first jet bundle} $J^1\pi$ is defined as the equivalence class of sections $\gamma \colon \mathbb{R}\longrightarrow \mathbf{Q}$ to first order. It can be naturally endowed with a smooth structure (see \cite{Saunders}).
\end{definition}
%%%
\begin{remark}[\textsc{Natural coordinates}] There are natural coordinates on $J^1 \pi$, for any given fibered coordinates $(t, q^i)$ on $\pi \colon \mathbf{Q} \longrightarrow \mathbb{R}$, which read as $( q^i, \dot{q}^i, t)$, representing the class of sections through $(q^i, t)$ with velocity $\pdv{t} + \dot{q}^i \pdv{q^i}$. 
\end{remark}
%%%
On $J^1 \pi$, the notion of vertical and complete lifts may be defined similarly to the case of $\T Q$, and take the same local expression as in \cref{Def: lifts vector fields} when the canonical coordinates $(q^j, \dot{q}^i, t)$ are chosen. In particular, one may define it using the following embedding:
\begin{remark}[\textsc{Embedding of jets into tangent bundle}] For a jet bundle $J^1 \pi$, of some (arbitrary) configuration bundle $\pi \colon \mathbf{Q} \longrightarrow \mathbb{R}$, we have a canonical embedding 
\[
i_{\pi} \colon J^1 \pi \hookrightarrow \T \mathbf{Q}\,.
\]
This embedding is defined using the global vector field $\pdv{t}$ on $\mathbb{R}$ as follows:
\[
i_{\pi}(j_t \gamma) := \gamma_\ast \left( \pdv{t}\right)\,.
\]
Defining natural coordinates $(t, q^i, \dot{t}, \dot{q}^i)$ on $ \T \mathbf{Q}$, it reads as $i_{\pi} ^\ast(t, q^i, \dot{t}, \dot{q}^i) = (t, q^i, 1, \dot{q}^i)$.
\end{remark}

Then, given a vertical vector field $X \in \mathfrak{X}(\mathbf{Q})$ with respect to the projection $\pi \colon \mathbf{Q} \longrightarrow \mathbb{R}$, we can define the vertical and complete lift simply by restricting the vertical and complete lift on $\T \mathbf{Q}$ to $J^1 \pi$ (which are tangent vectors). {On jet bundles one usually defines the so-called 1st order jet prolongation of a vector field on $\mathbf{Q}$. It can be defined for vertical vector fields like those we are considering here (and, indeed, in that case it is the same as what we are here calling complete lift), and, more generally, for projectable vector fields. 
%Maybe we can stress this in order to avoid comments from the referee.
}

\begin{remark}[\textsc{Local expressions of vertical and complete lifts}]
{Let $X = X^i(q, t) \pdv{q^i}$ be a vertical vector field on $\pi \colon \mathbf{Q} \longrightarrow \mathbb{R}$. Then the vertical and complete lift take the following expressions:
\[
X^v = X^i \pdv{\dot{q}^i} \qquad \text{and} \qquad X^{C} = X^i \pdv{q^i} + \left(\pdv{X^i}{t} + \pdv{X^i}{q^j} \dot{q}^j \right) \pdv{\dot{q}^i}\,.
\]}
\end{remark}

%%%
\begin{definition}[\textsc{The geometry of the jet bundle}]
\label{def:geometry_jet_bundle}
Let $\pi\colon\mathbf{Q} \longrightarrow \mathbb{R}$ denote a fiber bundle over $\mathbb{R}$, where the standard fiber has dimension $d$. The geometry of the first jet bundle is characterized by the following ingredients: 

\begin{itemize}
    \item A closed $1$-form, $\tau = \dd t$.
    \item The {\bf vertical endomorphism} ${S}$, which is a $(1,1)$-tensor field on $J^1 \pi$ that, employing the canonical set of coordinates $\{q^j, \dot{q}^j, t\}_{j = 1, \dots, d}$, reads as
    \be
    S = (\dd q^i - \dot{q}^i \dd t) \otimes \pdv{\dot{q}^i} 
    \ee
    and satisfies that $\Im S$ is an integrable distribution with
    \be
    {S}^2 =0 \,, \qquad  \rank  S = d\,, \qquad \dd t(\Im S) = 0\,, \qquad \text{and} \qquad [{S}, {S}]_{F-N} = 2\dd t \wedge S\,.
    \ee
    \item It is an affine bundle over $\mathbf{Q}$ modeled on the vector bundle $\ker \dd \pi \longrightarrow \mathbf{Q}$.
\end{itemize}
\end{definition}
%%%

\begin{remark}[\textsc{Trivialized bundles}] When the bundle is trivialized $\mathbf{Q} = Q \times\mathbb{R} \longrightarrow \mathbb{R}$, we have a canonical diffeomorphism $J^1 \pi\cong \T Q \times\mathbb{R}$. Since every fiber bundle over $\mathbb{R}$ is trivializable, it follows that $J^1 \pi \cong \T Q \times \mathbb{R}$, for some $Q$ (which is the standard fiber). However, this diffeomorphism depends on the trivialization, and breaks the jet geometry. The discussion that we present applies to the trivialized version as well, and has the advantage of being readily generalizable to more general variational problems, where the bundle may not be trivial.
\end{remark}
%%%
 When dealing with trivialized bundles, the geometry of $J^1 \pi$ is that of \emph{stable tangent structures} (see \cite{dLEO}). In this case, the vertical endomorphism canonically splits as 
 \[
 S = \overline{S} - \dd t \otimes \Delta\,,
 \]
 where $\Delta$ is the Liouville vector field inherited by the vector bundle structure on $\T Q$. In this case, one may introduce the following $(1,1)$-tensor
 \[
 \widetilde S= \overline{S} + \dd t \otimes \pdv{t}\,.
 \]
 These objects (and their relations) completely characterize stable tangent structures. Similarly to almost tangent structures, one can define an almost stable tangent structure on a manifold $M$ to be a collection of objects that satisfy point-wise the same properties that those canonical objects on $\T Q \times \mathbb{R}$ satisfy. For completeness, we collect the definitions and results here.

 \begin{definition}[\textsc{(Almost) stable tangent structure}] An {\bf almost stable tangent structure} on a $(2 d +1)$-dimensional manifold $M$ is a tuple $(\widetilde{S}, \tau, \xi)$ consisting of a $(1,1)$-tensor $\widetilde{S} \in \Gamma \left( \T^\ast M \otimes \T M\right)$, a $1$-form $\tau \in \Omega^1(M)$ and a vector field $\xi \in \mathfrak{X}(M)$ satisfying 
\be 
\widetilde{S}^2 = \tau \otimes \xi \,, \qquad \rank \widetilde{S} = d+1\,, \qquad \text{and} \qquad \tau(\xi) = 1\,.
\ee
If, in addition, $\dd \tau = 0$ and $[\widetilde{S}, \widetilde{S}]_{F-N} = 0$, we say that $(\widetilde{S}, \tau,\xi)$ is {\bf involutive} or that it defines a {\bf stable tangent structure}.
\end{definition}

Locally, every stable tangent structure looks like $\T Q \times \mathbb{R}$. To obtain a global isomorphism we need the linear structure, which is characterized by the existence of a Liouville vector field (as in the tangent case).

 \begin{theo}[\textsc{Characterization of stable tangent bundles}]\cite{dLEO}
\label{thm:global_stable_tangent_manifold}
Let $M$ be a $(2d +1)$-dimensional manifold equipped with a stable tangent structure $(\overline{S}, \tau, \xi)$ and a vector field $\Delta \in \mathfrak{X}(M)$ satisfying:
\begin{itemize}
    \item The vector field $\Delta$ is complete.
    \item The closed form $\tau $ is exact $\tau = \dd t$, for certain surjective function $t \colon M \longrightarrow \mathbb{R}$.
    \item The following set $Q := \{m \in M: \Delta|_m = 0 \quad \text{and} \quad t(m) = 0\}$ is a smooth $d$-dimensional embedded submanifold.
    \item Denoting by $F_t^\Delta$ the flow of $\Delta$, the limit $\lim_{t \to - \infty} F^\Delta_t(m)$ exists for all $m \in M$ and is a {surjective submersion onto $Q$}.
    \item $\Delta$ satisfies the following:
    \[
    \Lie_{\Delta} {\widetilde S} = - \overline{S} \,, \qquad \text{and} \quad {\widetilde S}(\Delta) = 0
    \]
\end{itemize}
Then, $M$ is diffeomorphic to a stable tangent bundle $\T Q \times \mathbb{R}$, for some configuration manifold $Q$.
\end{theo}

It is unknown to the authors if a similar characterization for \emph{jet structures} on a manifold $M$ exists. However, to deal with uniqueness of the Lagrangian regularization of non-autonomous systems, we need to introduce this notion abstractly. In principle, there may be a regularization which could be considered Lagrangian locally (due to the presence of a vertical endomorphism and $1$-form $\dd t$), but not considered Lagrangian globally, in the sense that it is not diffeomorphic to a jet bundle itself. The notion that we work with is the following:

\begin{definition}[\textsc{Almost jet structure}] Let $M$ be a $(2d + 1)$-dimensional manifold. An {\bf almost jet structure} on $M$ is a pair $(S, \xi)$, where $S$ is a $(1,1)$-tensor and $\xi$ is a nowhere zero $1$-form, that satisfy the following properties. 
\[
S^2 = 0\,, \qquad \operatorname{rank} S = d\,, \qquad \xi(\Im S) = 0\,.
\]
When the equalities $\dd \xi = 0$, $[S, S]_{F-N} = 2 \xi \wedge S$ hold and $\Im S$ is an integrable distribution, we call it a {\bf jet structure}.
\end{definition}

\section{Regularization of autonomous systems}
\label{Sec:Regularization of autonomous systems}

\subsection{Symplectic Hamiltonian systems}

{
\begin{definition}[\textsc{Symplectic manifold}]
\label{Def: Symplectic manifold}
A \textbf{symplectic manifold} is a pair $(M, \omega)$, where $M$ is a smooth manifold of even dimension $2n$ and $\omega \in \Omega^2(M)$ is a closed and non-degenerate differential 2-form, called the \textbf{symplectic form}.
\end{definition}
\begin{theo}[\textsc{Darboux's Theorem}]
\label{Thm: Darboux symplectic}
Let $(M, \omega)$ be a symplectic manifold of dimension $2n$. Around every point $m \in M$, there exist local coordinates $(q^1, \dots, q^n, p_1, \dots, p_n)$, called \textbf{Darboux coordinates}, such that the symplectic form locally reads:
$$
\omega = \dd q^i \wedge \dd p_i \,.
$$
\end{theo}
}

\begin{definition}[\textsc{Hamiltonian system}]
\label{Def: Hamiltonian system}
A \textbf{Hamiltonian system} is a triple $(M, \omega, H)$, where $(M, \omega)$ is a symplectic manifold (the \textit{phase space}) and $H \in C^\infty(M)$ is a smooth function (the \textit{Hamiltonian}).
\end{definition}

\begin{remark}[\textsc{Poisson manifolds}]
\label{Rem: Poisson manifold}\cite{dLR1989,libermann,narciso}
Some authors define a Hamiltonian system more generally as a pair $(M, \Lambda)$, where $M$ is a manifold and $\Lambda$ is a Poisson tensor (a bivector field whose Schouten-Nijenhuis bracket $[\Lambda, \Lambda]$ vanishes), together with a Hamiltonian function $H \in C^\infty(M)$.
%%%
Every symplectic manifold $(M, \omega)$ is a Poisson manifold, with the Poisson tensor $\Lambda$ being the bivector field $\omega^\sharp$ associated with $\omega$. The resulting Poisson bracket is given by $\{f, g\} = \Lambda(\dd f, \dd g) = \omega(X_f, X_g)$.
%%%
The converse, however, is not true, as a Poisson tensor may be degenerate (i.e., the map $\Lambda^\sharp \colon \T^\ast M \to \T M$ is not an isomorphism).
%%%
Throughout this paper, we will adhere to the definition given in \cref{Def: Hamiltonian system}, and the term Hamiltonian system will always refer to a system defined on a symplectic manifold.
\end{remark}
%%%

{
\noindent Since the $2$-form $\omega$ is non-degenerate, the musical morphism $\omega^\flat \colon \T M \to \T^\ast M$, given by $\omega^\flat(v) = i_v \, \omega$, actually defines a vector bundle isomorphism. 
%%%
This guarantees the existence of a unique vector field $X_H \in \mathfrak{X}(M)$, the \textit{Hamiltonian vector field}, satisfying the intrinsic \textit{Hamilton's equations}:
\be \label{Eq: Hamilton eq}
i_{X_H} \omega =\dd H \,.
\ee
%%%
In Darboux coordinates $(q^i, p_i)$, the vector field takes the local form:
\be
X_H \,=\, \frac{\partial H}{\partial p_i} \frac{\partial}{\partial q^i} - \frac{\partial H}{\partial q^i} \frac{\partial}{\partial p_i} \,,
\ee
and its integral curves $\gamma(t) = (q^i(t), p_i(t))$ are the solutions of the Hamiltonian system, satisfying the standard Hamilton's equations:
\be
\frac{\dd q^i}{\dd t} \,=\, \frac{\partial H}{\partial p_i} \,, \quad \frac{\dd p_i}{\dd t} \,=\, - \frac{\partial H}{\partial q^i} \,.
\ee
%%%
The non-degeneracy of $\omega$ ensures local existence and uniqueness of solutions for any given initial condition $m \in M$.}

\subsection{Coisotropic regularization of pre-symplectic Hamiltonian systems}
\label{Sec: Pre-symplectic Hamiltonian systems}

In general, when working with singular theories (such as gauge Hamiltonian theories and singular time-independent Lagrangian theories), we work on a pre-symplectic manifold, rather than on a symplectic one, say $(M, \omega)$.
%%%

\begin{definition}[\textsc{Pre-symplectic Hamiltonian system}]
\label{Def: Pre-symplectic Hamiltonian system}
A \textbf{pre-symplectic Hamiltonian system} is a triple $(M, \omega, H)$, where $(M, \omega)$ is a pre-symplectic manifold and $H \in C^\infty(M)$ is a Hamiltonian.
%%%
In this case, the characteristic distribution $\mathcal{V} := \ker \omega$ is non-trivial. 
%%%
The dynamics is still formally governed by the equation
\be \label{Eq: Pre-symplectic Hamilton eq}
i_X \omega =\dd H \,.
\ee
%%%
However, this equation poses two distinct problems:
\begin{enumerate}
    \item \textbf{Existence:} A vector field $X$ satisfying the equation may not exist.
    \item \textbf{Uniqueness:} If a solution $X$ exists, it is not unique (it is defined only up to the addition of any vector field $Y \in \mathcal{V}$).
\end{enumerate}
\end{definition}

\noindent These two problems identify two classes of pre-symplectic systems:

\paragraph{Inconsistent Hamiltonian Systems.}
A system is \textit{inconsistent} if the existence condition fails, i.e., $\dd H$ is not in the image of $\omega^\flat$. 
%%%
This happens at points $m$ where $\dd H_m$ does not annihilate the kernel: $\ker \omega \not\subseteq \ker \dd H$.
%%%
For these systems, one must first find the submanifold of $M$ where a consistent dynamical evolution exists. 
%%%
This is achieved by the \textit{pre-symplectic constraint algorithm (PCA)}, developed by \textit{M.J. Gotay}, \textit{J.M. Nester}, and \textit{G. Hinds} \cite{gotay0} (see also the papers \cite{gotay_nester1,gotay_nester2}).
%%%

\noindent The algorithm proceeds iteratively. 
%%%
We define $M_0 := M$ and define the first constraint manifold $M_1$ as the locus where $\dd H$ is compatible with $\omega$:
\be
M_1 := \{ m \in M_0 \mid (\dd H)_m(Y) = 0, \;\forall Y \in (T_m M_0)^{\perp_\omega} \} \,,
\ee
where $(T_m M_0)^{\perp_\omega} := \mathcal{V}_m = \ker \omega_m$.
%%%
Assuming $M_1$ is a smooth submanifold, the algorithm imposes solutions of \eqref{Eq: Pre-symplectic Hamilton eq} on $M_1$ to be tangent to $M_1$, which is a new consistency requirement. 
%%%
Eventually, at each step $k \ge 2$, one finds the submanifold $M_k \subset M_{k-1}$:
\be
M_k := \{ m \in M_{k-1} \mid (\dd H)_m(Y) = 0, \;\forall Y \in (T_m M_{k-1})^{\perp_\omega} \}\,.
\ee
where $(T_m M_{k-1})^{\perp_\omega} := \{ Y \in \T_m M \mid \omega_m(Y, Z) = 0, \; \forall Z \in \T_m M_{k-1} \}$.
%%%
{
Assuming that all of these subsets are actually smooth submanifolds, we have two main possibilities:
\begin{itemize}
    \item There is a certain $k$ for which $M_k = \varnothing$, in which case the dynamics are globally not well-defined.
    \item The algorithm stabilizes on a final constraint manifold $M_f \neq \varnothing$, meaning there exists $k \in \mathbb{N}$ such that $M_k = M_{k-1} =: M_f$.
\end{itemize}
%%%
If the algorithm stabilizes, then the pre-symplectic Hamiltonian system $(M_f, \omega_f, H_f)$ (where $\omega_f = \mathfrak{i}_f^\ast \omega$, $H_f = \mathfrak{i}_f^\ast H$) is, by construction, no longer inconsistent.
%%%
Equations
\be
i_{\Gamma_f} \omega_f \,=\, \dd H_f \,,
\ee
are now well-posed and the integral curves of $\Gamma_f$, embedded into the starting manifold $M$, are the solutions of the original pre-symplectic Hamiltonian system.
}

\begin{remark}
 The above (PCA) algorithm is a geometrization of the so-called Dirac-Bergmann constraint algorithm developed by both authors in an independent manner (see \cite{dirac1,dirac2,bergmann,bergmann2}). The reader can find more complete information in \textit{P.A.M. Dirac}'s monograph \cite{dirac} as well as in these two papers by \textit{M.J. Gotay} and \textit{J.M. Nester} \cite{gotay_nester1,gotay_nester2}.
\end{remark}

\paragraph{Consistent Hamiltonian Systems.}
A system is consistent if it admits a global dynamics, i.e., $\ker \omega \subseteq \ker \dd H$. 
%%%
This corresponds to a system that either started consistent (like a pure gauge theory) or is the \textit{result} $(M_f, \omega_f, H_f)$ of applying the PCA.
%%%
In this case, the existence problem is solved, but the uniqueness problem (gauge ambiguity) remains, as the dynamics $X$ is only defined up to $X \mapsto X+Y$ for $Y \in \mathrm{ker}\,\omega_f$.
%%%

\noindent To solve this remaining ambiguity, one can regularize the system using the coisotropic embedding theorem.

\begin{theo}[\textsc{The coisotropic embedding theorem}]
\label{Thm: Coisotropic Embedding Theorem}
Let $(M,\,\omega)$ be a pre-symplectic manifold with characteristic distribution $\mathcal{V} = \ker \omega$.
%%%
There exists a symplectic manifold $(\widetilde{M},\,\widetilde{\omega})$ and an embedding
\be
\mathfrak{i} \colon  M \hookrightarrow \widetilde{M} \,,
\ee
such that $\mathfrak{i}^\ast \widetilde{\omega} = \omega$ and $\mathfrak{i}(M)$ is a closed \textbf{coisotropic submanifold} of $(\widetilde{M}, \widetilde{\omega})$.
%%%
The pair $(\widetilde{M},\,\widetilde{\omega})$ is called a \textbf{symplectic thickening} of $(M,\,\omega)$. 
%%%
Furthermore, this thickening is unique up to a neighborhood equivalence \cite{gotay1}.
\end{theo}

\begin{remark}[\textsc{Construction of the symplectic thickening}]
\label{Rem: Construction of thickening}
The construction of $\widetilde{M}$ (see \cite{luca,zoo}) requires choosing an almost product structure $P$ of the type \eqref{Def: Almost Product Structure}. 
%%%
Locally, using Darboux coordinates $(q^a, p_a, f^A)$ such that $\omega = \dd q^a \wedge \dd p_a$ and $\mathcal{V} = \operatorname{span}\{\frac{\partial}{\partial f^A}\}$, the projector $P$ (which defines $\mathcal{H} = \ker P$) is given by
\be
P = P^A \otimes \frac{\partial}{\partial f^A} = \left(\dd f^A - {P_q}^A_a \dd q^a - {P_p}^{Aa} \dd p_a\right) \otimes \frac{\partial}{\partial f^A} \,.
\ee
%%%
The thickening $\widetilde{M}$ is a neighborhood of the zero section in the dual bundle $\mathcal{V}^\ast$, with coordinates $(q^a, p_a, f^A, {\mu}_A)$. 
%%%
The symplectic form is $\widetilde{\omega} := \tau^\ast \omega + \dd \vartheta^P$, where $\vartheta^P = {\mu}_A P^A$ is the tautological 1-form. 
%%%
In local coordinates:
\be
\widetilde{\omega} = \dd q^a \wedge \dd p_a + \dd {\mu}_A \wedge P^A - {\mu}_A \dd P^A \,.
\ee
%%%
It is symplectic only in a tubular neighborhood of the zero-section of $\tau$ (${\mu_A} \approx 0$, namely $\mu_A$ approaching zero), unless $P$ has a vanishing Nijenhuis tensor \cite{dLR1989}.
%%%
\end{remark}

\noindent The coisotropic embedding theorem provides the tool to regularize any consistent pre-symplectic Hamiltonian system.
%%%
If we start with an inconsistent system, we first apply the PCA to get the consistent system $(M_f, \omega_f, H_f)$. 
%%%
We then apply \cref{Thm: Coisotropic Embedding Theorem} to this $M_f$.
%%%

\noindent In both scenarios, the procedure is the same: we embed the consistent pre-symplectic manifold $(M, \omega)$ (which could be $M_0$ or $M_f$) into its symplectic thickening $(\widetilde{M}, \widetilde{\omega})$.
%%%
We extend the Hamiltonian $H$ to $\widetilde{H} \in C^\infty(\widetilde{M})$ as $\widetilde{H} = \tau^\ast H$. 
%%%
The new system $(\widetilde{M}, \widetilde{\omega}, \widetilde{H})$ is regular (symplectic), and its unique Hamiltonian vector field $X_{\widetilde{H}}$ is easily shown to be tangent to $M$, providing a (gauge-fixed) unique dynamical evolution for the original system.
%%%
Indeed, let us introduce the adapted basis of vector fields on $\widetilde{M}$:
\be
\left\{\, H_a \,=\, \frac{\partial}{\partial x^a} + P^A_a \frac{\partial}{\partial f^A} \,,\quad V_A \,=\, \frac{\partial}{\partial f^A} \,,\quad W^A \,=\, \frac{\partial}{\partial \mu_A} \,\right\} \,,
\ee
and its dual coframe of 1-forms:
\be
\left\{\, \dd x^a \,,\quad P^A \,=\, \dd f^A - P^A_a \dd x^a \,,\quad \dd \mu_A \,\right\} \,.
\ee
%%%
Notice that the characteristic distribution of $\omega$ is locally spanned by $V_A$, meaning $\ker \omega = \operatorname{span}\{V_A\}$, and the original pre-symplectic form locally reads $\omega = \frac{1}{2}\omega_{ab} \dd x^a \wedge \dd x^b$.
%%%
Since the extended Hamiltonian $\widetilde{H} = \tau^* H$ depends only on the base coordinates $(x^a, f^A)$, its exterior derivative can be naturally expanded in the dual coframe as:
\be
\dd \widetilde{H} \,=\, \frac{\partial H}{\partial x^a} \dd x^a + \frac{\partial H}{\partial f^A} \dd f^A \,=\, H_a(H) \dd x^a + V_A(H) P^A \,.
\ee
%%%

\noindent On the other hand, the thickened symplectic form is defined as $\widetilde{\omega} = \pi^*\omega + \dd(\mu_A 
{P}^A)$.
%%%
As discussed in \cref{Thm: Coisotropic Embedding Theorem} it is symplectic only in a tubular neighborhood of the zero-section of $\tau$ (where $\mu_A \approx 0$), and it reads
\be
\widetilde{\omega} \,=\, \frac{1}{2}\omega_{ab} \dd x^a \wedge \dd x^b + \dd \mu_A \wedge P^A \,.
\ee
%%%

\noindent The Hamiltonian vector field in the adapted basis considered reads $X_{\widetilde{H}} \,=\, X^a H_a + X^A V_A + X_{\mu_A} W^A$ (note that $\mu_A$ is one index, not to be confused with an expression of the sort $\mu A$). 
%%%
Computing its interior product with $\widetilde{\omega}$ gives:
\be
i_{X_{\widetilde{H}}} \widetilde{\omega} \,=\, X^a \omega_{ab} \dd x^b + X_{\mu_A} P^A - X^A \dd \mu_A \,.
\ee
%%%
Imposing the Hamiltonian condition $i_{X_{\widetilde{H}}} \widetilde{\omega} \,=\, \dd \widetilde{H}$ and matching the coefficients of the linearly independent 1-forms, we obtain the system:
\begin{align}
    X^a \omega_{ab} \,&=\, H_b(H) \,, \label{Eq: Horizontal dynamics} \\
    -X^A \,&=\, 0 \;\;\implies\;\; X^A \,=\, 0 \,, \label{Eq: Vertical gauge fixing} \\
    X_{\mu_A} \,&=\, V_A(H) \,=\,  \frac{\partial H}{\partial f^A} \,. \label{Eq: Transverse component}
\end{align}
%%%
Since the original pre-symplectic Hamiltonian system is consistent by hypothesis, the original Hamiltonian $H$ must annihilate the kernel of $\omega$, meaning $\dd H(V_A) = 0$, which translates to $\frac{\partial H}{\partial f^A} = 0$.
%%%
Substituting this into \cref{Eq: Transverse component}, we find that $X_{\mu_A} = 0$ everywhere on $M$.
%%%
Since the transversal components $X_{\mu_A}$ along the directions $W^A$ vanish on the zero section, the dynamical vector field $X_{\widetilde{H}}$ does not point outside the submanifold, meaning it is strictly tangent to $M$.
%%%

\subsection{Coisotropic regularization of degenerate Lagrangian systems}
\label{Sec: Degenerate Lagrangian systems}

We now shift our focus to the Lagrangian formalism. 
%%%
Using the intrinsic geometry of the tangent bundle, we recall how the dynamics of any Lagrangian system (regular or degenerate) can be formulated in a "Hamiltonian-like" manner on the velocity phase space $\T Q$.
%%%
The presence of the tangent bundle structure $(S, \Delta)$ requires slightly modifying the coisotropic regularization scheme presented for Hamiltonian systems.
%%%

\noindent First, using the geometry of the tangent bundle described in \cref{Subsec: Tangent structures}, let us recall the geometric definition of a second-order differential equation and let us define the geometric structures associated with any Lagrangian $L \in C^\infty(\T Q)$. 
%%%

\begin{definition}[\textsc{Second Order Differential Equation (SODE)}]
\label{Def: SODE}
A vector field $X \in \mathfrak{X}(\T Q)$ is a \textbf{Second Order Differential Equation (SODE)} field if its integral curves $\gamma(t) = (q^j(t), v^j(t))$ correctly relate the position and velocity coordinates, i.e., they satisfy the kinematic condition $\frac{\dd q^j}{\dd t} = v^j$.
%%%
\noindent In local coordinates $(q^j, v^j)$, a general vector field $X$ reads 
\be
X = A^j(q, v) \frac{\partial}{\partial q^j} + B^j(q, v) \frac{\partial}{\partial v^j} \,.
\ee
%%%
Its integral curves satisfy $\dot{q}^j = A^j$.
%%%
Therefore, for a SODE we must have $A^j = v^j$.
%%%

\noindent Intrinsically, this condition is expressed by
\be \label{Eq: SODE condition}
S(X) = \Delta \,.
\ee
%%%
\end{definition}
%%%

\begin{definition}[\textsc{Regular Lagrangian system}]
\label{Def: Regular Lagrangian}
A \textbf{Lagrangian system} is a pair $(Q, L)$. 
%%%
We define:
\begin{itemize}
    \item The \textbf{Poincaré-Cartan 1-form} $\theta_L := \dd_S L = S^*(\dd L)$, locally reading
    \be
    \theta_L \,=\, \frac{\partial L}{\partial v^j} \dd q^j \,.
    \ee
    %%%
    \item The \textbf{Lagrangian 2-form} $\omega_L := -\dd \theta_L$, locally reading 
    \be
    \omega_L \,=\, \frac{\partial^2 L}{\partial v^j \partial v^k} \dd q^j \wedge \dd v^k - \frac{\partial^2 L}{\partial q^k \partial v^j} \dd q^k \wedge \dd q^j \,.
    \ee
    %%%
    \item The \textbf{Lagrangian energy} $E_L := \Delta(L) - L$.
    %%%
\end{itemize}
%%%
A system is \textbf{regular} if $\omega_L$ is symplectic (i.e., its Hessian matrix $\left(W_{ij} = \frac{\partial^2 L}{\partial v^i \partial v^j}\right)$ is non-singular).
\end{definition}
%%%

\noindent The following theorem gives sufficient (which are, trivially, also necessary) conditions for a $2$-form $\omega$ on $\T Q$ to be Lagrangian.
%%%

\begin{theo}[\textsc{Characterization of Lagrangian 2-forms}]
\label{Thm: Helmholtz Conditions}
A 2-form $\omega$ on $\T Q$ is a (local) Lagrangian 2-form, i.e., $\omega = \omega_L = -\dd \dd_S L$ for some $L \in C^\infty(\T Q)$, if and only if it satisfies the following conditions:
\begin{align}
    \dd \omega \,=&\, 0 \,, \label{Eq: closure} \\
    \omega(SX,\,Y) \,=&\, \omega(SY,\,X) \,, \quad \forall X, Y \in \mathfrak{X}(\T Q) \,. \label{Eq: symmetry}
\end{align}
\begin{proof}
$(\implies)$ Condition \eqref{Eq: closure} implies that $\omega$ can be locally written as
\be
\omega \,=\, \dd (A_j \dd q^j + B_j \dd v^j) \,.
\ee
%%%
Condition \eqref{Eq: symmetry} for any pair of the type $X \,=\, \frac{\de}{\de q^j}$, $Y \,=\, \frac{\de}{\de v^k}$ gives
\be
\omega\left(\frac{\de}{\de v^j},\,\frac{\de}{\de v^k}\right) \,=\, 0 \,,\;\;\; \forall\,\, j,k=1,...,n \,,
\ee
namely that
\be
\frac{\de B_{[j}}{\de v^{k]}} \,=\, 0 \,,\forall\,\, j,k=1,...,n
\ee
(square brackets denoting skew-symmetrization) i.e., 
\be
B_j \,=\, \frac{\de B}{\de v^j}
\ee
for some $B \in \mathcal{C}^\infty(\T Q)$.
%%%
Thus, $\omega$ locally reads
\be
\omega \,=\, \dd \left(A_j \dd q^j + \frac{\de B}{\de v^j} \dd v^j \right) \,=\, \dd \left( A_j \dd q^j + \dd B - \frac{\de B}{\de q^j} \dd q^j \right) \,=\, \dd (C_j \dd q^j) \,,
\ee
for
\be
C_j \,=\, A_j - \frac{\de B}{\de q^j} \,.
\ee
%%%
On the other hand, condition \eqref{Eq: symmetry} for any pair of the type $X \,=\, \frac{\de}{\de q^j}$, $Y \,=\, \frac{\de}{\de q^k}$, gives
\be
\omega\left(\frac{\de}{\de v^j},\, \frac{\de}{\de q^k}\right) \,=\, \omega\left(\frac{\de}{\de v^k},\, \frac{\de}{\de q^j}\right) \,.
\ee
%%%
It is easy to see that this latter condition implies that
\be
\frac{\de C{[j}}{\de v^{k]}} \,=\, 0 \,,\forall\,\, j,k=1,...,n \,,
\ee
i.e.,
\be
C_j \,=\, \frac{\de L}{\de v^j} \,,
\ee
for some $L \in \mathcal{C}^\infty(\T Q)$.
%%%
This proves that there exists a local function $L$ on $\T Q$ such that $\omega \,=\, \dd \dd_S L$.
%%%

$(\impliedby)$ The necessity of condition \eqref{Eq: closure}-\eqref{Eq: symmetry} is trivial and follows from a straightforward computation. 
%%%
\end{proof}
\end{theo}
%%%

\noindent The solutions of a regular system are the integral curves of the unique \textit{SODE} (Second Order Differential Equation) field $X_L$ (namely, satisfying $S(X_L) = \Delta$), which is the unique solution to the intrinsic Euler-Lagrange equation
\be \label{Eq: Geometric EL eq}
i_{X_L} \omega_L = \dd E_L \,.
\ee
%%%

\begin{definition}[\textsc{Degenerate Lagrangian systems}]
\label{Def: Degenerate Lagrangian}
A Lagrangian system $(Q, L)$ is \textbf{degenerate} if $\omega_L$ is degenerate (pre-symplectic).
%%%
A degenerate Lagrangian system is precisely a \textbf{pre-symplectic Hamiltonian system} $(\T Q, \omega_L, E_L)$, but one which carries the "extra" kinematic constraint that its physical dynamics must be a SODE field.
%%%
\end{definition}

\noindent As in the Hamiltonian case, this pre-symplectic system $(\T Q, \omega_L, E_L)$ can be either inconsistent or consistent.
%%%

\paragraph{Inconsistent Lagrangian Systems.}
In this case $\ker \omega_L \not\subseteq \ker \dd E_L$, and, thus, the equation $i_X \omega_L = \dd E_L$ is not solvable on the entire $\T Q$, meaning no global dynamical field $X$ exists.
%%%

\noindent To find the (sub)manifold where consistent solutions exist, one must apply a constraint algorithm. 
%%%
Historically, this problem was tackled by the celebrated \textit{Dirac-Bergmann algorithm} \cite{dirac}. 
%%%
This term often encompasses two related but distinct procedures.
%%%

\textit{Dirac's Hamiltonian Algorithm}, developed by \textit{P.A.M. Dirac} and \textit{P.G. Bergmann} \cite{dirac}, follows the following steps:
\begin{itemize}
    \item The Legendre map $\mathcal{F}L \colon \T Q \to \T^\ast Q$, defined by the fiber derivative 
    \be
    \langle \mathcal{F}L(v), w \rangle := \dv{s} L(v+sw)\Big|_{s=0} \,, 
    \ee
    and locally reading
    \be
    (q^j,\, v^j) \mapsto \left(\,q^j,\, p_j = \frac{\de L}{\de v^j}\,\right) \,,
    \ee
    is assumed to be \textit{almost-regular} (meaning the Hessian $W_{ij}$ has constant rank). 
    %%%
    Consequently, its image $M_1 := \mathcal{F}L(\T Q)$ is a submanifold of $\T^\ast Q$ defined by a set of \textit{primary constraints} $\Phi_a^{(1)}(q, p) \approx 0$ (where $\approx$ means that the equality should be fulfilled along solutions of the equations of motion).
    %%%
    \item A \textit{canonical Hamiltonian} $H_C \in C^\infty(M_1)$ is defined (as it can be shown that $E_L$ is constant on the fibers of $\mathcal{F}L$ \cite{gotay_nester1}). This $H_C$ is then extended to a Hamiltonian $H_E \in C^\infty(\T^\ast Q)$ on the ambient space.
    \item The \textit{total Hamiltonian} is defined as $H_T = H_E + \mu^a \Phi_a^{(1)}$, where $\mu^a$ are arbitrary functions (Lagrange multipliers).
    \item The algorithm imposes the consistency condition $\dot{\Phi}_a^{(1)} \approx \{ \Phi_a^{(1)}, H_T \}_{\omega_Q} \approx 0$. This procedure iteratively generates a set of \textit{secondary} (and tertiary, etc.) constraints, defining a final constraint manifold $M_f \subset \T^\ast Q$.
\end{itemize}

\begin{remark}[\textsc{Relation between Dirac's algorithm and the PCA}]
\label{Rem: Dirac vs PCA}
We can now clarify the relationship between Dirac's algorithm and the geometric PCA (as defined in \cref{Sec: Pre-symplectic Hamiltonian systems}). Dirac's algorithm is, in essence, the local coordinate version of the geometric PCA.
%%%
Indeed, as it is proven in \cite{gotay_nester1}, Dirac's $k$-ary constraints $\Phi^{(k)} \approx 0$ locally select the submanifold $M_k \subset \T^\ast Q$. Furthermore, the Hamiltonian vector field $X_{\Phi^{(k)}}$ (associated with a $k$-ary constraint via the symplectic structure $\omega_{\T^*Q}$ of $\T^* Q$) is tangent to the constraint submanifold $M_k$ and belongs to the space $(T M_k)^{\perp_{\omega_{\T^*Q}}}$ used in the PCA iteration.
%%%
On the other hand, it is also proven that any vector field $Y \in (T_m M_{k-1})^{\perp_{\omega_{\T^*Q}}}$ (the space used by the PCA) gives rise to a local constraint condition $(\dd H_E)(Y) \approx 0$, which is precisely how Dirac generates the next set of constraints.
%%%
Therefore, one can conclude that the constraint conditions imposed by Dirac are the local coordinate expressions of the geometric conditions that select the submanifolds $M_k$ of the PCA.
%%%

\noindent However, a conceptual difference between the two approaches exists.
%%%
Dirac's algorithm works on the whole Phase Space $\T^\ast Q$ along the constraint submanifolds $M_k$, and defines solutions (vector fields $X_T$) on the whole $\T^\ast Q$. 
%%%
In contrast, one could apply the PCA (as defined in \cref{Sec: Pre-symplectic Hamiltonian systems}) intrinsically, starting from the pre-symplectic manifold $(M_1, \omega_1 = \omega_{\T^*Q}|_{M_1})$ as the "ambient" space.
%%%
As noted in \cite{gotay_nester1}, these two procedures are equivalent and stabilize on the same final constraint manifold $M_f$.
%%%

\noindent In this respect, it is proven in \cite{gotay_nester1} that the intrinsic solution $X_f$ (found on $M_f$ and satisfying $i_{X_f} \omega_f = \dd H_f$) can be lifted to a solution $X_T$ on the ambient space $\T^\ast Q$ satisfying the equations for the total Hamiltonian, $i_{X_T} \omega_Q = \dd H_T$.
\end{remark}
%%%

On the other hand, \textit{Bergmann's Lagrangian Algorithm}, developed by \textit{P.G. Bergmann} \cite{bergmann,bergmann2}, operates entirely on the tangent bundle (the velocity space) $\T Q$.
%%%%
It follows the steps:
\begin{itemize}
    \item The Euler-Lagrange equations are $W_{ij} \ddot{q}^j + (\frac{\partial^2 L}{\partial v^i \partial q^j} v^j - \frac{\partial L}{\partial q^i}) = 0$, where $(W_{ij} = \frac{\partial^2 L}{\partial v^i \partial v^j})$ is the singular Hessian.
    \item Contracting with a vector $Y^i$ in the kernel of $W$ ($Y^i W_{ij} = 0$) annihilates the $\ddot{q}$ term, yielding the \textit{primary constraints} $\Phi_a^{(1)}(q, v) \approx Y_a^i (\frac{\partial^2 L}{\partial v^i \partial q^j} v^j - \frac{\partial L}{\partial q^i}) \approx 0$.
    \item The algorithm imposes consistency by differentiating these constraints, $\frac{\dd}{\dd t} \Phi_a^{(1)}(q, v) \approx 0$. This introduces $\ddot{q}$ terms, which are then replaced using the "evolutive" part of the E-L equations, generating \textit{secondary constraints} $\Phi_b^{(2)}(q, v) \approx 0$.
\end{itemize}
%%%

\begin{remark}[\textsc{Relation between Bergmann's algorithm and the PCA}]
In our intrinsic formulation (\cref{Def: Degenerate Lagrangian}), the degenerate system is the pre-symplectic system $(\T Q, \omega_L, E_L)$. It can be proven that Bergmann's algorithm is exactly the local application of the PCA to this system.
%%%
\noindent Indeed, the PCA begins by defining $M_1 = \{ m \in \T Q \mid (\dd E_L)_m(Y) = 0, \;\forall Y \in \ker(\omega_L)_m \}$.
%%%
\noindent Let us identify the kernel $K = \ker \omega_L$. It can be locally decomposed as $K = K_V \oplus K_H$, where $K_V = K \cap \operatorname{Im} S$ (the vertical kernel, related to $\ker W$) and $K_H$ is a horizontal complement (related to $\ker W$ in the $\partial/\partial q$ directions).
%%%
\noindent We test the PCA condition $(\dd E_L)(Y) = 0$ on both parts:
\begin{itemize}
    \item For the vertical kernel $Y_Z = Z^i \frac{\partial}{\partial v^i} \in K_V$ (where $W_{ij} Z^j = 0$):
    \be
        (\dd E_L)(Y_Z) = \left( \frac{\partial E_L}{\partial v^i} \right) Z^i = (v^k W_{ki}) Z^i = v^k (W_{ik} Z^i) = 0 \,.
    \ee
    This condition is satisfied identically because $Z \in \ker W$. The vertical kernel generates no constraints.
    \item For the horizontal kernel $Y_Y = Y^i \frac{\partial}{\partial q^i} \in K_H$ (where $W_{ij} Y^j = 0$):
    \be
        (\dd E_L)(Y_Y) = \left( \frac{\partial E_L}{\partial q^i} \right) Y^i = \left( v^j \frac{\partial^2 L}{\partial q^i \partial v^j} - \frac{\partial L}{\partial q^i} \right) Y^i \approx 0 \,.
    \ee
\end{itemize}
This last equation is \textit{precisely} the set of primary constraints $\Phi_a^{(1)}(q, v) \approx 0$ derived from Bergmann's Lagrangian algorithm.
%%%
\noindent Since the subsequent steps of both algorithms are defined by the same iterative tangency requirement, the two algorithms are equivalent.
\end{remark}

The PCA only checks for the existence of some vector field $X$ satisfying $i_X \omega_L = \dd E_L$.
%%%
\noindent It does not check if this $X$ is a SODE field (i.e., $S(X) = \Delta$).
%%%
A system can be Hamilton-consistent ($M_f \neq \emptyset$) but Lagrangian-inconsistent if none of the solutions $X$ on $M_f$ are SODEs.
%%%

\noindent To solve this, one should use a "SODE-compatible" PCA (see \cite{R_MunozLecanda-RomanRoy-LagrangianTheoryPresymplectic-1992} for a more in-depth treatment). At each step $k$, one defines the next manifold $F_k \subset F_{k-1}$ as the locus of points $m$ where there exists a vector $X_m \in \T_m(\T Q)$ that satisfies \textit{all three} conditions:
\begin{enumerate}
    \item \textit{Hamiltonian condition:} $(i_{X_m} \omega_L + \dd E_L)_m (Y) = 0$ for all $Y \in (T_m F_{k-1})^{\perp_{\omega_L}}$.
    \item \textit{SODE condition:} $S(X_m) = \Delta_m$.
    \item \textit{Tangency condition:} $X_m \in \T_m F_{k-1}$.
\end{enumerate}
If this algorithm converges, it finds a final constraint manifold $F_f$ where the Lagrangian system becomes consistent, and which is, in general, a subset of the Hamiltonian one, $F_f \subseteq M_f$.
%%%

\noindent The case of inconsistent Lagrangian systems forces us to first apply the Lagrangian constraint algorithm to find the physical manifold $F_f$.
%%%
In general, there is no reason to expect $F_f$ to be a tangent bundle itself.

%%%
At this stage, the system is $(F_f, (\omega_L)_f, (E_L)_f)$, which is a consistent pre-symplectic system (generically with gauge freedom $\ker(\omega_L)_f \neq \{0\}$), but it is no longer a Lagrangian system.
%%%
\noindent To regularize the remaining gauge freedom, one can apply the coisotropic embedding (\cref{Thm: Coisotropic Embedding Theorem}) to $(F_f, (\omega_L)_f)$.
%%%
\noindent The result is a symplectic manifold $(\widetilde{F_f}, \widetilde{\omega_f})$. This manifold is generic, non-Lagrangian, and has lost the original physical tangent structure. 
%%%
This path solves the constraint problem but "destroys" the Lagrangian structure.
%%%
\begin{example}[\textsc{Affine Lagrangians}] 
A particularly enlightening example of the problem presented above is that of affine Lagrangians. Indeed, let $\alpha \in \Omega^1(Q)$ be a $1$-form and $f \in C^\infty(Q)$ be a function. Let 
\[
L \colon \T Q \longrightarrow \mathbb{R}\,,   \qquad L(v) := \alpha(v) + f(\pi_Q(v))\,.
\]
Locally, if $\alpha =  \alpha_i \dd q^i$, we have $L = \alpha_i \dot{q}^i + f(q)$. The Poincaré--Cartan form in this scenario is 
\[
\omega_L = - \dd \left(\alpha_i \dd q^i\right) = - \dd \alpha\,.
\]
And the Lagrangian energy is $E_L = -f$. Then, the equations of motion are 
$
\iota_{X} \dd \alpha = \dd f\,,
$
so that it is actually a first order equation on $Q$, and may be regarded as a pre-symplectic system. Then, the constraint algorithm on $(\T Q, L)$ is $\pi_Q$-related to the algorithm on the pre-symplectic system $(Q, \dd \alpha, f)$ as follows:
% https://q.uiver.app/#q=WzAsOCxbMywxLCJRIl0sWzMsMCwiXFxUIFEiXSxbMiwxLCJNXzEiXSxbMiwwLCJGXzEiXSxbMCwwLCJGX2YiXSxbMCwxLCJNX2YiXSxbMSwwLCJcXGNkb3RzIl0sWzEsMSwiXFxjZG90cyJdLFs0LDVdLFszLDJdLFsxLDBdLFszLDEsIiIsMSx7InN0eWxlIjp7InRhaWwiOnsibmFtZSI6Imhvb2siLCJzaWRlIjoidG9wIn19fV0sWzIsMCwiIiwxLHsic3R5bGUiOnsidGFpbCI6eyJuYW1lIjoiaG9vayIsInNpZGUiOiJ0b3AifX19XSxbNiwzLCIiLDEseyJzdHlsZSI6eyJ0YWlsIjp7Im5hbWUiOiJob29rIiwic2lkZSI6InRvcCJ9fX1dLFs0LDYsIiIsMSx7InN0eWxlIjp7InRhaWwiOnsibmFtZSI6Imhvb2siLCJzaWRlIjoidG9wIn19fV0sWzcsMiwiIiwxLHsic3R5bGUiOnsidGFpbCI6eyJuYW1lIjoiaG9vayIsInNpZGUiOiJ0b3AifX19XSxbNSw3LCIiLDEseyJzdHlsZSI6eyJ0YWlsIjp7Im5hbWUiOiJob29rIiwic2lkZSI6InRvcCJ9fX1dXQ==
\[\begin{tikzcd}[cramped]
	{F_f} & \cdots & {F_1} & {\T Q} \\
	{M_f} & \cdots & {M_1} & Q
	\arrow[hook, from=1-1, to=1-2]
	\arrow[from=1-1, to=2-1]
	\arrow[hook, from=1-2, to=1-3]
	\arrow[hook, from=1-3, to=1-4]
	\arrow[from=1-3, to=2-3]
	\arrow[from=1-4, to=2-4]
	\arrow[hook, from=2-1, to=2-2]
	\arrow[hook, from=2-2, to=2-3]
	\arrow[hook, from=2-3, to=2-4]
\end{tikzcd}\,,\]
and we have the equality 
\[
F_k = \{v \in \T_p M_{k-1} \colon\, p \in M_k \quad \text{and} \quad \iota_v \dd \alpha = \dd f\}\,.
\]
In particular, $(M_f, \dd \alpha, \dd f)$ is a consistent pre-symplectic Hamiltonian system and $F_f \longrightarrow M_f$ is an affine bundle modeled over ${\undertilde{K}} = \ker \dd \alpha$. Then, the characteristic distribution on $F_f$ is $\undertilde{K}^{C} |_{F_f}$, and the thickening is (an open subset of) 
\[
\widetilde{F}_f = \left( \undertilde{K}^{C}\right)^\ast |_{F_f} \longrightarrow F_f\,.
\]
In particular, by taking adapted coordinates $(x^a, f^A)$ to $\undertilde{K}$, the coordinates on $\widetilde{F}_f$ are $(x^a, f^A, \dot{f}^A, \mu_A, \dot{\mu}_A )$, and there is no reason to expect $\widetilde F_f$ to be a tangent bundle.
\end{example}

\begin{remark}[\textsc{Gotay and Nester's Lagrangian constraint algorithm}]
\label{remark:Gotay_constraint_algorithm}
There are different ways of performing the constraint algorithm to ensure that the final constraint not only has well-defined (tangent) dynamics, but the dynamics can be chosen to satisfy the SODE condition. To our knowledge, the most standard algorithm is the Lagrangian constraint algorithm by Gotay and Nester \cite{gotay_nester2}. The algorithm proceeds as follows. Let $L \in C^\infty(\T Q)$ be a singular Lagrangian. Then, instead of requiring the SODE condition at each step, one follows the Hamiltonian version of the algorithm, namely, by setting $M_0 = \T Q$ and then
\[
M_k := \{ m \in M_{k-1} \mid (\dd H)_m(Y) = 0, \;\forall Y \in (T_m M_{k-1})^{\perp_\omega} \}\,.\]
Suppose that the sequence $M_0, M_1, \dots$ stabilizes at $M_f$. Then, there is a vector field $X \in \mathfrak{X}(M_f)$ such that 
\[
i_{X} \omega_f = \dd E_f\,,
\]
where $\omega_f = (\omega_L) |_{M_f}$ and $E_f := (E_L)|_{M_f}$. However, the vector field $X$ may not satisfy the SODE condition. The idea by Gotay and Nester is to find a submanifold of $M_f$, in which $X$ can be chosen to satisfy it (we refer the reader to  \cite{R_Batlle-Gomis-Pons-RomanRoy-Equivalence-1986, R_Kamimura-SingularLagrangians-1982} for a complete proof). The construction employs the Legendre transformation, and that the algorithms (in the Lagrangian and Hamiltonian side) are conveniently related by it. Indeed, when $L$ is almost regular, by defining 
\[
P_0 := \operatorname{Leg}_L\left( \T Q\right) \subset \T^\ast Q\,,
\]
one has a fibration $\operatorname{Leg}_L \colon \T Q \longrightarrow P_0$. One can show that the Energy $E_L$ is constant along this fibers, so that there is a well defined Hamiltonian
% https://q.uiver.app/#q=WzAsMyxbMCwwLCJcXFQgUSJdLFswLDEsIlBfMCJdLFsxLDEsIlxcbWF0aGJie1J9Il0sWzAsMiwiRV9MIl0sWzAsMSwiXFxvcGVyYXRvcm5hbWV7TGVnfV9MIiwyXSxbMSwyLCJIXzAiLDJdXQ==
\[\begin{tikzcd}[cramped]
	{\T Q} & \\
	{P_0} & {\mathbb{R}}
	\arrow["{\operatorname{Leg}_L}"', from=1-1, to=2-1]
	\arrow["{E_L}", from=1-1, to=2-2]
	\arrow["{H_0}"', from=2-1, to=2-2]
\end{tikzcd}\,.\]
Then, by denoting $P_0, P_1, \dots$ the submanifolds obtained by the pre-symplectic constraint algorithm applied to $(P_0, (\omega_Q)|_{P_0}, H_0)$, one has that they are related by the Legendre transformation:
% https://q.uiver.app/#q=WzAsOSxbMywxLCJQXzAiXSxbMiwxLCJQXzEiXSxbMCwxLCJQX2YiXSxbNCwxLCJcXFReXFxhc3QgTSJdLFs0LDAsIlxcVCBNIl0sWzMsMCwiTV8xIl0sWzEsMCwiTV9mIl0sWzIsMCwiXFxjZG90cyJdLFsxLDEsIlxcY2RvdHMiXSxbNCwwLCJcXG9wZXJhdG9ybmFtZXtMZWd9X0wiLDJdLFs1LDEsIlxcb3BlcmF0b3JuYW1le0xlZ31fTCIsMl0sWzYsMiwiXFxvcGVyYXRvcm5hbWV7TGVnfV9MIiwyXSxbNSw0LCIiLDEseyJzdHlsZSI6eyJ0YWlsIjp7Im5hbWUiOiJob29rIiwic2lkZSI6InRvcCJ9fX1dLFs2LDcsIiIsMCx7InN0eWxlIjp7InRhaWwiOnsibmFtZSI6Imhvb2siLCJzaWRlIjoidG9wIn19fV0sWzcsNSwiIiwwLHsic3R5bGUiOnsidGFpbCI6eyJuYW1lIjoiaG9vayIsInNpZGUiOiJ0b3AifX19XSxbMiw4LCIiLDIseyJzdHlsZSI6eyJ0YWlsIjp7Im5hbWUiOiJob29rIiwic2lkZSI6InRvcCJ9fX1dLFs4LDEsIiIsMSx7InN0eWxlIjp7InRhaWwiOnsibmFtZSI6Imhvb2siLCJzaWRlIjoidG9wIn19fV0sWzEsMCwiIiwxLHsic3R5bGUiOnsidGFpbCI6eyJuYW1lIjoiaG9vayIsInNpZGUiOiJ0b3AifX19XSxbMCwzLCIiLDEseyJzdHlsZSI6eyJ0YWlsIjp7Im5hbWUiOiJob29rIiwic2lkZSI6InRvcCJ9fX1dXQ==
\[\begin{tikzcd}[cramped]
	& {M_f} & \cdots & {M_1} & {\T Q} \\
	{P_f} & \cdots & {P_1} & {P_0} & {\T^\ast Q}
	\arrow[hook, from=1-2, to=1-3]
	\arrow["{\operatorname{Leg}_L}"', from=1-2, to=2-1]
	\arrow[hook, from=1-3, to=1-4]
	\arrow[hook, from=1-4, to=1-5]
	\arrow["{\operatorname{Leg}_L}"', from=1-4, to=2-3]
	\arrow["{\operatorname{Leg}_L}"', from=1-5, to=2-4]
	\arrow[hook, from=2-1, to=2-2]
	\arrow[hook, from=2-2, to=2-3]
	\arrow[hook, from=2-3, to=2-4]
	\arrow[hook, from=2-4, to=2-5]
\end{tikzcd}\,.\]
Under mild regularity conditions, $\operatorname{Leg}_L \colon M_f \longrightarrow P_f$ defines a fiber bundle. Gotay and Nester then solve the SODE problem as follows: Choose a vector field $X \in \mathfrak{X}(P_f)$ solving $i_X (\omega_L)_{P_f} = \dd (E_L)_{P_f}$, (where $(\omega_L)_{P_f}$ reads the pull-back of $(\omega_Q)|_{P_0}$ to $P_f$ and $(E_L)_{P_f}$ is the pull-back of $E_L$ to $P_f$) and $Y \in \mathfrak{X}(M_f)$ which is $\operatorname{Leg}_L$-related to $X$, namely $(\operatorname{Leg}_L)_\ast Y = X$. Then, $Y$ solves the equation $i_Y (\omega_L)_{M_f} = \dd (E_L)_{M_f}$ (where $(\omega_L)_{M_f}$ is the pull-back of $\omega_L$ to $M_f$ and $(E_L)_{M_f}$ is the pull-back of $E_L$ to $M_f$), but, as discussed, does not necessarily satisfy the SODE condition $S(Y) = \Delta$. Hence, one studies the defect 
\[Y^\ast = S(Y) - \Delta\,.\]
Let us show that $\lim_{t \to \infty} \psi^{ Y^\ast}_{t}$ (where $\psi^{Y^\ast}_t$ denotes the local flow of $Y^\ast$) exists and defines a submanifold $S_f \hookrightarrow M_f$ which is diffeomorphic to $P_f$ through the Legendre transformation. Indeed, if 
$
Y = a^i \pdv{q^i} + b^i \pdv{\dot{q}^i}\,,
$
we have 
\[
Y^\ast = (a^i- \dot{q}^i)\pdv{\dot{q}^i}\,.
\]
Because $Y$ is $\operatorname{Leg}_L$-projectable, $a^i$ is constant on the fibers of $M_f \longrightarrow P_f$ and, furthermore, since we have $i_{S(Y)} \omega_L = i_{\Delta} \omega_L$ for every vector field $Y$, in particular we have $Y^\ast \in \ker \omega_L$. This, together with the fact that $\ker \dd \operatorname{Leg}_L = \ker \dd \pi_Q \cap \ker \omega_L$, implies that $Y^\ast$ is tangent to the fibers. Now it is clear that the integral curve of $Y^\ast$ through the point $(q_0^i, \dot{q}_0^i)$ is 
\[
\gamma(t) = (q_0^i, a^i + e^{-t}(\dot{q}_0^i - a^i))\,,
\]
so that its limit exists, and is the point with coordinates $(q^i_0, a^i)$. In particular, $S_f$ is the image of a section $\sigma \colon P_f \longrightarrow M_f$
% https://q.uiver.app/#q=WzAsMTAsWzMsMSwiUF8wIl0sWzIsMSwiUF8xIl0sWzAsMSwiUF9mIl0sWzQsMSwiXFxUXlxcYXN0IE0iXSxbNCwwLCJcXFQgTSJdLFszLDAsIk1fMSJdLFsxLDAsIk1fZiJdLFsyLDAsIlxcY2RvdHMiXSxbMSwxLCJcXGNkb3RzIl0sWzAsMCwiU19mIl0sWzQsMF0sWzUsMV0sWzYsMl0sWzUsNCwiIiwxLHsic3R5bGUiOnsidGFpbCI6eyJuYW1lIjoiaG9vayIsInNpZGUiOiJ0b3AifX19XSxbNiw3LCIiLDAseyJzdHlsZSI6eyJ0YWlsIjp7Im5hbWUiOiJob29rIiwic2lkZSI6InRvcCJ9fX1dLFs3LDUsIiIsMCx7InN0eWxlIjp7InRhaWwiOnsibmFtZSI6Imhvb2siLCJzaWRlIjoidG9wIn19fV0sWzIsOCwiIiwyLHsic3R5bGUiOnsidGFpbCI6eyJuYW1lIjoiaG9vayIsInNpZGUiOiJ0b3AifX19XSxbOCwxLCIiLDEseyJzdHlsZSI6eyJ0YWlsIjp7Im5hbWUiOiJob29rIiwic2lkZSI6InRvcCJ9fX1dLFsxLDAsIiIsMSx7InN0eWxlIjp7InRhaWwiOnsibmFtZSI6Imhvb2siLCJzaWRlIjoidG9wIn19fV0sWzAsMywiIiwxLHsic3R5bGUiOnsidGFpbCI6eyJuYW1lIjoiaG9vayIsInNpZGUiOiJ0b3AifX19XSxbOSw2LCIiLDAseyJzdHlsZSI6eyJ0YWlsIjp7Im5hbWUiOiJob29rIiwic2lkZSI6InRvcCJ9fX1dLFs5LDJdLFsyLDksIlxcc2lnbWEiLDAseyJvZmZzZXQiOi0zLCJjdXJ2ZSI6LTIsInN0eWxlIjp7ImJvZHkiOnsibmFtZSI6ImRhc2hlZCJ9fX1dXQ==
\[\begin{tikzcd}[cramped]
	{S_f} & {M_f} & \cdots & {M_1} & {\T Q} \\
	{P_f} & \cdots & {P_1} & {P_0} & {\T^\ast Q}
	\arrow[hook, from=1-1, to=1-2]
	\arrow[from=1-1, to=2-1]
	\arrow[hook, from=1-2, to=1-3]
	\arrow[from=1-2, to=2-1]
	\arrow[hook, from=1-3, to=1-4]
	\arrow[hook, from=1-4, to=1-5]
	\arrow[from=1-4, to=2-3]
	\arrow[from=1-5, to=2-4]
	\arrow["\sigma", shift left=3, curve={height=-12pt}, dashed, from=2-1, to=1-1]
	\arrow[hook, from=2-1, to=2-2]
	\arrow[hook, from=2-2, to=2-3]
	\arrow[hook, from=2-3, to=2-4]
	\arrow[hook, from=2-4, to=2-5]
\end{tikzcd}\,.\]
By definition, $Y^\ast$ vanishes on $S_f$, so that $Y$ solves the equation and satisfies the SODE condition $S(Y) = \Delta$.
However $Y$ need not be tangent to $S_f$, but it is enough to consider $\widetilde Y := \sigma_\ast (X)$, which
\begin{itemize}
    \item Is tangent to $S_f$.
    \item Solves the equations by definition.
    \item Satisfies the SODE condition. Indeed, $Y$ satisfies it, and $Y - \widetilde Y$ is a vertical vector field (with respect to the projection $ \pi_Q \colon \T Q \longrightarrow Q$), since $\ker (\dd \operatorname{Leg}_L) = \ker \dd \pi \cap \ker \omega_L$.
\end{itemize}
\end{remark}

\begin{remark}[\textsc{Relation between both algorithms}] 
The algorithm presented is related to the one by Gotay and Nester as follows. Denote by $F_0, F_1, \dots$ the constraint submanifold obtained by requiring the SODE condition at each step. Then, it is clear that at each step $F_k \subseteq M_k$, so that $F_f \subseteq M_f$. Furthermore, $S_f \subset F_f$, as we clearly have $S_f \subseteq F_k$, for every $k$. Indeed, the inclusion $S_f \subseteq F_0$ is clear, and the subsequent ones are obtained iteratively by definition.
\end{remark}

\paragraph{Consistent Lagrangian Systems.}
If the system is consistent but has gauge ambiguities, we can bypass the constraint algorithm and apply the coisotropic embedding theorem directly to the initial, degenerate Lagrangian system $(\T Q, \omega_L)$.
%%%
This approach raises the true "tangent structure problem": does this procedure preserve the tangent structure? Specifically:
\begin{enumerate}
    \item Is the regularized symplectic manifold $\widetilde{M}$ diffeomorphic to a tangent bundle $\T \widetilde{Q}$?
    \item If so, is the new symplectic form $\widetilde{\omega}$ a \textit{regular Lagrangian 2-form} $\omega_{\widetilde{L}}$?
\end{enumerate}
In the {next section, we show that for a specific, physically relevant class of degeneracies, the answer to the first question is yes, while the answer to the second one is no, unless the coisotropic regularization scheme used in the Hamiltonian setting is slightly modified.}
%%%
{
\subsection{Existence and uniqueness of Lagrangian regularization}}

{The objective of this section is to prove the existence of an autonomous Lagrangian regularization, under specific conditions on the gauge ambiguities of $L$. We also discuss the matter of uniqueness. Although global uniqueness is not guaranteed, as a plethora of extended Lagrangians may be considered, we prove that any tangent structure on a particular symplectic regularization $\widetilde{M}$ must be ``isomorphic on $\T Q$'' to the one we build. Namely, the first-order germ of the extension is unique.
%%%
The main assumption that we will make to endow the regularization with a Lagrangian structure (as in \cite{ibort-marin,ibortcos}) is that the characteristic distribution $K \,=\, \ker \omega_L$ providing the characteristic bundle $\mathbf{K} \subset \T(\T Q)$ is the complete lift (tangent distribution) of an integrable distribution $\undertilde{K}$ on $Q$, which defines a regular foliation $\mathcal{F}_{\undertilde{K}}$.}
%%%
Let $Q$ have local coordinates $(x^a, f^A)$, where $x^a$ are coordinates on the leaves of $\mathcal{F}_{\undertilde{K}}$ and $f^A$ parameterize the fibers (the distribution $\undertilde{K}$). 
%%%
The tangent bundle $\T Q$ has coordinates $(x^a, f^A, \dot{x}^a, \dot{f}^A)$.
%%%
Under this hypothesis, the kernel of $\omega_L$ is $K = {\undertilde{K}^C}$, locally spanned by 
\be
\text{span}\left\{ \left(\frac{\de}{\de f^A}\right)^C \,=\, \frac{\partial}{\partial f^A},\, \left(\frac{\de}{\de f^A}\right)^V \,=\, \frac{\partial}{\partial {\dot{f}}^A}\right\} \,.
\ee
%%%

\noindent The symplectic thickening $(\widetilde{M}, \widetilde{\omega})$ is constructed as a neighborhood of the zero section in the dual bundle $\mathbf{K}^* \to \T Q$, as described in \cref{Thm: Coisotropic Embedding Theorem}.
%%%
As discussed in \cite{ibort-marin}, such thickening coincides with the whole $\mathbf{K}^*$ if the almost product structure $P$ can be chosen to have vanishing Nijenhuis tensor.
%%%
{We assume this is the case from now on.}
%%%

\noindent On the other hand, the thickened space $\widetilde{M} = \mathbf{K}^*$ can be identified as the cotangent bundle of the foliation $\mathcal{F}_K$ in the sense of the following proposition:
\begin{proposition}
The following canonical isomorphism exists:
\be \label{Eq: thickening=cotangent bundle foliation}
\widetilde{M} \,=\, \mathbf{K}^* \,\simeq\, \mathscr{T}^* \mathcal{F}_K \,:=\, \bigsqcup_{F \in \mathcal{F}_K} \T^* F \,,
\ee
where $F$ denotes a leaf of $\mathcal{F}_K$.
%%%
\begin{proof}
A point $p \in \mathbf{K}^*$ is, by definition, an element of the dual bundle to $\mathbf{K}$. 
%%%
It consists of a pair $(m, \alpha_m)$, where $m \in \T Q$ is the base point, and $\alpha_m \in K_m^*$ is a linear functional on the fiber $K_m = \ker(\omega_L)_m$. 
%%%
Thus, $\alpha_m: K_m \to \R$.
%%%

\noindent On the other hand, a point $q \in \mathscr{T}^* \mathcal{F}_K$ is, by its definition as a disjoint union, an element of the cotangent bundle of a leaf $F \in \mathcal{F}_K$. 
%%%
This point consists of a pair $(m, \beta_m)$, where $m \in F$ is the base point, and $\beta_m \in \T_m^* F$ is a linear functional on the tangent space to that leaf, $\T_m F$. 
%%%
Thus, $\beta_m: \T_m F \to \R$.
%%%

\noindent The foliation $\mathcal{F}_K$ is, by construction, the integral foliation of the distribution $\mathbf{K}$. 
%%%
This means that, at any point $m \in \T Q$, the tangent space to the unique leaf $F$ passing through $m$ is precisely the subspace $K_m$:
\be
\T_m F = K_m \,.
\ee
%%%%

\noindent Since the domain spaces $K_m$ and $\T_m F$ are the same vector space, their dual spaces $K_m^*$ and $\T_m^* F$ are also canonically identical.
%%%
Therefore, there is a natural, fiber-preserving isomorphism $\Phi: \mathbf{K}^* \to \mathscr{T}^* \mathcal{F}_K$ given by $\Phi(m, \alpha_m) = (m, \alpha_m)$, which simply re-interprets the covector $\alpha_m \in K_m^*$ as an element of $\T_m^* F$. 
%%%
This establishes the identity $\mathbf{K}^* \cong \mathscr{T}^* \mathcal{F}_K$.
%%%
\end{proof}
\end{proposition}

\begin{remark}
\label{remark:vanishing_Nijenhuis_tensor}
In the case where the almost product structure $P$ does not have vanishing Nijenhuis tensor, the identification of $\mathbf{K}^\ast$ as a global tangent manifold does not make sense anymore. Indeed, the coisotropic embedding theorem forces us to restrict to a tubular neighborhood of $\T Q$ in $\mathbf{K}^\ast$ for the form to be symplectic. This neighborhood may no longer be a vector bundle, thus destroying the \emph{global} tangent structure. However, it inherits a natural tangent structure, and may be considered a tangent bundle locally.
%%%
For convenience, we restrict to the case in which $\widetilde{\omega}$ is globally symplectic, although without much difficulty all the constructions extend to the situation where one must work on an open subset (indeed, the local expression and constructions that we show work for any almost product structure). 
%%%
Moreover, if the Hamiltonian vector fields are complete (in particular, if the Hamiltonian vector field in the thickening is complete), the dynamics always remain in this open subset, as the flow maintains regularity.
\end{remark}

\noindent The coordinates of $\widetilde{M}$ are $(x^a, f^A, \dot{x}^a, \dot{f}^A, {\mu_f}_A, {\mu_{\dot{f}}}_A)$, where $({\mu_f}_A, {\mu_{\dot{f}}}_A)$ are the fiber coordinates dual to the kernel generators
\be
\left\{\,\frac{\partial}{\partial f^A}, \frac{\partial}{\partial {\dot{f}}^A} \,\right\} \,.
\ee
%%%
We now define a new configuration manifold $\widetilde{Q}$. 
%%%
We identify $\widetilde{Q}$ as the cotangent bundle of the foliation $\mathcal{F}_{\undertilde{K}}$, denoted $\widetilde{Q} := \mathscr{T}^* \mathcal{F}_{\undertilde{K}} \,\equiv\, \undertilde{\mathbf{K}}^*$, and defined as
\be
\widetilde{Q} \,=\, \mathscr{T}^*\mathcal{F}_{\undertilde{K}} \,:=\, \bigsqcup_{\undertilde{F} \in \mathcal{F}_{\undertilde{K}}} \T^* \undertilde{F} \,,
\ee
where by $\undertilde{F}$ we denote a leaf of $\mathcal{F}_{\undertilde{K}}$.
%%%
The manifold $\widetilde{Q}$ has local coordinates $(\widetilde{q}) = (x^a, f^A, {\mu}_A)$.
%%%
The tangent bundle of this new space is $\T \widetilde{Q} = \mathbf{T}(\mathscr{T}^* \mathcal{F}_{\undertilde{K}})$, with local coordinates $(\widetilde{q}, \dot{\widetilde{q}}) = (x^a, f^A, {\mu}_A, \dot{x}^a, \dot{f}^A, \dot{\mu}_A)$.
%%%

\begin{proposition}
There exists a canonical isomorphism $\alpha$ that relates $\T \widetilde{Q}$ to the thickened space $\widetilde{M} \simeq \mathscr{T}^* \mathbf{T} \mathcal{F}_{\undertilde{K}}$
\be
\alpha \colon \mathbf{T}\widetilde{Q} \to \widetilde{M}
\ee
%%%
\begin{proof}
The isomorphism is the Tulczyjew isomorphism for foliations defined in \cref{Subsec: A Tulczyjew isomorphism for foliations}
\be
\begin{split}
\alpha &\colon \T \widetilde{Q} \,=\, \T \mathscr{T}^* \mathcal{F}_{\undertilde{K}} \to \mathscr{T} \mathcal{F}_{\mathbf{T} \undertilde{K}} \,=\, \mathscr{T} \mathcal{F}_{K} \\
&:  (x^a, f^A, \mu_A, \dot{x}^a, \dot{f}^A, \dot{\mu}_A) \mapsto (x^a, f^A, \dot{x}^a, \dot{f}^A, {\mu_f}_A = {\dot{\mu}}_A, {\mu_{\dot{f}}}_A = \mu_A) \,.
\end{split}
\ee
\end{proof}
\end{proposition}

\noindent The isomorphism $\alpha$ is evidently differentiable, thus defining a diffeomorphism, and allows us to endow the regularized manifold $\widetilde{M}$ with the structure of a tangent bundle, by "pushing forward" the canonical tangent structure $(S_{\T\widetilde{Q}}, \Delta_{\T\widetilde{Q}})$ from $\T\widetilde{Q}$ to $\widetilde{M}$.
%%%
\noindent We define the tangent structure $(\widetilde{S}, \widetilde{\Delta})$ on $\widetilde{M}$ as:
\begin{eqnarray}
\widetilde{\Delta} &:=& \alpha_*(\Delta_{\T\widetilde{Q}}) \\
\widetilde{S} &:=& \alpha^*(S_{\T\widetilde{Q}}) \,.
\end{eqnarray}
%%%
Intrinsically, $\widetilde{\Delta}$ is the vector field on $\widetilde{M}$ that is $\alpha$-related to the canonical Liouville field $\Delta_{\T\widetilde{Q}}$. 
%%%
The tensor $\widetilde{S}$ is the unique $(1,1)$-tensor on $\widetilde{M}$ satisfying $\widetilde{S} \circ \alpha_* = \alpha_* \circ S_{\T\widetilde{Q}}$.
%%%
Since $\alpha$ is a diffeomorphism, this new structure $(\widetilde{M}, \widetilde{S}, \widetilde{\Delta})$ automatically satisfies the tangent bundle axioms (\cref{Eq: cond tangent}).
%%%

\noindent To see this structure explicitly, we compute its local form.
%%%
The canonical structure on $\T\widetilde{Q}$ (with coordinates $(x^a, f^A, \mu_A, \dot{x}^a, \dot{f}^A, \dot{\mu}_A)$) is:
\begin{eqnarray}
\Delta_{\T\widetilde{Q}} &=& \dot{x}^a \frac{\partial}{\partial \dot{x}^a} + \dot{f}^A \frac{\partial}{\partial {\dot{f}}^A} + \dot{\mu}_A \frac{\partial}{\partial \dot{\mu}_A} \,, \\
S_{\T\widetilde{Q}} &=& \dd x^a \otimes \frac{\partial}{\partial \dot{x}^a} + \dd f^A \otimes \frac{\partial}{\partial \dot{f}^A} + \dd \mu_A \otimes \frac{\partial}{\partial \dot{\mu}_A} \,.
\end{eqnarray}
%%%

At this point, a fundamental issue arises. 
%%%
Let us pull-back the standard regularized symplectic form $\widetilde{\omega}$ from $\widetilde{M}$ to $\T\widetilde{Q}$ via the Tulczyjew isomorphism $\alpha$, and check if the resulting 2-form $\widehat{\omega} := \alpha^* \widetilde{\omega}$ is Lagrangian with respect to the canonical tangent structure $S_{\T\widetilde{Q}}$.
%%%

\noindent Recall that the thickened symplectic form constructed via the standard Hamiltonian coisotropic embedding is:
\be
\widetilde{\omega} \,=\, \tau^* \omega_L + \dd \left( {\mu_f}_A P^A + {\mu_{\dot{f}}}_A R^A \right) \,,
\ee
where the 1-forms $P^A$ and $R^A$ are the 1-forms defining an almost-product structure $P$ adapted to $K$
\be
P \,=\, P^A \otimes \frac{\de}{\de f^A} + R^A \otimes \frac{\de}{\de \dot{f}^A} \,,
\ee
where
\begin{align}
P^A \,&=\, \dd f^A - P^A_a \dd x^a - P'^A_a \dd \dot{x}^a \,, \\
R^A \,&=\, \dd \dot{f}^A - R^A_a \dd x^a - R'^A_a \dd \dot{x}^a \,.
\end{align}
%%%
Applying the pull-back $\alpha^*$, we obtain:
\be
\widehat{\omega} \,=\, \tau^* \omega_L + \dd \dot{\mu}_A \wedge P^A + \dot{\mu}_A \dd P^A + \dd \mu_A \wedge R^A + \mu_A \dd R^A \,.
\ee
%%%
For $\widehat{\omega}$ to be a regular Lagrangian 2-form on $\T\widetilde{Q}$, it must satisfy the symmetry condition \eqref{Eq: symmetry}, namely $\widehat{\omega}(S_{\T\widetilde{Q}} X, Y) = \widehat{\omega}(S_{\T\widetilde{Q}} Y, X)$ for any pair of vector fields $X, Y$.
%%%
Let us test this condition using the canonical tangent structure
\be
S_{\T\widetilde{Q}} \,=\, \dd x^a \otimes \frac{\partial}{\partial \dot{x}^a} + \dd f^A \otimes \frac{\partial}{\partial \dot{f}^A} + \dd \mu_A \otimes \frac{\partial}{\partial \dot{\mu}_A} \,,
\ee
and the specific pair of coordinate vector fields $X = \frac{\partial}{\partial \mu_A}$ and $Y = \frac{\partial}{\partial f^B}$.
%%%
Applying the vertical endomorphism, we have $S_{\T\widetilde{Q}} X = \frac{\partial}{\partial \dot{\mu}_A}$ and $S_{\T\widetilde{Q}} Y = \frac{\partial}{\partial \dot{f}^B}$. 
%%%
Evaluating the left-hand side of the symmetry condition yields:
\be
\widehat{\omega}(S_{\T\widetilde{Q}} X, Y) \,=\, \widehat{\omega}\left( \frac{\partial}{\partial \dot{\mu}_A},\, \frac{\partial}{\partial f^B} \right) \,.
\ee
%%%
The only term in $\widehat{\omega}$ containing $\dd \dot{\mu}_A$ is $\dd \dot{\mu}_A \wedge P^A$. 
%%%
Since $P^A\left(\frac{\partial}{\partial f^B}\right) = \delta^A_B$, we obtain
\be
\widehat{\omega}(S_{\T\widetilde{Q}} X, Y) \,=\, \delta^A_B \,.
\ee
%%%
\noindent Conversely, evaluating the right-hand side yields
\be
\widehat{\omega}(S_{\T\widetilde{Q}} Y, X) \,=\, \widehat{\omega}\left( \frac{\partial}{\partial \dot{f}^B},\, \frac{\partial}{\partial \mu_A} \right) \,=\, - \widehat{\omega}\left( \frac{\partial}{\partial \mu_A},\, \frac{\partial}{\partial \dot{f}^B} \right) \,.
\ee
The only term in $\widehat{\omega}$ containing $\dd \mu_A$ is $\dd \mu_A \wedge R^A$. 
%%%
Since $R^A\left(\frac{\partial}{\partial \dot{f}^B}\right) = \delta^A_B$, we obtain
\be
\widehat{\omega}(S_{\T\widetilde{Q}} Y, X) \,=\, - \delta^A_B \,.
\ee
%%%
This implies that for the 2-form to be Lagrangian, we would fundamentally need $\delta^A_B = - \delta^A_B$. 
%%%
This shows that $\widehat{\omega}$ is never a Lagrangian 2-form, regardless of the choice of the almost-product structure $P$.
%%%
Therefore, the standard coisotropic embedding inherited from the Hamiltonian setting fundamentally breaks the tangent bundle geometry, making it mandatory to slightly modify the regularization scheme to preserve the Lagrangian nature of the system.
%%%

At this stage, having proved that the standard coisotropic embedding inevitably breaks the Lagrangian nature of the system with respect to the canonical tangent structure on $\T\widetilde{Q}$, there are essentially two paths to proceed:
\begin{itemize}
    \item Modifying the isomorphism: One can abandon the canonical Tulczyjew isomorphism $\alpha$ and construct a different bundle isomorphism between the thickened space $\widetilde{M}$ and the tangent bundle $\T\widetilde{Q}$. 
    %%%
    This is the approach adopted by \textit{A. Ibort} and \textit{J. Marín-Solano} in \cite{ibort-marin}, where they introduce an arbitrary Riemannian metric on the fibers of the vector bundle $\mathbf{K} \to \T Q$ to build a non-canonical, metric-dependent isomorphism that correctly "twists" the variables to get a Lagrangian 2-form.
    %%%
    \item Modifying the regularized 2-form: One can preserve the canonical, purely geometric Tulczyjew isomorphism $\alpha$ and modify the definition of the regularized 2-form itself.
\end{itemize}

\noindent In the present work, we adopt the second approach. 
%%%
Specifically, rather than relying on the standard Hamiltonian coisotropic form $\widetilde{\omega}$, we construct the regularized Lagrangian 2-form $\widehat{\omega}$ on $\T\widetilde{Q}$ by taking the pull-back of the original degenerate Lagrangian form, $\alpha^* \tau^* \omega_L$, and adding a correction term that is Lagrangian by construction. 
%%%
This term takes the form $-\dd\dd_{S_{\T\widetilde{Q}}} F$, where $F \in \mathcal{C}^\infty(\T\widetilde{Q})$ is a globally defined smooth function.
%%%
The construction of such a function $F$ requires fixing an auxiliary connection on the bundle $\widetilde{Q} \to Q$ satisfying suitable properties (which are fulfilled, for example, by any linear connection).
%%%
This new methodology presents two significant advantages over the existing literature:
\begin{itemize}
    \item It requires fixing a less restrictive geometric structure (a connection) compared to the requirement of a full Riemannian metric.
    \item It yields a \textit{globally} defined regularized Lagrangian function $\widetilde{L}$ that generates the dynamics, unlike the approach in \cite{ibort-marin} which only guarantees the existence of local Lagrangian functions.
\end{itemize}
%%%

\noindent The definition of the function $F$ is not canonical, and depends on the choice of two ingredients:
\begin{itemize}
    \item An Ehresmann connection $\nabla$ on the bundle $\widetilde{Q} \,=\, \mathbf{K}^* \longrightarrow Q$, given by a splitting of the tangent bundle in vertical and horizontal vectors
    \[
    \T \widetilde{Q}^\ast  = \mathcal{V} \oplus \mathcal{H}_\nabla\,.
    \]
    %%%
    This connection is chosen so that the splitting at $Q$ (identified as a submanifold via the zero section), is the canonical splitting 
    \[
    \T \widetilde{Q}  |_Q = \mathcal{V} \oplus \T Q\,,
    \]
   in order for $F$ to be zero at $Q$ (and hence, to define an extension of $L$). 
   %%%
   This can be achieved simply by choosing a linear connection, though it is not necessary.
   %%%
    \item An almost product structure $P$ on $Q$, which complements the distribution $\undertilde K$.
\end{itemize}

\begin{remark}[\textsc{Coordinate expressions}]
\label{remark:coordinate_expression_nabla_P_symp}
Locally, we express the components of the connection as
\[
\mathcal{H}_\nabla = \operatorname{span}\left\{ \pdv{x^a} + \Gamma_{aA} \pdv{\mu_A}\,, \pdv{f^B} + \Gamma_{BA} \pdv{\mu_A} \right\}\,.
\]
%%%
The condition on $\nabla$ inducing the canonical splitting at the zero section is reflected in the $\Gamma$'s vanishing at $Q$ (again identified via the zero section).
%%%

\noindent On the other hand, we express the projector defining the almost product structure as 
\[
P = P^A \otimes \pdv{f^A} = \left(\dd f^A - P^A_a \dd x^a \right) \otimes \pdv{f^A}\,.
\]
\end{remark}

Now, consider the following maps 
% https://q.uiver.app/#q=WzAsOCxbMSwwLCJKXjEgXFx3aWRldGlsZGUgXFxwaSJdLFswLDAsIihcXHVuZGVydGlsZGV7S31eXFx1cGFycm93KV5cXGFzdCJdLFsyLDAsIlxcbWF0aGNhbHtWfSBcXG9wbHVzIFxcbWF0aGNhbHtIfV9cXG5hYmxhID0gXFxUIFxcdW5kZXJ0aWxkZXtLfV5cXGFzdCJdLFsyLDEsIlxcbWF0aGNhbHtWfSJdLFsyLDIsIlxcdW5kZXJ0aWxkZXtLfV5cXGFzdCJdLFswLDEsIkpeMSBcXHBpIl0sWzEsMSwiXFxUIFxcbWF0aGJme1F9Il0sWzEsMiwiXFx1bmRlcnRpbGRle0t9Il0sWzAsMiwiaV97XFx3aWRldGlsZGVcXHBpfSIsMCx7InN0eWxlIjp7InRhaWwiOnsibmFtZSI6Imhvb2siLCJzaWRlIjoiYm90dG9tIn19fV0sWzIsMywicF9cXG1hdGhjYWx7Vn0iXSxbMyw0XSxbMSw1LCJcXHRhdSIsMl0sWzUsNiwiaV9cXHBpIiwyLHsic3R5bGUiOnsidGFpbCI6eyJuYW1lIjoiaG9vayIsInNpZGUiOiJ0b3AifX19XSxbNiw3LCJQIiwyXSxbMCwxLCJcXGFscGhhIiwyXV0=
\[\begin{tikzcd}[cramped]
	{\widetilde{M}} & {\T \widetilde{Q} \simeq \mathcal{V} \oplus \mathcal{H}_\nabla} \\
	{\T Q} &   {\mathcal{V}} \\
	  {\undertilde{\mathbf{K}}} &  {\undertilde{\mathbf{K}}^\ast}
	\arrow["\tau"', from=1-1, to=2-1]
	\arrow["\alpha"', from=1-2, to=1-1]
	\arrow["{p_\mathcal{V}}", from=1-2, to=2-2]
	\arrow["P"', from=2-1, to=3-1]
	\arrow[from=2-2, to=3-2]
\end{tikzcd}\,,\]
where $p_{\mathcal{V}}$ denotes the projection from $\T \widetilde{Q}$ to $\mathcal{V}$ defined by the connection $\nabla$ chosen and the arrow $\mathcal{V} \longrightarrow \undertilde{K}^\ast$ is the identification of the vertical bundle with the fiber of a vector bundle.
%%%
Then, we define the map 
\be
F^{P, \nabla} \colon \T \widetilde{Q} \longrightarrow \mathbb{R}
\ee
for $\xi \in \T \widetilde{Q}$ via the natural pairing between $\undertilde{\mathbf{K}}$ and $\undertilde{\mathbf{K}}^\ast$
\be
F^{P,\nabla}(\xi) = \langle (P \circ \tau \circ \alpha)(\xi), \,p_\mathcal{V} (\xi) \rangle\,.
\ee

\begin{remark}[\textsc{Local expression of $F^{P, \nabla}$}] Using the coordinate components of $\nabla$ and $P$ from \cref{remark:coordinate_expression_nabla_P_symp}, we have that 
\be
F^{P, \nabla} = \left( \dot{\mu}_A - \dot{x}^a \Gamma_{aA} - \dot{f}^B \Gamma_{BA}\right) \left( \dot{f}^A - \dot{x}^a P^A_a\right)\,.
\ee
\end{remark}

\noindent We then have the following:

\begin{theo}[\textsc{Lagrangian coisotropic embedding}]
\label{thm:Lagrangian_coisotropic_embedding_symplectic}
Let $L \colon \T Q \longrightarrow \mathbb{R}$ be a singular Lagrangian. 
%%%
Suppose that $L$ is consistent and that the characteristic distribution $K = \ker \omega_L$ is the complete lift of a distribution $\undertilde K$ on $Q$. 
%%%,
Then, given an Ehresmann connection $\nabla$ on $\mathbf{K}^*$ and an almost product structure $P$ on $Q$ as above, the embedding 
\be
\T Q \hookrightarrow \mathbf{K}^*
\ee
is a coisotropic embedding on a neighborhood of $\T Q$ for the symplectic structure $\omega_{\widetilde L}$, where $\widetilde L = L + F^{P, \nabla}$.
\end{theo}
\begin{proof} We will first show that $\T \left( \mathbf{K}^*\right)\big |_{\T Q}$ is a symplectic vector bundle, so that $\widetilde \omega$ defines a symplectic structure on some neighborhood of $\T Q$. 
%%%
Indeed, a quick computation shows that 
\begin{align}
    \dd_{S_{\T \widetilde{Q}}} \widetilde L =& \dd_S L + \dot{f}^A \dd \mu_A + \dot{\mu}_A \dd f^A\\
    &- \left(\Gamma_{aA}(\dot{f}^A - \dot{x}^b P^A_b) + P^{A}_a(\dot{\mu}_A - \dot{x}^b \Gamma_{bA} - \dot{f}^B \Gamma_{BA})\right) \dd x^a \\
    & - \left(\Gamma_{BA} (\dot{f}^A - \dot{x}^a P^A_a) + \dot{x}^a \Gamma_{aA} + \dot{f}^B \Gamma_{BA}\right)\dd f^A\\
    &- \dot{x}^a P^A_a \dd \mu_A\,.
 \end{align}
 %%%
 Hence, taking the exterior differential and restricting to the zero section (so that all $\Gamma$'s vanish), we obtain the following $2$-form
\be
\dd \dd_{S_{\T \widetilde{Q}}} \widetilde L = \tau^* \omega_L + \dd \dot{f}^A \wedge \dd  \mu_A + \dd {\dot{\mu}_A} \wedge \dd f^A + \text{(semi-basic terms)}\,.
\ee
%%%
Notice that the first three terms in the right-hand side define a symplectic structure. 
%%%
Since adding semi-basic terms (with respect to the projection onto $Q$) does not change regularity, we have that $\T \left( \mathbf{K}^*\right)\big |_{\T Q}$ is a symplectic vector bundle. 
%%%
Finally, notice that it is a coisotropic embedding, as 
\[
\dd_{S_{\T \widetilde{Q}}} L\big |_{\T Q} = \dd_S L\,.
\]
\end{proof}

Having established a constructive method for a regularized Lagrangian system, it is natural to ask to what extent this regularization depends on the specific choices made (i.e., the connection $\nabla$ and the almost-product structure $P$). 
%%%
While the global geometry of the thickened space cannot be unique—since it heavily depends on these arbitrary choices evaluated away from the zero section—its behavior infinitesimally close to the original physical system is completely rigid. 
%%%
In mathematical terms, we can prove that its \textit{first-order germ} along the original manifold $\T Q$ (that is, the regularized symplectic form and the extended tangent structure evaluated exactly on the points of $\T Q$) is geometrically unique, provided the restricted tangent structures coincide.
%%%

\noindent To prove this, we first need a purely algebraic lemma regarding symplectic vector spaces equipped with nilpotent endomorphisms (which act as local models for tangent structures).

\begin{lemma} \label{Lemma: Lagrangian complement trick}
Let $(E, \omega)$ be a symplectic vector space, equipped with a $(1,1)$-tensor $J$ satisfying $J^2 = 0$ and $\omega(Jx, y) = \omega(Jy, x)$ (which is equivalent to $i_J \omega = 0$). 
%%%
Let $L \subset E$ be a Lagrangian subspace such that $J(L) \subseteq L$. 
%%%
If there exists a complementary subspace $W$ such that $E = L \oplus W$ and $J(W) \subseteq W$, then we can construct a new complement $\mathbf{A}(W)$ which is Lagrangian and remains invariant under $J$.
\end{lemma}
\begin{proof}
Since $L$ is a Lagrangian subspace, it is maximally isotropic, which means $\dim L = \frac{1}{2} \dim E$ and $\omega(l_1, l_2) = 0$ for any $l_1, l_2 \in L$. 
%%%
Because $W$ is a complement ($E = L \oplus W$), it immediately follows that $\dim W = \dim E - \dim L = \frac{1}{2} \dim E$. 
%%%
Furthermore, the non-degeneracy of $\omega$ ensures that the map $\phi \colon L \to W^*$ defined by $\phi(l) := (i_l \omega)|_W$ is a linear isomorphism. 
%%%

We define a deformation map $\mathbf{A} \colon W \to L \oplus W$ by adding a specific correction term in $L$ to every vector in $W$. Let $l_w := -\frac{1}{2}\phi^{-1}(i_w \omega|_W) \in L$, and define:
\be
\mathbf{A}(w) \,:=\, w + l_w \,.
\ee
Notice that $\mathbf{A}$ is injective: if $\mathbf{A}(w) = 0$, then $w = -l_w$. Since $w \in W$ and $l_w \in L$, and their intersection is trivial ($L \cap W = \{0\}$), it must be that $w = 0$. 
%%%
Since $\mathbf{A}$ is an injective linear map, the image subspace $\mathbf{A}(W)$ has exactly the same dimension as $W$, namely $\dim \mathbf{A}(W) = \frac{1}{2} \dim E$. 
%%%
Additionally, we must ensure that $\mathbf{A}(W)$ is a valid complement to $L$, meaning their intersection is trivial. 
%%%
Suppose a vector $v = w + l_w \in \mathbf{A}(W)$ also belongs to $L$. 
%%%
Since $l_w \in L$, this implies $w = v - l_w \in L$. 
%%%
However, since the original sum $E = L \oplus W$ is direct, we have $L \cap W = \{0\}$, which forces $w = 0$. 
%%%
Consequently, $l_w = 0$, meaning the only vector in the intersection is the zero vector. 
%%%
Thus, $\mathbf{A}(W)$ intersects $L$ trivially and serves as a valid complementary subspace.

Next, we explicitly show that $\mathbf{A}(W)$ is an isotropic subspace. Let $w_1, w_2 \in W$ and consider their images under $\mathbf{A}$. Expanding the symplectic form using bilinearity, we get:
\be
\begin{split}
    \omega(\mathbf{A}(w_1),\, \mathbf{A}(w_2)) \,&=\, \omega(w_1 + l_{w_1},\, w_2 + l_{w_2}) \\
    \,&=\, \omega(w_1, w_2) + \omega(w_1, l_{w_2}) + \omega(l_{w_1}, w_2) + \omega(l_{w_1}, l_{w_2}) \,.
\end{split}
\ee
The last term $\omega(l_{w_1}, l_{w_2}) = 0$, because $L$ is isotropic.
%%%

\noindent By definition of the isomorphism $\phi$, we have $\omega(l_{w_1}, w_2) = \phi(l_{w_1})(w_2)$. 
%%%
Substituting $l_{w_1}$, we get $\phi \left(-\frac{1}{2}\phi^{-1}(i_{w_1} \omega) \right)(w_2) = -\frac{1}{2}(i_{w_1} \omega)(w_2) = -\frac{1}{2}\omega(w_1, w_2)$.
%%%

\noindent With this in mind we get
\be
\omega(\mathbf{A}(w_1),\, \mathbf{A}(w_2)) \,=\, \omega(w_1, w_2) - \frac{1}{2}\omega(w_1, w_2) - \frac{1}{2}\omega(w_1, w_2) + 0 \,=\, 0 \,.
\ee
%%%
Therefore, $\mathbf{A}(W)$ is isotropic. 
%%%
Being an isotropic subspace with dimension exactly half of $\dim E$, $\mathbf{A}(W)$ is Lagrangian.
%%%

Finally, we must ensure that $\mathbf{A}(W)$ is invariant under $J$. 
%%%
Let $w \in W$. 
%%%
It will be enough to show that $J$ commutes with the correction term, i.e., $J(l_w) = l_{J(w)}$, which is equivalent to $J(\phi^{-1}(w)) = \phi^{-1}(J(w))$.
%%%
Let $l = \phi^{-1}(w)$, which by definition means $\omega(l, x) = \omega(w, x)$ for all $x \in W$. 
%%%
Since $W$ is invariant under $J$, the vector $x = J(w')$ also belongs to $W$ for any $w' \in W$. 
%%%
Substituting this into our defining equation yields $\omega(l, J(w')) = \omega(w, J(w'))$. 
%%%
Using the symmetry of $J$ with respect to $\omega$ ($i_J \omega = 0$), we can move $J$ to the first slot, obtaining $-\omega(J(l), w') = -\omega(J(w), w')$. 
%%%
Since this holds for all $w' \in W$ and the map $\phi$ is an isomorphism, it follows that $J(l) = \phi^{-1}(J(w))$, concluding the proof.
%%%
\end{proof}
%%%

\begin{remark}[\textsc{The case of symplectic vector bundles}]
{
\label{remark:complement_symplectic_vector_bundles}
Notice that \cref{Lemma: Lagrangian complement trick} applies as well to the case of symplectic vector bundles $(E, \omega) \longrightarrow M$ together with a nilpotent endomorphism $J$, an invariant Lagrangian subbundle $L$, and an invariant complement $W$. Indeed, the construction presented is global and maintains smoothness.
}
\end{remark}

\begin{remark}
\label{remark:Existence_of_J_invariant_complement}
{
Finally, notice that \cref{Lemma: Lagrangian complement trick} requires the existence of a $J$-invariant complement $W$. In the case of nilpotent endomorphisms, this may be built as follows. Let $\widetilde W$ be a complement of $J(L)$ in $J(E)$, so that $J(E) = J(L) \oplus \widetilde W$. Then, we can define $W := J^{-1}(\widetilde W)$, which clearly satisfies $\widetilde W = J(W)\subseteq W$ ($J$ being nilpotent) and $E = L\oplus W$. This holds as well in the case of symplectic vector bundles when $J$ and $J|_{L}$ have constant rank, which will certainly hold in our case.
}
\end{remark}

\noindent Using this algebraic tool, we can now prove the uniqueness theorem for the Lagrangian regularization.
%%%

\begin{theo}[\textsc{Uniqueness of Lagrangian coisotropic embedding to first order}]
\label{thm:uniqueness_sympelctic_first_order}
Let $(\widetilde{M}_1, \widetilde{\omega}_1)$ and $(\widetilde{M}_2, \widetilde{\omega}_2)$ be two symplectic regularizations of the degenerate Lagrangian system $(\T Q, \omega_L)$. 
%%%
Suppose that both regularizations admit a tangent structure $\widetilde{S}_i$ making the respective forms Lagrangian (i.e., $i_{\widetilde{S}_i} \widetilde{\omega}_i = 0$), and that the embeddings $\mathfrak{i}_i \colon \T Q \hookrightarrow \widetilde{M}_i$ preserve the tangent structure: {$(\mathfrak{i}_i)_* \circ  S_{\T Q} = \widetilde{S}_i|_{\T Q} \circ (\mathfrak{i}_i)_*$}.
%%%
Then, there exist tubular neighborhoods $U_1, U_2$ of $\T Q$ in $\widetilde M_1$ and $\widetilde{M}_2$, respectively, together with a local diffeomorphism $\psi \colon U_1 \to U_2$ restricting to the identity on $\T Q$, such that its pushforward provides an exact isomorphism of the tangent-symplectic structures over $\T Q$:
\be
\psi_* (\widetilde{\omega}_1,\, \widetilde{S}_1)\big|_{\T Q} \,=\, (\widetilde{\omega}_2,\, \widetilde{S}_2)\big|_{\T Q} \,.
\ee
\end{theo}
\begin{proof}
{Let $n+r = \dim Q$, where} $K = \ker \omega_L$ is the characteristic distribution on $\T Q$, with rank $2r$ (because of the hypothesis that $K$ is the tangent distribution to a rank $r$ distribution $\undertilde{K}$ on $Q$). 
%%%
We can decompose the tangent space of $\T Q$ as $\T(\T Q) = K \oplus W$, where the complementary subbundle $W$ must be chosen to be invariant under the vertical endomorphism $S_{\T Q}$. 
%%%
Such a complement naturally arises from the geometry of the tangent bundle: we can choose a distribution $\undertilde{H}$ complementary to $\undertilde{K}$ on the base manifold $Q$ (so that $\T Q = \undertilde{K} \oplus \undertilde{H}$), and define $W$ as its tangent distribution. 
%%%
Since $W$ is pointwise spanned by the complete lifts $X^C$ and vertical lifts $X^V$ of vector fields $X \in \undertilde{H}$, the fundamental properties $S_{\T Q}(X^C) = X^V$ and $S_{\T Q}(X^V) = 0$ intrinsically guarantee that $S_{\T Q}(W) \subseteq W$. 
%%%
Furthermore, since $K$ is the kernel of $\omega_L$, the restriction of $\omega_L$ to $W$ is non-degenerate, making $(W, \omega_L|_W)$ a symplectic vector bundle of rank $2n$.
%%%

Inside the tangent space of the thickened manifold $\T \widetilde{M}_i|_{\T Q}$, we consider the symplectic orthogonal to $W$, defined as $W^{\perp, \widetilde{\omega}_i} = \{v \mid \widetilde{\omega}_i(v, W) = 0\}$, {which is again a symplectic vector bundle}.
%%%
Crucially, $W^{\perp, \widetilde{\omega}_i}$ is invariant under $\widetilde{S}_i$. Indeed, taking $v \in W^{\perp, \widetilde{\omega}_i}$ and testing it against $W$:
{
\be
\widetilde{\omega}_i(\widetilde{S}_i (v),\, W) \,=\, \widetilde{\omega}_i( \widetilde{S}_i (W),\, v) \,=\, \widetilde{\omega}_i( S_{\T Q} (W), \, v) \,=\, 0 \,,
\ee
}
where we used the symmetry of $\widetilde{S}_i$, the hypothesis that the embedding preserves the structure ($\widetilde{S}_i|_W = S_{\T Q}|_W$), and the fact that $S_{\T Q}$ preserves $W$.
%%%

By dimensional counting on the coisotropic embedding, $\dim \widetilde{M}_i = \dim(\T Q) + \dim K = (2n + 2r) + 2r = 2n + 4r$, where $2n$ is the dimension of $W$. 
%%%
Since $W$ is symplectic, then $\dim \widetilde{M}_i \,=\, \dim W + \dim W^{\perp, \widetilde{\omega}_i}$, implying $\dim W^{\perp, \widetilde{\omega}_i} = 4r$. 
%%%
By definition, $K \subset W^{\perp, \widetilde{\omega}_i}$ is isotropic. Since $\dim K = 2r = \frac{1}{2} \dim W^{\perp, \widetilde{\omega}_i}$, $K$ is a Lagrangian subbundle of $W^{\perp, \widetilde{\omega}_i}$.

Applying \cref{Lemma: Lagrangian complement trick} {(together with \cref{remark:complement_symplectic_vector_bundles} and \cref{remark:Existence_of_J_invariant_complement})} fiber-wise, we can construct a $J$-invariant Lagrangian complement for $K$, allowing us to identify $W^{\perp, \widetilde{\omega}_i} \cong K \oplus K^*$. 
%%%
This provides a global, structure-preserving local isomorphism, say $\psi$:
\be
\T \widetilde{M}_i\big|_{\T Q} \,\cong\, W \oplus K \oplus K^* \,.
\ee
We can therefore choose local coordinates $(x^a, \dot{x}^a)$ for $W$, $(f^A, \dot{f}^A)$ for $K$, and fiber coordinates $(\mu_A, \dot{\mu}_A)$ for $K^*$. 
%%%
In this universal adapted frame evaluated precisely on $\T Q$ (where $\mu_A = 0, \dot{\mu}_A = 0$), any regularized symplectic form must locally read:
\be
\omega \,=\, \omega_L + \dd \mu_A \wedge \dd f^A + \dd \dot{\mu}_A \wedge \dd \dot{f}^A \,.
\ee
%%%

It remains to show that the extension of the tangent structure $\widetilde{S}_i$ is also uniquely determined on $\T Q$. Let $\widehat{S}$ be an arbitrary $(1,1)$-tensor extending $S_{\T Q}$. In our universal adapted frame evaluated on $\T Q$ (where $\mu_A = 0, \dot{\mu}_A = 0$), the most general matrix form for $\widehat{S}$ that acts as $S_{\T Q}$ on the base manifold is:
\begin{align}
    \widehat{S} \,=\,& \left( \dd x^a + F^{aA} \dd \mu_A + \dot{F}^{aA} \dd \dot{\mu}_A \right) \otimes \frac{\partial}{\partial \dot{x}^a} \\
    &+ \left( \dd f^A + G^{AB} \dd \mu_B + \dot{G}^{AB} \dd \dot{\mu}_B \right) \otimes \frac{\partial}{\partial \dot{f}^A} \\
    &+ \left( H^B_A \dd \mu_B + \dot{H}^B_A \dd \dot{\mu}_B \right) \otimes \frac{\partial}{\partial \mu_A} \\
    &+ \left( I^B_A \dd \mu_B + \dot{I}^B_A \dd \dot{\mu}_B \right) \otimes \frac{\partial}{\partial \dot{\mu}_A} \,.
\end{align}

\noindent We must impose the Lagrangian condition $i_{\widehat{S}} \omega = 0$. Recall that the interior product of a $(1,1)$-tensor with a 2-form acts as $(i_{\widehat{S}}\omega)(X,Y) = \omega(\widehat{S}X, Y) + \omega(X, \widehat{S}Y)$, which extends to wedge products as $i_{\widehat{S}}(\alpha \wedge \beta) = (\widehat{S}^*\alpha) \wedge \beta + \alpha \wedge (\widehat{S}^*\beta)$. 
%%%
Applying the pull-back $\widehat{S}^*$ to the basic 1-forms yields:
\begin{align}
    \widehat{S}^*(\dd x^a) \,&=\, 0 \,, \\
    \widehat{S}^*(\dd f^A) \,&=\, 0 \,, \\
    \widehat{S}^*(\dd \dot{x}^a) \,&=\, \dd x^a + F^{aA} \dd \mu_A + \dot{F}^{aA} \dd \dot{\mu}_A \,, \\
    \widehat{S}^*(\dd \dot{f}^A) \,&=\, \dd f^A + G^{AB} \dd \mu_B + \dot{G}^{AB} \dd \dot{\mu}_B \,, \\
    \widehat{S}^*(\dd \mu_A) \,&=\, H^B_A \dd \mu_B + \dot{H}^B_A \dd \dot{\mu}_B \,, \\
    \widehat{S}^*(\dd \dot{\mu}_A) \,&=\, I^B_A \dd \mu_B + \dot{I}^B_A \dd \dot{\mu}_B \,.
\end{align}

We now evaluate $i_{\widehat{S}}$ term by term on the symplectic form $\omega = \omega_L + \dd \mu_A \wedge \dd f^A + \dd \dot{\mu}_A \wedge \dd \dot{f}^A$.
%%%
For the base form $\omega_L$, since its kernel is exactly $K = \operatorname{span}\{\frac{\partial}{\partial f^A}, \frac{\partial}{\partial \dot{f}^A}\}$, it only contracts non-trivially with coordinates $(x, \dot{x})$. 
%%%
Because $\omega_L$ is already Lagrangian with respect to the base tangent structure ($i_{S_{\T Q}}\omega_L = 0$), applying $i_{\widehat{S}}$ only extracts the newly added transverse coefficients
\be
i_{\widehat{S}} \omega_L \,=\, \left( F^{aA} \dd \mu_A + \dot{F}^{aA} \dd \dot{\mu}_A \right) \wedge i_{\frac{\partial}{\partial \dot{x}^a}} \omega_L \,.
\ee
For the second term, applying the product rule and using $\widehat{S}^*(\dd f^A) = 0$ one gets
\be
i_{\widehat{S}}(\dd \mu_A \wedge \dd f^A) \,=\, \left( H^B_A \dd \mu_B + \dot{H}^B_A \dd \dot{\mu}_B \right) \wedge \dd f^A \,,
\ee
whereas, for the third term one obtains
\be
i_{\widehat{S}}(\dd \dot{\mu}_A \wedge \dd \dot{f}^A) \,=\, \left( I^B_A \dd \mu_B + \dot{I}^B_A \dd \dot{\mu}_B \right) \wedge \dd \dot{f}^A + \dd \dot{\mu}_A \wedge \left( \dd f^A + G^{AB} \dd \mu_B + \dot{G}^{AB} \dd \dot{\mu}_B \right) \,.
\ee
%%%

\noindent Now, the form $\dd \dot{f}^A$ only appears in the third piece, multiplied by the $I$ matrices. 
%%%
Since it is linearly independent from all other forms, its coefficients must vanish, forcing $I^B_A = 0$ and $\dot{I}^B_A = 0$.
%%%

\noindent Gathering the $\dd f^A$ wedges from the second and third pieces, we get $\left( H^B_A \dd \mu_B + \dot{H}^B_A \dd \dot{\mu}_B + \dd \dot{\mu}_A \right) \wedge \dd f^A = 0$, forcing $H^B_A = 0$ and $\dot{H}^B_A = -\delta^B_A$.
%%%

\noindent The 1-forms $i_{\frac{\partial}{\partial \dot{x}^a}} \omega_L$ are linear combinations of $\dd x^b$. 
%%%
Since $\dd x^b$ do not appear anywhere else in the expansion, the $F$ matrices multiplying them must vanish, forcing $F^{aA} = 0$ and $\dot{F}^{aA} = 0$.
%%%

\noindent We are left with only $\dd \dot{\mu}_A \wedge (G^{AB} \dd \mu_B + \dot{G}^{AB} \dd \dot{\mu}_B) = 0$ from the third term. 
%%%
The linear independence of $\dd \dot{\mu}_A \wedge \dd \mu_B$ forces $G^{AB} = 0$. 
%%%
For the second part, $\dot{G}^{AB} \dd \dot{\mu}_A \wedge \dd \dot{\mu}_B = 0$ implies that the matrix $\dot{G}^{AB}$ must be symmetric ($\dot{G}^{AB} = \dot{G}^{BA}$).
%%%

\noindent Finally, we impose the structural condition established by \cref{Lemma: Lagrangian complement trick}, namely that the extended tangent structure must leave the dual Lagrangian complement invariant: $\widehat{S}(K^*) \subseteq K^*$. 
%%%
The subspace $K^*$ is generated by $\{\frac{\partial}{\partial \mu_C}, \frac{\partial}{\partial \dot{\mu}_C}\}$. 
%%%
Applying our simplified tensor $\widehat{S}$ to the basis vector $\frac{\partial}{\partial \dot{\mu}_C}$, we obtain:
\be
\widehat{S}\left( \frac{\partial}{\partial \dot{\mu}_C} \right) \,=\, \dot{G}^{CA} \frac{\partial}{\partial \dot{f}^A} - \frac{\partial}{\partial \mu_C} \,.
\ee
For this resulting vector to remain within $K^*$, it cannot possess any component along $\frac{\partial}{\partial \dot{f}^A}$, which belongs to $K$. This geometrically forces the symmetric matrix $\dot{G}^{CA}$ to be strictly zero.

Thus, all unknown coefficients are strictly identically zero or uniquely fixed. 
%%%
Consequently, there is only one algebraically permissible tensor $\widehat{S}$ on $\T Q$. 
%%%
Therefore, the map $\psi$ identifying the two universal splittings satisfies $\psi_* \widetilde{S}_1 = \widetilde{S}_2$ along $\T Q$, concluding the proof.
%%%
\end{proof}

\section{Regularization of non-autonomous systems}
\label{Sec:Regularization of non-autonomous systems}
\subsection{Cosymplectic Hamiltonian systems}

While the natural geometric setting for autonomous Hamiltonian systems is that of symplectic geometry, for non-autonomous systems it is that of cosymplectic geometry (see \cite{paulette,libermann}).
%%%

{
\begin{definition}[\textsc{Cosymplectic manifold}]
\label{Def: Cosymplectic manifold}
A \textbf{cosymplectic manifold} is a triple $(M, \omega, \tau)$ where $M$ is a smooth manifold of dimension $2n+1$, $\omega \in \Omega^2(M)$ is a closed 2-form, and $\tau \in \Omega^1(M)$ is a closed 1-form, satisfying the non-degeneracy condition $\tau \wedge \omega^n \neq 0$.
%%%
On every cosymplectic manifold $(M, \omega, \tau)$ there is a unique vector field $R \in \mathfrak{X}(M)$, called the \textbf{Reeb vector field}, that satisfies $i_R \omega = 0$ and $i_R \tau = 1$.
\end{definition}
\begin{theo}[\textsc{Darboux's Theorem for cosymplectic manifolds}]
\label{Thm: Darboux cosymplectic}
The Darboux theorem can be generalized to cosymplectic manifolds. Let $(M, \omega, \tau)$ be a cosymplectic manifold of dimension $2n+1$. Around every point $m \in M$, there exist local coordinates $(q^i, p_i, t)$, called \textbf{Darboux coordinates}, such that:
$$
\omega = \dd q^i \wedge \dd p_i \,, \quad \tau = \dd t \,.
$$
\end{theo}
\begin{definition}[\textsc{Cosymplectic Hamiltonian system}]
\label{Def: Cosymplectic Hamiltonian system}
A \textbf{cosymplectic Hamiltonian system} is a tuple $(M, \omega, \tau, H)$, where $(M, \omega, \tau)$ is a cosymplectic manifold and $H \in C^\infty(M)$ is a smooth function (the \textit{Hamiltonian}).
\end{definition}
}

{\noindent Similar to the symplectic case, the isomorphism $\flat$ guarantees the existence of a unique vector field ${\rm grad} \; H \in \mathfrak{X}(M)$, the \textit{gradient vector field}, satisfying:
\be
\flat({\rm grad} \; H) \,=\, \dd H \,.
\ee
%%%
To obtain the dynamics defined by the Hamiltonian (see \cite{franslacomba}), we use ${\rm grad} \; H$ to define two additional vector fields. First, the \textit{Hamiltonian vector field}:
\be
X_H \,=\, {\rm grad} \; H - R(H) R \,,
\ee
and second, the \textit{evolution vector field}:
\be \label{Eq: Evolution vector field}
{\mathcal E}_H \,=\, X_H + {\mathcal R} \,.
\ee
%%%
In Darboux coordinates $(q^i, p_i, t)$, using the local expression for ${\rm grad} \; H$, the evolution vector field takes the form:
\be \label{Eq: hcosymp2}
{\mathcal E}_H \,=\, \frac{\partial H}{\partial p_i} \frac{\partial}{\partial q^i} - \frac{\partial H}{\partial q^i} \frac{\partial}{\partial p_i} + \frac{\partial}{\partial t} \,.
\ee
%%%
The solutions of the non-autonomous Hamiltonian system are the integral curves $(q^i(\varepsilon), p_i(\varepsilon), t(\varepsilon))$ of ${\mathcal E}_H$, satisfying the time-dependent Hamilton's equations:
\be \label{Eq: hsympl3}
\frac{\dd q^i}{\dd \varepsilon} \,=\, \frac{\partial H}{\partial p_i} \,, \quad \frac{\dd p_i}{\dd \varepsilon} \,=\, - \frac{\partial H}{\partial q^i} \,, \quad \frac{\dd t}{\dd \varepsilon} \,=\, 1 \,.
\ee
Since $\frac{\dd t}{\dd \varepsilon} = 1$, we have $t = \varepsilon + \text{const}$, allowing us to identify the curve parameter with the time coordinate $t$ and recover the standard non-autonomous Hamilton's equations.}

\noindent A particular case of high relevance is obtained by taking the product
$\T^*Q \times \mathbb{R}$, where $\T^*Q$ denotes the cotangent bundle of the configuration manifold $Q$. Indeed, if we denote by $\omega = - \dd \theta_Q$, where $\theta_Q$ is the Liouville 1-form on $\T^*Q$, then the pair $(\omega, \dd t)$ defines a cosymplectic structure on $\T^*Q \times \mathbb{R}$, where here $\omega$ and $\dd t$ are the obvious extensions, $t$ being the standard coordinate in $\mathbb{R}$. A direct calculation shows that the natural bundle coordinates $(t, q^i, p_i)$ are Darboux coordinates for this cosymplectic manifold.

\begin{remark}[\textsc{Cosymplectic dynamics are Reeb dynamics}]
\label{remark:cosymplectic_dynamics_Reeb_dynamics}
\noindent Assume that $H$ is a Hamiltonian function on a cosymplectic manifold $(M, \omega, \tau)$. Then, we can construct an additional cosymplectic structure depending on $H$, say
$$
(\omega_H = \omega + \dd H \wedge \tau, \tau)\,.
$$
\noindent A simple computation shows that the evolution vector field for $H$ with respect to $(\omega, \tau)$ coincides with the Reeb vector field for
$(\omega_H, \tau)$, ${\cal E}_H = {\cal R}_H$, so that nonautonomous Hamiltonian dynamics may be studied as Reeb dynamics. This point of view will be particularly useful in the Lagrangian setting.
\end{remark}

\noindent The point of view of \cref{remark:cosymplectic_dynamics_Reeb_dynamics} is the picture that we will adhere to onwards, as it is the most natural setting to study the non-autonomous Lagrangian side {(see \cite{dLR1989, Krupkova-GeometryOVE-1997} for a comprehensive account on the Lagrangian description of time-dependent mechanics in terms of jet bundles, and \cite{manolo1,manolo2} for the singular case)}.

\subsection{Coisotropic regularization of pre-cosymplectic Hamiltonian systems}

{
In general, when working with singular time-dependent theories (such as time-dependent Hamiltonian gauge theories and time-dependent singular Lagrangian theories), we work on a pre-cosymplectic manifold rather than on a cosymplectic one.
\begin{definition}[\textsc{Pre-cosymplectic Hamiltonian system}]
\label{Def: Pre-cosymplectic Hamiltonian system}
A \textbf{pre-cosymplectic Hamiltonian system} can be fundamentally understood through its Reeb dynamics (as per \cref{remark:cosymplectic_dynamics_Reeb_dynamics}). Let $(M, \omega, \tau)$ be a pre-cosymplectic manifold. The dynamics are formally governed by the equations:
\be \label{Eq: Pre-cosymplectic Reeb eq}
i_X \omega = 0 \,, \quad \text{and} \quad i_X \tau = 1 \,.
\ee
%%%
As in the autonomous case, these equations pose two distinct problems:
\begin{enumerate}
    \item \textbf{Existence:} A vector field $X$ satisfying both equations may not exist.
    \item \textbf{Uniqueness:} If a solution $X$ exists, it is not unique, as it is defined only up to the addition of any vector field $Y \in \mathcal{V} := \ker \omega \cap \ker \tau$.
\end{enumerate}
\end{definition}
\noindent These two problems identify, again, two classes of pre-cosymplectic systems:
\paragraph{Inconsistent Hamiltonian Systems.}
A system is \textit{inconsistent} if the existence condition fails. In this case, a cosymplectic generalization of the pre-symplectic constraint algorithm \cite{chinea} can be used to find the submanifold where consistent dynamics can be defined. \\
%%%
\noindent The algorithm proceeds iteratively too. We define $M_0 := M$ and define the first constraint manifold $M_1$ as the locus of points where the equations are compatible with tangency conditions:
\be
M_1 := \{ m \in M_0 \mid \exists X \in \T_m M_0 \text{ with } (i_X \omega)_m = 0 \text{ and } \tau_m(X) = 1 \} \,.
\ee
%%%
Assuming $M_1$ is a smooth submanifold, the algorithm imposes solutions of \eqref{Eq: Pre-cosymplectic Reeb eq} on $M_1$ to be tangent to $M_1$. Eventually, at each step $k \ge 2$, one finds the submanifold $M_{k} \subset M_{k-1}$:
\be
M_{k} := \{ m \in M_{k-1} \mid \exists X \in \T_m M_{k-1} \text{ with } (i_X \omega)_m = 0 \text{ and } \tau_m(X) = 1 \} \,.
\ee
%%%
\noindent As for the pre-symplectic case, if the algorithm stabilizes, we denote by $\omega_f$ and $\tau_f$ the restrictions of $\omega$ and $\tau$ to $M_f$. 
%%%
Then, $(M_f, \omega_f, \tau_f)$ is a consistent pre-cosymplectic manifold which has Reeb dynamics defined.
\paragraph{Consistent Hamiltonian Systems.}
A system is consistent if it admits global Reeb dynamics. 
%%%
This corresponds to a system that either started consistent or is the result $(M_f, \omega_f, \tau_f)$ of applying the constraint algorithm. \\
%%%
\noindent As for the pre-symplectic case, the dynamics can still be modified by any vector field $Y$ taking values in the characteristic distribution $\mathcal{V} = \ker \omega_f \cap \ker \tau_f$ and one can regularize the system using a cosymplectic version of the coisotropic embedding theorem \cite{zoo}. 
%%%
}

\begin{remark}[\textsc{Construction of the cosymplectic thickening}]
\label{Rem: Construction of cosymplectic thickening}
Let $\mathcal{V} := \ker \omega \cap \ker \tau$. Since it is the intersection of the kernels of two closed forms, it is integrable in the sense of Frobenius. Let $(x^a, f^A)$ denote adapted coordinates to the foliation. {Since, by construction, there is a Reeb vector field}, we may further specify coordinates adapted to the pair $(\omega, \tau)$ obtaining coordinates $(q^a, p_a, f^A, t)$ such that 
\be
\omega = \dd q^a \wedge \dd p_a \,, \qquad \tau = \dd t \,.
\ee
Then, the regularized structure is defined on the bundle $\mathcal{V}^\ast \to M$, and is defined in general using an almost-product structure complementing the distribution $\mathcal{V}$. In general, this will be defined by a projector
\be
P = P^A \otimes \frac{\partial}{\partial f^A} =  \left(\dd f^A - P^A_a \dd q^a - P^{Aa} \dd p_a - Q^A \dd t \right) \otimes \frac{\partial}{\partial f^A} \,.
\ee
This projector, as in the symplectic case, induces an embedding 
\be
\mathfrak{j}^P \colon \mathcal{V}^\ast \to \T^\ast M\,, \qquad \theta \mapsto P^\ast \theta \,.
\ee
Then, if $\theta_M$ denotes the tautological $1$-form on $\T^\ast M$, we define $\vartheta^P:= (\mathfrak{j}^P)^\ast \theta_M$ and let
\be
\widetilde{\omega} = \pi^*\omega + \dd \vartheta^P \,, \qquad \widetilde{\tau} = \pi^\ast \tau \,,
\ee
where $\pi \colon \mathcal{V}^\ast \to M$ denotes the projection. Locally, employing coordinates $(q^a, p_a, t, f^A, \mu_A)$ on $\mathcal{V}^\ast$, this reads as 
\be
\widetilde{\omega} = \dd q^a \wedge \dd p_a + \dd \mu_A \wedge P^A + \mu_A \dd P^A \,, \qquad \widetilde{\tau} = \dd t \,.
\ee
\end{remark}
However, unlike the symplectic case, this embedding is not unique.
%%%
Recently, the authors studied all possible coisotropic embeddings, and uniqueness is guaranteed by a choice of Reeb vector field tangent to the final constraint and the orbits of the Reeb vector field on the extension. More particularly:
\begin{theo}[\textsc{Uniqueness of pre-cosymplectic coisotropic embedding}]\cite{zoo}
\label{thm:uniqueness_coisotropic_precosymplectic}
Let $(M, \omega, \tau)$ be a pre-cosymplectic manifold of constant rank and $\mathfrak{i}_j \colon M \hookrightarrow \widetilde{M}_j$ be coisotropic embeddings into cosymplectic manifolds $(\widetilde{M}_j, \widetilde{\omega}_j, \widetilde{\tau}_j)$, for $j = 1,2$. Then:
\begin{itemize}
    \item There are neighborhoods $U_1$ and $U_2$ of $M$ in $\widetilde{M}_1$ and $\widetilde{M}_2$ and a diffeomorphism $\psi \colon U_1 \to U_2$.
    \item Furthermore, if the Reeb vector fields $R_1$ and $R_2$ of $\widetilde{M}_1$ and $\widetilde{M}_2$ (which are tangent to $M$) coincide on $M$, we have that the diffeomorphism $\psi$ can be chosen such that its pushforward
    \be
    \psi_\ast \colon \T\widetilde{M}_1|_M \longrightarrow \T\widetilde{M}_2|_M
    \ee
    defines an isomorphism of cosymplectic vector bundles.
    \item Finally, if there is a diffeomorphism $\psi_0 \colon U_1 \to U_2$ that is the identity on $M$ and satisfies $(\psi_0)_\ast R_1 = R_2$ ({namely, both Reeb vector fields have the same orbits}), $\psi$ can be chosen to be a cosymplectomorphism.
\end{itemize}
\end{theo}
\begin{remark}
In short, coisotropic embeddings of a pre-cosymplectic manifold $(M, \omega, \tau)$ (in particular, $(M_f, \omega_f, \tau_f)$) are unique topologically. Furthermore, if in advance we fix a Reeb vector field on $M_f$, then the coisotropic embedding is unique "on $M$". Finally, if the Reeb dynamics of both thickenings are conjugate, the embeddings are neighborhood equivalent.
\end{remark}

\subsection{Coisotropic regularization of non-autonomous Lagrangian systems}

{
Let $\pi \colon \mathbf{Q} \to \mathbb{R}$ be a fiber bundle with standard fiber $Q$, which will denote the configuration manifold. As in the autonomous case, we first formally define the geometric structures associated with any non-autonomous Lagrangian $L \in C^\infty(J^1 \pi)$.
\begin{definition}[\textsc{Second Order Differential Equation (SODE)}]
\label{Def: SODE non-autonomous}
A vector field $X \in \mathfrak{X}(J^1 \pi)$ is a \textbf{Second Order Differential Equation (SODE)} field if it correctly relates the position, velocity, and time coordinates. 
%%%
\noindent In natural local coordinates $(q^i, \dot{q}^i, t)$, a general SODE vector field takes the form:
\be
X \,=\, \frac{\partial}{\partial t} + \dot{q}^i \frac{\partial}{\partial q^i} + X^i(q, \dot{q}, t) \frac{\partial}{\partial \dot{q}^i} \,.
\ee
%%%
Intrinsically, this condition is expressed by the two equations:
\be \label{Eq: SODE condition non-autonomous}
(\dd t)(X) \,=\, 1 \,, \quad \text{and} \quad S(X) \,=\, 0 \,,
\ee
where $S$ is the vertical endomorphism on $J^1 \pi$ (see \cref{def:geometry_jet_bundle}).
\end{definition}
\begin{definition}[\textsc{Regular Non-autonomous Lagrangian system}]
\label{Def: Regular Non-autonomous Lagrangian}
A \textbf{non-autonomous Lagrangian system} is a pair $(\mathbf{Q}, L)$, where $L \colon J^1 \pi \to \mathbb{R}$. We define:
\begin{itemize}
    \item The \textbf{Poincaré-Cartan 1-form} $\theta_L := L \dd t + \dd_S L = L \dd t + i_S \dd L$, locally reading:
    \be
    \theta_L \,=\, \left(L - \frac{\partial L}{\partial \dot{q}^i} \dot{q}^i\right)\dd t + \frac{\partial L}{\partial \dot{q}^i} \dd q^i \,.
    \ee
    \item The \textbf{Lagrangian 2-form} $\omega_L := -\dd \theta_L$.
\end{itemize}
%%%
A system is \textbf{regular} if the pair $(\omega_L, \dd t)$ defines a cosymplectic structure on $J^1 \pi$. This happens if and only if the Hessian matrix $\left(W_{ij} = \frac{\partial^2 L}{\partial \dot{q}^i \partial \dot{q}^j}\right)$ is non-singular. 
%%%
In this case, the Euler-Lagrange equations correspond precisely to the unique Reeb dynamics of the cosymplectic manifold $(J^1 \pi, \omega_L, \dd t)$.
\end{definition}
\noindent The following theorem extends \cref{Thm: Helmholtz Conditions} to the cosymplectic context. 
\begin{theo}[\textsc{Characterization of Lagrangian 2-forms}] Let $\pi \colon \mathbf{Q} \longrightarrow \mathbb{R}$ be a fiber bundle. Then, a $2$-form $\omega \in \Omega^2(J^1 \pi)$ is locally a Lagrangian $2$-form if and only if the following conditions hold:
\begin{itemize}
    \item It is closed: $\dd \omega = 0$.
    \item It satisfies $i_{S} \omega = 0$, where $S$ is the vertical endomorphism (equivalently, $\omega(S(X), Y) = \omega(S(Y), X)$).
\end{itemize}
\end{theo}
\begin{proof} Let us introduce fibered coordinates on $\mathbf{Q}$, which we denote by $(q^i, t)$. Since $\omega$ is closed, it is locally exact so that 
\[
\omega = \dd( A_j \dd q^j + B_j \dd \dot{q}^j + C \dd t)\,.
\]
By computing $i_S\omega$, and employing the same argument as in \cref{Thm: Helmholtz Conditions}, we may reduce it to the case where $B_j = 0$, in which case the contraction reads as 
\[
\pdv{A_j}{\dot{q}^i} \left(\dd q^i - \dot{q}^i \dd t\right) \wedge \dd q^j + \pdv{C}{{\dot{q}^i}} \left(\dd q^i - \dot{q}^i \dd t\right) \wedge \dd t\,.
\]
For this to vanish, two conditions need to hold:
\[
\pdv{A_j}{\dot{q}^i} = \pdv{A_i}{\dot{q}^j} \qquad \text{and} \qquad \pdv{A_j}{\dot{q}^i} \dot{q}^i + \pdv{C}{\dot{q}^j} = 0\,.
\]
The first condition implies the existence of $L$ such that $A_j = \pdv{L}{\dot{q}^j}$. Rearranging the second term, this implies that there is a function $F = F(q, t)$ such that 
$
C = L - \pdv{L}{\dot{q}^i} \dot{q}^i + F\,.
$ By making the change $\widetilde L = L +F$, we have that 
\[
\omega = \dd \left( \pdv{\widetilde L}{\dot{q}^i} \dd q^i + \left(\widetilde L - \pdv{\widetilde L}{ \dot{q}^i} \dot{q}^i \right)\dd t\right)\,,
\]
so that it is Lagrangian. Necessity follows trivially.
\end{proof}
\begin{definition}[\textsc{Degenerate Non-autonomous Lagrangian systems}]
\label{Def: Degenerate Non-autonomous Lagrangian}
A non-autonomous Lagrangian system $(\mathbf{Q}, L)$ is \textbf{degenerate} (or \textbf{singular}) if the pair $(\omega_L, \dd t)$ is pre-cosymplectic (i.e., the Hessian matrix $W_{ij}$ is singular).
%%%
A degenerate non-autonomous Lagrangian system is precisely a \textbf{pre-cosymplectic Hamiltonian system} $(J^1 \pi, \omega_L, \dd t)$, but one which carries the "extra" kinematic constraint that its physical dynamics must be a SODE field.
\end{definition}
\noindent As in the autonomous case, this pre-cosymplectic system $(J^1 \pi, \omega_L, \dd t)$ can be either inconsistent or consistent.
}

\paragraph{Inconsistent non-autonomous Lagrangian systems.}
In the presence of an inconsistent system, we may initiate the analogue of the constraint algorithm in the non-autonomous Lagrangian setting. This is obtained by requiring, at each step, the following conditions. Set $F_0 = J^1 \pi$. Then, we iteratively define $F_{k+1}$ to be the set of points $m \in F_k$ (where $F_k$ is assumed to be a submanifold) such that there exists a tangent vector {$X_m \in \T_m F_k$ satisfying
\begin{itemize}
    \item \textit{Reeb condition:} $i_{X_m} \omega_{L} = 0$ and $(\dd t)(X_m) = 1$.
    \item \textit{SODE condition:} $S(X_m) \,=\,0$.
\end{itemize}
As in the autonomous case, if the algorithm converges at a finite step, it finds a final constraint manifold $F_f$ where the Lagrangian system becomes consistent, and which is, in general, a subset of the Hamiltonian one, $F_f \subseteq M_f$.
%%%
Now, $(F_f, \omega_L, \dd t)$ may have gauge ambiguity, in which case we can select a product structure adapted to $\mathcal{V} = \ker \omega_\mathcal{L} \cap \ker \dd t$ and apply the cosymplectic coisotropic embedding theorem to remove the ambiguity in the equations. 
%%%
In general $F_f$ fails to be a jet manifold (or, more generally, to have a jet structure), so that there is little to no hope that the regularization $\widetilde F \longrightarrow F_f$ inherits a jet structure making the regularized system Lagrangian.}

\begin{remark}
A similar discussion applies to that of \cref{remark:Gotay_constraint_algorithm}. The geometric version of the constraint algorithm for time-dependent Lagrangians was developed in \cite{chinea}, generalizing the procedure presented, together with the SODE construction. The same relation applies: In general, $S_f$ , namely the final submanifold of the non-autonomous version of the constraint algorithm is a submanifold of $F_f$.
\end{remark}

\paragraph{Consistent non-autonomous Lagrangian systems.} 
When in the presence of a consistent Lagrangian $L \in C^\infty(J^1 \pi)$, where $\pi \colon \mathbf{Q} \longrightarrow \mathbb{R}$ denotes the configuration bundle, we have the existence of global Reeb dynamics which are of second-order, namely we have that there exists $X \in \mathfrak{X}(J^1 \pi)$ satisfying 
\[
i_X \omega_L = 0\,, \qquad (\dd t)(X) = 1\,, \qquad \text{and} \qquad S(X) = 0\,.
\]
However, if $\{0\} \neq \ker \omega_L \cap \ker \dd t$, the dynamics {fails to be uniquely defined}, yielding gauge ambiguities. 
%%%
After applying the pre-cosymplectic coisotropic embedding theorem, we obtain a cosymplectic manifold $(\widetilde M, \widetilde\omega,\widetilde \tau)$ in which $(J^1 \pi, \omega, \tau)$ embeds as a coisotropic submanifold. 
%%%
Here we may wonder whether $\widetilde M$ admits a jet structure in such a way that the cosymplectic manifold $(\widetilde M, \widetilde \omega,\widetilde \tau)$ arises from an extended Lagrangian 
\[
\widetilde L \colon \widetilde M \longrightarrow \mathbb{R}\,.
\]
More particularly, we are looking for:
\begin{itemize}
    \item A fiber bundle $\widetilde{\pi} \colon \widetilde{\mathbf{Q}} \longrightarrow \mathbb{R}$, together with a fiber bundle embedding $\mathbf{Q} \hookrightarrow \widetilde {\mathbf{Q}}$. Denote by $\widetilde S$ the vertical endomorphism on $J^1 \widetilde \pi$. 
    \item A Lagrangian $\widetilde L \colon \widetilde M \longrightarrow \mathbb{R}$ such that it restricts to the original Lagrangian $L$ under the previous inclusion.
    \item A diffeomorphism $\alpha \colon \widetilde M \longrightarrow J^1 \widetilde \pi$ in such a way that the cosymplectic structure $(\widetilde \omega, \widetilde \tau)$ is precisely the cosymplectic structure obtained as follows 
    \[
    \widetilde \omega = \dd \left(\widetilde L \dd t + i_{\alpha^\ast \widetilde S} \dd \widetilde L \right)\,, \qquad \tau = \dd t\,.
    \]
\end{itemize}

\subsection{Existence and uniqueness of Lagrangian regularization}

The objective of this section is to prove the existence of a non-autonomous Lagrangian regularization, under some conditions on the gauge ambiguities of $L$. We also discuss the matter of uniqueness, and { although global uniqueness is not guaranteed, as a plethora of Lagrangians may be considered, we prove that any tangent structure on a particular cosymplectic regularization $\widetilde M$ must be `isomorphic on M' to the one we build, given that the Reeb vector fields coincide. Namely, as in the autonomous case, the first order germ of the extension is unique.}\\

{\noindent The main assumption that we will make to endow the regularization with a Lagrangian structure is that the characteristic distribution $\mathcal{V} = \ker \omega_L \cap \ker \dd t$ on $J^1 \pi$ is the complete lift of a vertical (completely integrable) distribution on $\pi \colon \mathbf{Q} \to \mathbb{R}$. This may be thought of as a time-dependent generalization of the case presented in \cref{Sec:Regularization of autonomous systems}. That is, $\mathcal{V} = \undertilde{K}^{C}$.
%%%
Let $\mathbf{Q}$ have local fibered coordinates $(t, x^a, f^A)$, where $x^a$ are coordinates on the leaves of the regular foliation $\mathcal{F}_{\undertilde{K}}$ induced by $\undertilde{K}$, and $f^A$ parameterize the fibers (the distribution $\undertilde{K}$).
%%%
The first jet bundle $J^1 \pi$ has natural coordinates $(t, x^a, f^A, \dot{x}^a, \dot{f}^A)$.
%%%
Under this hypothesis, the characteristic distribution $\mathcal{V}$ is locally spanned by:
\be
\text{span}\left\{ \left(\frac{\partial}{\partial f^A}\right)^C \,=\, \frac{\partial}{\partial f^A},\, \left(\frac{\partial}{\partial f^A}\right)^V \,=\, \frac{\partial}{\partial {\dot{f}}^A}\right\} \,.
\ee
%%%
\noindent The cosymplectic thickening $(\widetilde{M}, \widetilde{\omega}, \widetilde{\tau})$ is constructed as a neighborhood of the zero section in the dual bundle $\mathcal{V}^* \to J^1 \pi$, as described in \cref{Rem: Construction of cosymplectic thickening}.
%%%
}
As in the autonomous case, we assume this thickening coincides with the whole $\mathcal{V}^*$ by assuming that an almost product structure $P$ with a vanishing Nijenhuis tensor can be chosen.
%%%
\noindent On the other hand, the thickened space $\widetilde{M} = \mathcal{V}^*$ can be identified as the cotangent bundle of the foliation $\mathcal{F}_{\mathcal{V}}$ generated by $\mathcal{V}$:
\begin{proposition}
The following canonical isomorphism exists:
\be \label{Eq: thickening=cotangent bundle foliation non-autonomous}
\widetilde{M} \,=\, \mathcal{V}^* \,\simeq\, \mathscr{T}^* \mathcal{F}_{\mathcal{V}} \,:=\, \bigsqcup_{F \in \mathcal{F}_{\mathcal{V}}} \T^* F \,,
\ee
where $F$ denotes a leaf of $\mathcal{F}_{\mathcal{V}}$.
%%%
\begin{proof}
The proof is completely analogous to the symplectic case presented in \cref{Sec: Degenerate Lagrangian systems}, taking into account that the leaves $F$ are exactly the maximal integral manifolds of the characteristic distribution $\mathcal{V}$.
\end{proof}
\end{proposition}
\noindent The coordinates of $\widetilde{M}$ are $(t, x^a, f^A, \dot{x}^a, \dot{f}^A, {\mu_f}_A, {\mu_{\dot{f}}}_A)$, where $({\mu_f}_A, {\mu_{\dot{f}}}_A)$ are the fiber coordinates dual to the kernel generators $\{\frac{\partial}{\partial f^A}, \frac{\partial}{\partial {\dot{f}}^A}\}$.
%%%
We now define a new extended configuration bundle $\widetilde{\pi} \colon \widetilde{\mathbf{Q}} \to \mathbb{R}$.
%%%
We identify $\widetilde{\mathbf{Q}}$ as the cotangent bundle of the foliation $\mathcal{F}_{\undertilde{K}}$, denoted $\widetilde{\mathbf{Q}} := \mathscr{T}^* \mathcal{F}_{\undertilde{K}} \,\equiv\, \undertilde{\mathbf{K}}^*$.
%%%
The manifold $\widetilde{\mathbf{Q}}$ has local coordinates $(t, x^a, f^A, \mu_A)$.
%%%
The first jet bundle of this new space is $J^1 \widetilde{\pi}$, with local coordinates $(t, x^a, f^A, \mu_A, \dot{x}^a, \dot{f}^A, \dot{\mu}_A)$.
%%%

\begin{remark} 
Again, as pointed out in \cref{remark:vanishing_Nijenhuis_tensor}, we can work with an almost product structure without vanishing Nijenhuis tensor, simply restricting to an open subset to obtain regularity.
\end{remark}

\begin{proposition}[\textsc{Tulczyjew isomorphism for jets}]
\label{prop:Tulczyjew_iso_jets}
There exists a canonical isomorphism $\alpha$ that relates $J^1\widetilde{\pi}$ to the thickened space $\widetilde{M}$:
\be
\alpha \colon J^1\widetilde{\pi} \to \widetilde{M} \,.
\ee
\begin{proof}
The proof is an immediate generalization of the theory presented in \cref{Subsec: A Tulczyjew isomorphism for foliations}. Let $j^1_s \mu \in J^1 \widetilde{\pi}$ denote a jet, where $\mu \colon \mathbb{R} \to \undertilde{\mathbf{K}}^*$ is a section, and $X^{C} \in \mathfrak{X}(J^1 \pi)$ denote a vector field tangent to $\mathcal{V} = \undertilde{K}^{C}$, where $X$ is a vertical vector field taking values in $\undertilde{K}$. Then, it is enough to define the pairing:
\be
\langle j^1_s \mu,\, X^{C} \rangle \,:=\, \frac{\dd}{\dd t}\bigg|_{t = s} \langle \mu,\, X \rangle \,.
\ee
In local coordinates, the isomorphism reads exactly as in the autonomous case \eqref{Eq: Tulczyjew foliations}, simply carrying the time coordinate $t$:

\be
\alpha (t, x^a, f^A, \mu_A, \dot{x}^a, \dot{f}^A, \dot{\mu}_A) \,=\, (t, x^a, f^A, \dot{x}^a, \dot{f}^A, {\mu_f}_A = \dot{\mu}_A, {\mu_{\dot{f}}}_A = \mu_A) \,.
\ee

\end{proof}
\end{proposition}

\begin{remark}
As thoroughly discussed in \cref{Sec: Degenerate Lagrangian systems} for the autonomous scenario, the construction of the regularized 2-form $\widetilde{\omega}$ based solely on the choice of an almost-product structure $P$ on $J^1 \pi$ cannot be adapted so that $\widetilde{\omega}$ is a Lagrangian form with respect to the canonical jet structure on $J^1 \widetilde{\pi}$. There is a fundamental incompatibility between the standard coisotropic approach and the SODE geometry, which carries over identically to this time-dependent setting.
\end{remark}

{\noindent Therefore, to regularize the system, we will proceed exactly as introduced in the pre-symplectic case. 
%%%
We construct the regularized cosymplectic 2-form $\widetilde{\omega}$ on $J^1 \widetilde{\pi}$ by adding a correction term that is Lagrangian by construction. 
%%%
This term takes the form $-\dd\dd_{\widetilde{S}} F$, where $F \in \mathcal{C}^\infty(J^1 \widetilde{\pi})$ is a globally defined smooth function. 
%%%
We then define the extended Lagrangian as:
\be
\widetilde{L} \,=\, L + F \,,
\ee
so that we obtain a global Lagrangian cosymplectic structure $(\widetilde{\omega} = -\dd \dd_{\widetilde{S}} \widetilde{L}, \tau = \dd t)$.
%%%
\noindent The definition of the function $F$, exactly as in the autonomous case, is not canonical and depends on the choice of two specific ingredients:}
\begin{itemize}
    \item An Ehresmann connection $\nabla$ on the bundle $\undertilde{K}^\ast \longrightarrow \mathbf{Q}$ (again given by a splitting of the tangent bundle in vertical and horizontal vectors
    $
    \T \undertilde{K}^\ast  = \mathcal{V} \oplus \mathcal{H}_\nabla
    $) in such a way that the splitting at $\mathbf{Q}$ is the canonical splitting 
    \[
    \T \undertilde{K}^\ast  |_{\mathbf{Q}} = \mathcal{V} \oplus \T \mathbf{Q}\,.
    \]
    \item An almost product structure $P$ on $\mathbf{Q}$, which complements the distribution $\undertilde K$.
\end{itemize}

\begin{remark}[\textsc{Coordinate expressions}]
\label{remark:coordinate_expression_nabla_P}
Locally, we express the components of the connection as
\[
\mathcal{H}_\nabla = \operatorname{span} \left \{ \pdv{t} + \Gamma_A \pdv{\mu_A}\,, \pdv{q^i} + \Gamma_{iA} \pdv{\mu_A}\,, \pdv{f^B} + \Gamma_{BA} \pdv{\mu_A} \right\}\,.
\]
The condition on $\nabla$ inducing the canonical splitting at the zero section is reflected on the $\Gamma$'s vanishing at $\mathbf{Q}$ (again identified via the zero section).
And we express the components of the almost product structure as 
\[
P = P^A \otimes \pdv{f^A} = \left(\dd f^A - Q^A \dd t - P^A_i \dd q^i \right) \otimes \pdv{f^A}\,.
\]
\end{remark}

\noindent  Generalizing the construction made in \cref{Sec: Degenerate Lagrangian systems}, consider the following diagram
% https://q.uiver.app/#q=WzAsOCxbMSwwLCJKXjEgXFx3aWRldGlsZGUgXFxwaSJdLFswLDAsIihcXHVuZGVydGlsZGV7S31eXFx1cGFycm93KV5cXGFzdCJdLFsyLDAsIlxcbWF0aGNhbHtWfSBcXG9wbHVzIFxcbWF0aGNhbHtIfV9cXG5hYmxhID0gXFxUIFxcdW5kZXJ0aWxkZXtLfV5cXGFzdCJdLFsyLDEsIlxcbWF0aGNhbHtWfSJdLFsyLDIsIlxcdW5kZXJ0aWxkZXtLfV5cXGFzdCJdLFswLDEsIkpeMSBcXHBpIl0sWzEsMSwiXFxUIFxcbWF0aGJme1F9Il0sWzEsMiwiXFx1bmRlcnRpbGRle0t9Il0sWzAsMiwiaV97XFx3aWRldGlsZGVcXHBpfSIsMCx7InN0eWxlIjp7InRhaWwiOnsibmFtZSI6Imhvb2siLCJzaWRlIjoiYm90dG9tIn19fV0sWzIsMywicF9cXG1hdGhjYWx7Vn0iXSxbMyw0XSxbMSw1LCJcXHRhdSIsMl0sWzUsNiwiaV9cXHBpIiwyLHsic3R5bGUiOnsidGFpbCI6eyJuYW1lIjoiaG9vayIsInNpZGUiOiJ0b3AifX19XSxbNiw3LCJQIiwyXSxbMCwxLCJcXGFscGhhIiwyXV0=
\[\begin{tikzcd}[cramped]
	{(\undertilde{K}^C)^\ast} & {J^1 \widetilde \pi} & {\mathcal{V} \oplus \mathcal{H}_\nabla = \T \undertilde{K}^\ast} \\
	{J^1 \pi} & {\T \mathbf{Q}} & {\mathcal{V}} \\
	& {\undertilde{K}} & {\undertilde{K}^\ast}
	\arrow["\tau"', from=1-1, to=2-1]
	\arrow["\alpha"', from=1-2, to=1-1]
	\arrow["{i_{\widetilde\pi}}", hook', from=1-2, to=1-3]
	\arrow["{p_\mathcal{V}}", from=1-3, to=2-3]
	\arrow["{i_\pi}"', hook, from=2-1, to=2-2]
	\arrow["P"', from=2-2, to=3-2]
	\arrow[from=2-3, to=3-3]
\end{tikzcd}\,,\]
where the map $\tau \colon (\undertilde{K}^C)^\ast \longrightarrow J^1 \pi$ is the canonical projection and the arrow $\mathcal{V} \longrightarrow \undertilde{K}^\ast$ is the usual identification of the vertical bundle with the fiber of a vector bundle. Define the map 
\[
F^{P, \nabla} \colon J^1 \widetilde \pi \longrightarrow \mathbb{R}
\]
by
\[
F^{P,\nabla}(\xi) = \langle (P \circ i_\pi \circ \tau \circ \alpha)(\xi), (p_\mathcal{V} \circ i_{\widetilde\pi}) (\xi) \rangle\,,
\]
where $\xi \in J^1 \widetilde \pi$, and $\langle \cdot , \cdot \rangle$ denotes the natural pairing between $\undertilde{K}$ and $\undertilde{K}^\ast$.

\begin{remark}[\textsc{Local expression of} $F^{P, \nabla}$] Using the coordinate components of $\nabla$ and $P$ from \cref{remark:coordinate_expression_nabla_P}, we have that 
\[
F^{P, \nabla} = \left( \dot{\mu}_A - \Gamma_A - \dot{q}^i \Gamma_{iA} - \dot{f}^B \Gamma_{BA}\right) \cdot \left( \dot{f}^A - Q^A - \dot{q}^jP^A_j\right)\,.
\]
\end{remark}

\noindent We then have the following:

\begin{theo}[\textsc{Lagrangian coisotropic embedding}]
\label{thm:Lagrangian_pre-cosymplectic_regularization}
Let $L \colon J^1 \pi \longrightarrow \mathbb{R}$ be a singular Lagrangian, where $\pi \colon \mathbf{Q} \longrightarrow \mathbb{R}$ denotes a configuration bundle. Suppose that $L$ is consistent and that the characteristic distribution $K = \ker \omega_L \cap \ker \dd t$ is the complete lift of a vertical distribution $\undertilde K$ on $\mathbf{Q}$. Then given an Ehresmann connection $\nabla$ and an almost product structure $P$ as above the embedding 
\[
J^1 \pi \hookrightarrow (\undertilde{K}^C)^\ast \cong J^1 \widetilde \pi
\]
is a coisotropic embedding on a neighborhood of $J^1 \pi$ for the cosymplectic structure $(\omega_{\widetilde L}, \dd t)$, where $\widetilde L = L + F^{P, \nabla}$.
\end{theo}
\begin{proof} Follows with a similar discussion as in the symplectic case. We will first show that $\T \left( J^1 \widetilde \pi\right)\big |_{J^1 \pi}$ is a cosymplectic vector bundle, so that the pair $(\widetilde \omega, \dd t)$ defines a cosymplectic structure on some neighborhood of $J^1 \pi$. Indeed
\begin{align}
    \dd_{\widetilde S} \widetilde L =& \dd_S L + \dot{f}^A \dd \mu_A + \dot{\mu}_A \dd f^A\\
    &- \left(\Gamma_{iA}(\dot{f}^A - Q^A - \dot{q}^j P^A_j) + P^{A}_i(\dot{\mu}_A - \Gamma_A - \dot{q}^i \Gamma_{iA} - \dot{f}^B \Gamma_{BA})\right) \dd q^i \\
    & - \left(\Gamma_{BA} (\dot{f}^A - Q^A - \dot{q}^j P^A_j)  + \Gamma_A + \dot{q}^i \Gamma_{iA} + \dot{f}_B \Gamma_{BA}\right)\dd f^A\\
    &- (Q^A + \dot{q}^j P^A_j) \dd \mu_A\,.
 \end{align}
So that
\[
\dd \dd_{\widetilde S} \widetilde L = \omega_L + \dd \dot{f}^A \wedge \dd  \mu_A + \dd {\dot{\mu}_A} \wedge \dd f^A + \text{(semi-basic terms)}\,.
\]
Notice that the first three terms on the right hand side, together with the $1$-form $\dd t$, define a cosymplectic structure. 
%%%
Since adding semi-basic terms (with respect to the projection onto $\mathbf{Q}$) does not change regularity, we have that $\T \left( J^1 \widetilde \pi\right)\big |_{J^1 \pi}$ is a cosymplectic vector bundle. Finally, notice that it is a coisotropic embedding as, again, 
\[
\dd_{\widetilde S} L\big |_{J^1 \pi} = \dd_S L\,.
\]
\end{proof}

\noindent Hence, we conclude the discussion of the existence of Lagrangian regularization in the cosymplectic scenario. Again, we may ask about its uniqueness. A similar discussion applies. However, we encounter one natural obstruction to uniqueness, which is the arbitrariness of the Reeb vector field on $J^1 \pi$. Indeed, recall (\cref{thm:uniqueness_coisotropic_precosymplectic}) that uniqueness 'on $J^1 \pi$' for the coisotropic embedding is guaranteed provided the Reeb vector field is fixed. Nevertheless, once this is taken into account, we have that the Lagrangian regularization is rigid to first order:

\begin{theo}[\textsc{Uniqueness of Lagrangian coisotropic embedding to first order}]
\label{thm:precosymplectic_Lag_uniqueness}
Let $(\widetilde M_i, \widetilde \omega_i, \tau_i)$, be cosymplectic regularizations of $(J^1 \pi, \omega_L, \dd t)$, with $i = 1,2$. Suppose that $\widetilde M_i$ is endowed with a jet structure {$(\widetilde S_i, \tau_i)$} in such a way that $i_{\widetilde S_i} \widetilde \omega_i= 0$, and the embedding 
\[
J^1 \pi \hookrightarrow \widetilde M_i
\]
preserves the jet structure. If the induced Reeb vector fields (which are tangent to $J^1 \pi$) coincide, there exist neighborhoods $U_i$ of $J^1\pi$ in $\widetilde M_i$ and a diffeomorphism 
\[
\psi \colon U_1 \longrightarrow U_2
\]
that is the identity on $J^1 \pi$ such that the induced map
\[
\psi_\ast \colon \T \widetilde M_1 \big|_{J^1 \pi} \longrightarrow \T \widetilde M_2 \big |_{J^1 \pi}
\]
preserves all tensors, namely $\psi_\ast (\widetilde \omega_1, \widetilde \tau_1, \widetilde S_1) = (\widetilde \omega_2, \widetilde \tau_2, \widetilde S_2)$ on $J^1 \pi$.
\end{theo}
\begin{proof}Denote by $\mathcal{V} = \ker \omega_L \cap \ker \dd t$ the characteristic distribution. Let $R$ denote the induced Reeb vector field on $J^1 \pi$ (by any of the embeddings), and let $W$ be a distribution on $J^1 \pi$ such that $\T J^1 \pi = \mathcal{V} \oplus \operatorname{span} \left \{ R \right \} \oplus W$ and such that $S(W) \subseteq W$ (this can be achieved by taking a complete lift). Then, $(W, \omega_L)$ is a symplectic vector bundle, and its $(\widetilde \omega_i, \widetilde \tau_i)$-orthogonal \[
W^{\perp, \widetilde \omega_i, \widetilde \tau_i} := \{v \in \T \widetilde M_i \colon i_v \widetilde \tau_i = 0 \text{ and } \widetilde \omega_i\left(v, W \right) = 0\}
\]
is as well and satisfies $\widetilde S(W^{\perp,  \widetilde \omega_i, \widetilde \tau_i}) \subseteq W^{\perp}$. Indeed, the first property is an elementary consequence of symplectic linear algebra and the second follows by the compatibility of $\widetilde S$ with $\widetilde \omega_i$, as we have 
\[
\widetilde \omega_i \left( \widetilde S(W^{\perp,  \widetilde \omega_i, \widetilde \tau_i}), W\right) = \widetilde \omega_i(\widetilde S(W) , W^{\perp,  \widetilde \omega_i, \widetilde \tau_i})  = \widetilde \omega_i(S(W), W^{\perp,  \widetilde \omega_i, \widetilde \tau_i}) = 0\,,
\]
since we are requiring $S(W) \subseteq W$.\\

\noindent Notice that, by construction, $\mathcal{V} \subseteq W^{\perp, \widetilde \omega_i}$ and $\mathcal{V}$ is isotropic. Let us show that it is actually Lagrangian, by computing its dimension. Suppose $\dim J^1 \pi = 2(n+r) + 1$, where $2r$ is the rank of $\mathcal{V}$. Then, since $J^1 \pi \hookrightarrow \widetilde M_i$ is coisotropic, 
\[
2r = \rank( \T (J^1 \pi)^{\perp, \widetilde \omega_1, \widetilde \tau_i}) = \dim \widetilde M_i - \dim J^1 \pi\,,
\]
we conclude $\dim \widetilde M_i = \dim J^1 \pi + 2r$. Now, $\rank W = 2n$, so that $\rank W^{\perp, \widetilde \omega_i, \widetilde \tau_i} = 2n + 4r - 2n = 4r$. Finally, since $\rank \mathcal{V} = 2r$, we conclude that it is Lagrangian. By usual techniques (see \cref{Lemma: Lagrangian complement trick}), we can build a vector bundle symplectomorphism 
\[
\Phi \colon W^{\perp, \widetilde \omega_i, \widetilde \tau_i} \longrightarrow \mathcal{V} \oplus \mathcal{V}^\ast\,,
\]
so that we obtain the isomorphism 
\[
\T \widetilde M_i\big|_{J^1 \pi}  = \operatorname{span}\{ R\} \oplus W \oplus W^{\perp, \widetilde \omega_i, \widetilde \tau_i} \cong \operatorname{span}\{ R\} \oplus W \oplus \mathcal{V} \oplus \mathcal{V}^\ast\,,
\]
and such that, by defining $\widehat{S}_i = \Phi \circ \widetilde S_i \circ \Phi^{-1} $, we have $\widehat{S}_i(\mathcal{V}^\ast) \subseteq \mathcal{V}^\ast$. The previous isomorphism of vector bundles is clearly the identity on $\T J^1 \pi = \operatorname{span} \left\{ R\right \} \oplus W \oplus \mathcal{V}$. Now, since 
\[
\operatorname{span}\{ R\} \oplus W \oplus \mathcal{V} \oplus \mathcal{V}^\ast = \T \mathcal{V}^\ast |_{J^1 \pi}\,,
\]
we may choose a neighborhood $V$ of $J^1 \pi$ in $\mathcal{V}^\ast$ (identified as the zero section) and two diffeomorphisms 
$
\psi_i \colon U_i \longrightarrow V\,,
$
where $U_i$ are neighborhoods of $J^1 \pi$ in $\widetilde M_i$, and in such a way that it induces the isomorphism above. Clearly, the above identification makes $(\psi_i)_\ast$ an isomorphism of cosymplectic vector bundles over $J^1 \pi$. It only remains to show that it induces an isomorphism of jet structures as well. This follows by employing the same technique as in \cref{thm:uniqueness_sympelctic_first_order}: we will show that under the previous construction, there exists a unique tensor $\widehat S$ that satisfies the following properties (all consequence of the construction presented):
\begin{itemize}
    \item It is a \emph{jet structure}: $\widehat S^2 $ = 0.
    \item It makes $\omega$ \emph{Lagrangian}: $i_S \omega = 0$.
    \item It \emph{extends} the vertical endomorphism of $J^1 \pi$, which we denote by $S$.
    \item It satisfies $\widehat{S}(\mathcal{V}^\ast) \subseteq\mathcal{V}^\ast$.
\end{itemize}
In particular $(\psi_{i})_\ast \widetilde S_i = \widehat{S}$, and the diffeomorphism $\psi_2^{-1} \circ \psi_1$ necessarily preserves the jet structure on $J^1 \pi$.\\

\noindent Indeed, by taking adapted coordinates $(t, x^a, f^A, \dot{x}^a, \dot{f}^A)$ on $J^1 \pi$ and coordinates $(\mu_A, \dot{\mu}_A)$ in the fibers of $\mathcal{V}^\ast$, we have that the canonical form on $\mathcal{V}^\ast |_{J^1 \pi}$ reads as 
\[
\omega = \omega_L + \dd \mu_A \wedge \dd f^A + \dd \dot{\mu}_A \wedge \dd \dot{f}^A {-(\dot{f}^A  \dd {\mu}_A + Q^A\dd \dot{\mu}_A)\wedge \dd t}\,,
\]
where $R = \pdv{t} + \dot{x}^a \pdv{x^a}+ \dot{f}^A \pdv{f^A} + R^a \pdv{\dot{x}^a} +Q^A\pdv{\dot{f}^A}$ is the chosen Reeb vector field. A general $(1,1)$ tensor $\widehat S$ extending $S =\left( \dd x^a - \dot{x}^a \dd t \right) \otimes \pdv{\dot{x}^a} + (\dd f^A - \dot{f}^A \dd t) \otimes \pdv{\dot{f}^A}$ takes the following expression
\begin{align}
    \widehat{S} =& \left(\dd x^a - \dot{x}^a \dd t + F^{iA} \dd \mu_A + \dot{F}^{iA} \dd \dot{\mu}_A\right) \otimes \pdv{\dot{x}^a}\\
    &+ \left(\dd f^A - \dot{f}^A\dd t + G^{AB} \dd \mu_B + \dot{G}^{AB} \dd \dot{\mu}_B \right) \otimes\pdv{\dot{f}^A} 
    \\
    &+ \left(H^{B}_A \dd \mu_B + \dot{H}^{B}_A \dd \dot{\mu}_B  + H_A \dd t\right) \otimes \pdv{\mu_A}\\
    &+ \left( I^{B}_A \dd \mu_B + \dot{I}^{B}_A \dd \dot{\mu}_B + I_A \dd t\right) \otimes \pdv{\dot{\mu}_A}\,.
\end{align}
It follows by similar computations that the three properties above imply
\[
\widehat S = (\dd x^a - \dot{x}^a \dd t)\otimes \pdv{\dot{x}^a} + (\dd f^A - \dot{f}^A \dd t) \otimes \pdv{\dot{f}^A}  - \dd \dot{\mu}_A \otimes \pdv{\mu_A}\,.
\]
Finally, defining $\psi := \psi_2^{-1} \circ \psi_1$, we have that it preserves the cosymplectic structure by construction and that 
\[
\psi_\ast (\widetilde S_1) = (\psi_2^{-1})_\ast \left( (\psi_1)_{\ast} \widetilde S_1\right) = (\psi_2^{-1})^\ast \widehat S = \widetilde S_2\,,
\]
which shows that it preserves the jet structure as well.
\end{proof}
\subsection{Examples}

\subsubsection{Trivialized bundles}

Here we deal with configurations bundles $\pi \colon\mathbf{Q} \longrightarrow \mathbb{R}$ which are trivialized, namely that we choose a bundle isomorphism $\mathbf{Q} \cong Q \times \mathbb{R} \longrightarrow \mathbb{R}$. Then, this splitting induces a diffeomorphism $J^1 \pi \cong \T Q \times \mathbb{R}$ and, in particular, any Lagrangian $L \colon J^1 \pi \longrightarrow \mathbb{R}$ simply reads as a time-dependent Lagrangian 
\[
L \colon \T Q \times \mathbb{R} \longrightarrow \mathbb{R}.
\]
We study in this section how the (Lagrangian) regularization procedure behaves in the trivialized case. First notice that, by definition, vertical and complete lifts correspond to vertical and complete lifts on $\T Q$ after trivializing. The Poincaré--Cartan form still reads (in natural coordinates $(q^i, \dot{q}^i, t)$) as 
\[
\theta_L = \left( L - \pdv{L}{\dot{q}^i} \dot{q}^i \right)\dd t + \pdv{L}{\dot{q}^i} \dd q \,.
\]

\begin{remark}[\textsc{Hypothesis on the characteristic distribution}]
Recall that to obtain a Lagrangian regularization, we imposed the condition on the characteristic distribution 
\[
\mathcal{K} := \ker \omega_L \cap \ker \dd t
\]
to be the complete lift of a vertical distribution on $\mathbf{Q} \longrightarrow \mathbb{R}$. This may be stated using the trivialized bundle as follows. First notice that the characteristic distribution $\mathcal{K}$ on the trivialized bundle $\T Q \times \mathbb{R} \longrightarrow \mathbb{R}$ may be thought of as time-dependent completely integrable distribution on $\T Q$. Denote by $\mathcal{K}_t$ the distribution at time $t$. Then, the condition on $\mathcal{K}$ to be the complete lift of a vertical distribution $\undertilde{K}_t$ on $\mathbf{Q} \longrightarrow \mathbb{R}$ translates to $\mathcal{K}_t$ being a complete lift of a completely integrable distribution $\undertilde{K}_t$, for every $t$. Incidentally, $\undertilde{K}$ is simply the gluing of all $\undertilde{K}_t$.
\end{remark}

Now, if we are in the case above, we may wonder whether on the bundle $\undertilde{K}^\ast \longrightarrow \mathbb{R}$ we have a natural trivialization. This would be the case if the `time dependent' distribution $\undertilde{K}_t$ is constant but, otherwise, would fail to hold. We could also ask whether, although $\undertilde{K}_t$ is not constant, they are all isomorphic, in the sense that there is a smooth family of diffeomorphisms 
\[
\psi_t \colon Q \longrightarrow Q\,, \qquad \psi_0 = \operatorname{id}_Q
\]
such that $\undertilde{K}_t = (\psi_t)_\ast( K_0)$.  This, again, does not hold in general, as the following example shows:

\begin{example}[\textsc{Lagrangian on trivial bundle with non-constant characteristic distribution}] The following example, albeit artificial, shows the existence of time-dependent Lagrangians on a trivialized bundle \[
L \colon \T Q \times \mathbb{R} \longrightarrow \mathbb{R}
\]
such that $\mathcal{V} = \ker \omega_L \cap \ker \dd t = (\undertilde K)^C$, for certain distribution $\undertilde K$ on $Q \times \mathbb{R}$ which is \emph{not constant}. Here, ``not constant'' means that there is not a time dependent family of diffeomorphisms $\psi_t \colon Q \longrightarrow Q$ such that $\psi_{t}^\ast (\undertilde{K}|_{Q \times \{t\}}) = K|_{Q \times \{0\}}$, essentially forcing the jet bundle point of view presented.

\noindent On $\mathbb{S}^1$, let 
\[
\exp \colon \mathbb{R} \longrightarrow\mathbb{S}^1 \,, \qquad t \mapsto (\cos t, \sin t)
\]
denote the exponential and define $I_t$ to be the image of the interval $\left[- \frac{t^2}{1 + t^2}, \frac{t^2}{1+t^2}\right]$ under $\exp$, namely \[
I_t := \exp\left (\left[-\frac{t^2}{1 + t^2}, \frac{t^2}{1+t^2}\right]\right)\,.
\]
We clearly have that $I_t$ is diffeomorphic to the closed interval $[0,1]$ for $t \neq 0$ and only a point for $t = 0$. Define 
\[
Q_t := (\mathbb{S}^1 \times \mathbb{S}^1) \setminus (I_t \times \{p\})\,,
\]
for a fixed point $p \in \mathbb{S}^1$. Standard techniques of differential topology \cite{milnor} show that there is a family of embeddings 
\[
\psi_t \colon (\mathbb{S}^1 \times \mathbb{S}^1) \setminus (\{p\} \times \{p\}) \longrightarrow \mathbb{S}^1 \times \mathbb{S}^1 \,,
\]
that varies smoothly with $t \in \mathbb{R}$ and such that it defines a diffeomorphism with $Q_t$, for each $t$. Define the following Lagrangian 
\[
\widetilde L \colon \T \left( \mathbb{S}^1 \times \mathbb{S}^1\right) \longrightarrow \mathbb{R}\,, \qquad \widetilde L(\theta_1, \theta_2, \dot{\theta_1}, \dot{\theta_2}) = \frac{(\dot{\theta_1})^2}{2}\,,
\]
where $(\theta_1, \theta_2)$ denote angular (local) coordinates on the torus, and $(\dot{\theta}_1, \dot{\theta}_2)$ denote the induced global coordinates on $\T (\mathbb{S}^1 \times \mathbb{S}^1)$. Let $Q:= (\mathbb{S}^1 \times \mathbb{S}^1) \setminus (\{p\} \times \{p\})$ and define the following time-dependent Lagrangian:
\[
L \colon \T Q \times \mathbb{R} \longrightarrow \mathbb{R} \,, \qquad L(v, t) := \widetilde L((\psi_t)_\ast v)\,.
\]
Locally, after a change of coordinates, the previous Lagrangian is simply $\widetilde L$, but not globally. Indeed, in $t = 0$, the characteristic distribution of $L$ has precisely one non-compact leaf. This no longer holds for $t \neq 0$, which shows that the characteristic distribution $\undertilde{K}$ cannot be made constant after a global time dependent family of diffeomorphisms.
\end{example}

\subsubsection{Degenerate metrics}

Let us deal with the example of degenerate metrics. A canonical example of an autonomous degenerate Lagrangian is that of a degenerate metric, namely a symmetric and positive semidefinite tensor $g$ on a manifold $M$. Given such a tensor, we may study its kinetic energy
\[
L \colon \T M\longrightarrow \mathbb{R} \,, \qquad L(v) := \frac{1}{2}g(v,v)\,.
\]
In the nonautonomous realm we can generalize this in two ways:
\begin{itemize}
    \item A degenerate metric on a manifold $M$, which is time dependent, say $g_t$ for $t \in \mathbb{R}$ and study its time-dependent energy
    \[
    L \colon \T M \times \mathbb{R} \longrightarrow\mathbb{R}\,, \qquad L(v, t):= \frac{1}{2} g_t(v,v)\,.
    \]
    \item Or, given a fiber bundle $\pi \colon\mathbf{Q} \longrightarrow \mathbb{R}$, to work with a (possibly degenerate) metric $\mathbf{g}$ on $\mathbf{Q}$, together with the energy
    \[
    L\colon J^1 \pi \longrightarrow \mathbb{R}\,, \qquad L(j^1_t \gamma) := \mathbf{g} \left(\gamma_\ast \pdv{t}, \gamma_\ast \pdv{t} \right)\,.
    \]
\end{itemize}
The latter has the advantage of including the first as a particular case and, also, allowing for potentials. Indeed, in general, and for fibered coordinates $(q^i, t)$ on $\mathbf{Q}$, we have 
\[
\mathbf{g} := g_{ij} \dd q^i \dd q^j + 2A_i \dd t \dd q^i - 2V \dd t^2\,.
\]
Then, with natural coordinates $(q^i, \dot{q}^i, t)$ on $J^1 \pi$, the Lagrangian reads as
\[
L = \frac{1}{2}g_{ij} \dot{q}^i \dot{q}^j + A_i \dot{q}^i - V\,.
\]
This Lagrangian yields the equations for the movement of a {charged (with charge $1$)} particle on $Q$ (the standard fiber) under a{n electric} potential $V$ {and} {magnetic} \red{potential $A_i \dd q^i$} {in a curved (by the metric $g_{ij}$) space}.  Let us study its regularity. First, notice that its Poincaré--Cartan form is 
\[
\theta_L = L \dd t +\dd_S L = \left(\frac{1}{2} g_{ij} \dot{q}^i \dot{q}^j + V \right) \dd t + (g_{ij} \dot{q}^j +A_i) \dd q^i\,,
\]
so that 
\begin{align}
    \omega_L &= - \dd \theta_L\\
    &= g_{ij} \dd q^i \wedge \dd \dot{q}^j - \pdv{g_{ij}}{q^k} \dot{q}^j \dd q^k \wedge \dd q^i - \pdv{A_i}{q^j} \dd q^j \wedge \dd q^i\\
    &\quad\,- \left( \dd (\frac{1}{2} g_{ij} \dot{q}^i \dot{q}^j + V) + \left(\pdv{g_{ij}}{t} \dot{q}^i + \pdv{A_i}{t}\right)\dd q^i\right)\wedge  \dd t
\end{align}

Hence, it is immediate to see that

\begin{proposition} Let $Q_t = \pi^{-1}(t)$ denote the fiber of $\pi \colon \mathbf{Q} \longrightarrow \mathbb{R}$ and let $g_t$ denote the restriction of $\mathbf{g}$ to $Q_t$. Then, the Lagrangian $L$ is regular if and only if each metric $g_t$ is non-degenerate. 
\end{proposition}

Now, let us study the consistency conditions when $g_t$ is not definite positive (so that $(\omega_L, \dd t)$ is no longer a cosymplectic manifold). Here, letting $X$ denote a SODE field 
\[
X = \pdv{t} + \dot{q}^i\pdv{q^i} + X^i \pdv{\dot{q}^i}
\]
and imposing $i_{X} \omega_L = 0$ we get the equations 
\begin{align}
    -g_{ij} X^j =& \frac{1}{2}\left(  \pdv{g_{ij}}{q^k} + \pdv{g_{kj}}{q^i} - \pdv{g_{ki}}{q^j}\right) \dot{q}^j \dot{q}^k +\left( \pdv{A_i}{q^j} - \pdv{A_j}{q^i} + \pdv{g_{ij}}{t}\right) \dot{q}^j \\
    & + \pdv{V}{q^i} + \pdv{A_i}{t}\,.
\end{align}
By applying the constraint algorithm, if the Lagrangian is degenerate (hence the metric), we may contract on both sides by a vector $W = W^i \pdv{q^i}$ taking values in the characteristic distribution 
\[
C_t  = \{w \in \T Q_t : i_w g_t = 0\}
\]
to get the following consistency conditions:
\begin{align}
    0 =& \frac{w^i}{2}\left( \pdv{g_{ij}}{q^k} + \pdv{g_{kj}}{q^i} - \pdv{g_{ki}}{q^j}\right) \dot{q}^j \dot{q}^k + w^i\left( \pdv{A_i}{q^j} - \pdv{A_j}{q^i} + \pdv{g_{ij}}{t}\right) \dot{q}^j \\
    & + w^i \left( \pdv{V}{q^i} + \pdv{A_i}{t}\right)\,.
\end{align}
This, in general, imposes new conditions that we need to investigate further. Nevertheless, it gives us sufficient and necessary conditions for $L$ to be consistent. Indeed, if the above equations are trivially satisfied, by taking derivatives with respect to $\dot{q}^i$ two times we obtain the following conditions on the metric:
\begin{align}
\label{eq:consistency_1}
    w^i\left( \pdv{g_{ij}}{q^k} + \pdv{g_{kj}}{q^i} - \pdv{g_{ki}}{q^j}\right) & = 0\\
    \label{eq:consistency_2}
    w^i\left( \pdv{A_i}{q^j} - \pdv{A_j}{q^i} + \pdv{g_{ij}}{t}\right) & = 0\\
    \label{eq:consistency_3}
    w^i \left( \pdv{V}{q^i} + \pdv{A_i}{t}\right)& = 0\,,
\end{align}
for every $W  = w^i \pdv{q^i}$ taking values in the characteristic distribution. 
\begin{remark}[\textsc{Geometric conditions for $L$ to be consistent}] Suppose that we choose a trivialization $\mathbf{Q} = Q \times \mathbb{R}$. Then, the metric $\mathbf{g}$ is specified by a choice of 
\begin{itemize}
    \item A time dependent metric on $Q$, which we denote by $g$.
    \item A time dependent $1$-form on $Q$, which we denote by $A$.
    \item A time dependent potential $V$ on $Q$, which we denote by $V$. 
\end{itemize}
Denote by $C_t$ the characteristic distribution on $Q$, for every $t$. Then, the previous conditions read as 
\begin{align}
    \Lie_{W} g &= 0\,,\\
    i_W \left( \pdv{g}{t} - \dd A\right) &= 0\,,\\
    i_W \left(\dd V + \pdv{A}{t}\right) &=0\,,
\end{align}
for all $W$ taking values in the characteristic distribution.
\end{remark}

To deal with the issue of Lagrangian regularization, as discussed, we need to focus on consistent Lagrangians, so that henceforth we assume the previous conditions to hold. Then, we need to study when the characteristic distribution $\mathcal{V} = \ker \omega_L \cap \ker \dd t$ is the complete lift of a vertical integrable distribution on $\mathbf{Q}$. If it were the case, it is clear that the choice of the vertical distribution on $\mathbf{Q}$ would be the disjoint union of the characteristic distributions of each $(Q_t, g_t)$
\[
\undertilde{K} = \bigsqcup_{t \in \mathbb{R}} C_t\,,
\]
as we would need $\undertilde{K}^v$ to be the vertical elements in $\mathcal{V}$. 
Hence, we only need to find conditions on $C_t$ that, together with the above consistency conditions, guarantee that $\mathcal{V} = \undertilde{K}^{C}$.
\begin{proposition} If the Lagrangian defined by the metric $\mathbf{g}$ is consistent, we have that $\mathcal{V} = \undertilde{K}^{C}$ if and only if $C_t$ is integrable and the {$1$-form $A$} satisfies $i_{C_t} \dd A = 0$.
\end{proposition}
\begin{proof} If each distribution $C_t$ (and hence $\undertilde{K}$) is integrable, we can find adapted coordinates $(x^a, f^A, t)$ on $\mathbf{Q}$ in such a way that \[
C_t = \operatorname{span}\left \{\pdv{f^A}\right\}
\]
By definition, since $C_t$ is the characteristic distribution, the metric $\mathbf{g}$ reads as 
\[
\mathbf{g} = g_{ab} \dd x^a \dd x^b + 2 A \dd t - 2 V \dd t^2\,, \
\]
where $A$ denotes a $1$-form on each fiber.
Since $L$ is consistent (by applying \cref{eq:consistency_1}), we have that $g_{ab}$ only depends on $x^a$, so that its Poincaré--Cartan form reads as 
\[
\theta_L = \left( \frac{1}{2}g_{ab} \dot{x}^a \dot{x}^b + V\right) \dd t + g_{ab} \dot{x}^a \dd x^b + A\,.
\]
Hence, 
\[
\omega_L = \{\text{terms that only depend on $x^a$}\} - \dd A - \pdv{A}{t}\wedge \dd t - \dd V \wedge \dd t
\]
Contracting by an arbitrary element in 
\[
\undertilde{K}^C = \operatorname{span}\left \{\pdv{f^A}, \pdv{\dot{f}^A} \right \}\,
\]
and taking into account the compatibility condition of \cref{eq:consistency_3} we conclude the result.
\end{proof}

\begin{remark} A different, but equivalent way of stating that $C_t$ is an integrable distribution and that $\Lie_W g_t = 0$, if $W \in \Gamma(C_t)$, is to require the existence of an adapted, torsionless connection $\nabla_t$ to the metric $g_t$, for every $t$. If such a connection can be chosen for every $t$, it is not so complicated to show that it can be chosen so that $\nabla_t$ varies smoothly.
\end{remark}
Now, under the conditions of the above result, as we showed, the regularization procedure recovers a Lagrangian system. Here, we may employ the existence of an adapted connection to build the connection on the bundle 
\[
\undertilde{K}^\ast = \operatorname{span}\left \{\pdv{\mu_A}\right\} \longrightarrow \mathbf{Q}\,,
\]
but not so much with the product structure. In local, adapted coordinates to the characteristic distribution $C$, an arbitrary connection $\nabla$ (identified with its linear splitting of the tangent bundle $\T M$) takes the following local expression
\begin{align}
    \T M =& \operatorname{span}\left \{ \pdv{\dot{x}^a}, \pdv{\dot{f}^A}\right\} \oplus \mathcal{H}_\nabla\,, 
\end{align}
where 
\[
\mathcal{H}_\nabla =  \operatorname{span}\left \{ \pdv{x^a} + \Gamma_{ab}^c \dot{x}^b \pdv{\dot{x}^c} + \left(\Gamma^{A}_{ab}\dot{x}^b + \Gamma^{A}_{a B} \dot{f}^B \right) \pdv{\dot{f}^A}\,, \pdv{f^A} + \left(\Gamma_{AB}^C\dot{f}^B+ \Gamma_{Aa}^C \dot{x}^a\right) \pdv{\dot{f}^C} \right\}\,.
\]
Since $\nabla_t (g_t) = 0$, we also have $\nabla_t \left(\Gamma(C_t) \right) \subseteq \Gamma(C_t)$, so that, in particular, there is an induced linear connection on the bundle
$
\undertilde{K} \longrightarrow \mathbf{Q}
$. In particular, it induces a dual linear connection on $\undertilde{K}^\ast$, which we denote by $\nabla^\ast$. Identified as a splitting of the tangent bundle, it reads as follows:
\[
\T \undertilde{K}^\ast = \mathcal{V} \oplus\operatorname{span}\left \{ \pdv{x^a} - \Gamma_{aB}^C \mu_C \pdv{\mu_B}\,, \pdv{f^A} - \Gamma_{AB}^C \mu_A \pdv{\mu_C}\right \}\,.
\]
Finally, given an almost product structure on $\mathbf{Q}$, say 
\[
\T \mathbf{Q} = \operatorname{span}\left \{ \pdv{f^A}\right \} \oplus \operatorname{span}\left \{\pdv{t} + Q^A \pdv{f^A}\,, \pdv{x^a} + P^A_a \pdv{f^A}\right \}\,,
\]
we obtain the regularized Lagrangian 
\[
\widetilde L = \frac{1}{2}g_{ab} \dot{x}^a \dot{x}^b + A_a \dot{x}^a + A_A \dot{f}^A - V + \left( \dot{f}^A - Q^A - P^A_a \dot{x}^a\right) \left( \dot{\mu}_A + \dot{x}^a \mu_B \Gamma_{aA}^B + \dot{f}^B\mu_C \Gamma_{BA}^C \right) 
\]

\section{Conclusions and further work}

In this paper, we have developed a method for regularizing singular time-dependent Lagrangian systems. To do so, we first analysed in detail the method developed by \textit{A. Ibort} and \textit{J. Marín-Solano} \cite{ibort-marin} for singular time-independent Lagrangian systems, improving some of their results, namely explicitly constructing a global regular Lagrangian with the auxiliary help of a connection rather than a Riemannian metric in \cref{thm:Lagrangian_coisotropic_embedding_symplectic} and proving that the embedding is unique to first order in \cref{thm:uniqueness_sympelctic_first_order}. Since the basis of the construction was the coisotropic embedding theorem in pre-symplectic manifolds, our method has been based on the coisotropic embedding for pre-cosymplectic manifolds, which yields the analogue of the previous results, namely \cref{thm:Lagrangian_pre-cosymplectic_regularization} and \cref{thm:precosymplectic_Lag_uniqueness}, respectively. The other key ingredients have been the use of almost-product structures adapted to the singularity of the Lagrangian, which facilitates the use of the constraints algorithm for singular Lagrangians; and the use of Tulczyjew triples adapted to a foliation, which allows for the regularization to inherit a natural tangent (or jet) structure.

This paper opens some new and interesting research lines that we aim to discuss in coming papers:

\begin{itemize}

\item Extend the regularization problem for singular contact systems; a constraint algorithm has been recently developed in \cite{lainz}.

\item Extend the geometric approach to the inverse problem for implicit equations of \cite{Schiavone-InverseProblemElectrodynamics-2024, Schiavone-InverseProblemImplicit-2024} to time-dependent implicit differential equations.

\item A main issue is the extension of the regularization problem for singular classical field theories, due to the complexity of multisymplectic geometry. In such a case, we should develop a covariant theory on the space of solutions \cite{aitor, Ciaglia-DiCosmo-Ibort-Marmo-Schiavone-Zampini-Peierls1-2024, Ciaglia-DiCosmo-Ibort-Marmo-Schiavone-Zampini-Symmetry-2022}.

\item Other interesting research purposes are the following ones: the case when we have symmetries for the Lagrangian function, or the regularization of the Hamilton--Jacobi equation \cite{vaquero}, and furthermore, the discretization of the original singular Lagrangian and its relation with the regularized one.

\end{itemize}

\section*{Acknowledgements}

We acknowledge financial support of the 
{\sl Ministerio de Ciencia, Innovaci\'on y Universidades} (Spain), grant PID2022-125515NB-C21; we also acknowledge financial support from the Severo Ochoa Programme for Centers of Excellence in R\&D  and Grant CEX2023-001347-S funded by
MICIU/AEI/10.13039/501100011033.
Pablo Soto 
also acknowledges a JAE-Intro scholarship for undergraduate students. Rubén Izquierdo-López
wishes to thank the Spanish Ministry of Science, Innovation and Universities
for the contract FPU/02636. Finally, we would like to thank the anonymous referee for drawing our attention to several relevant references in the literature.

\bibliographystyle{abbrv}
{\small
\bibliography{sample.bib}
}

\end{document}